\documentclass[reqno,11pt]{amsart}
\usepackage[colorlinks=true, linkcolor=blue, citecolor=blue]{hyperref}
    
\usepackage{amssymb}
\usepackage{amsmath, graphicx, rotating}
\usepackage{color}
\usepackage{soul}
\usepackage[dvipsnames]{xcolor}
               
\usepackage{ifthen}    
\usepackage{xkeyval}
\usepackage{todonotes}     
\usepackage[T1]{fontenc}
\usepackage{lmodern}
\usepackage[english]{babel}

\usepackage{ upgreek }
\usepackage{stmaryrd}
\SetSymbolFont{stmry}{bold}{U}{stmry}{m}{n}
\usepackage{amsthm}
\usepackage{float}

\usepackage{ bbm }
\usepackage{ stmaryrd }
\usepackage{ mathrsfs }
\usepackage{ frcursive }
\usepackage{ comment }

\usepackage{pgf, tikz}
\usetikzlibrary{shapes}
\usepackage{varioref}
\usepackage{enumitem}

\usepackage{mathtools}  

\usepackage{dsfont}

\definecolor{rouge}{rgb}{0.7,0.00,0.00}
\definecolor{vert}{rgb}{0.00,0.5,0.00}
\definecolor{bleu}{rgb}{0.00,0.00,0.8}

\usepackage[textheight=7.61in,textwidth=4.98in]{geometry} 
\newtheorem{theorem}{Theorem}[section]
\newtheorem*{theorem*}{Theorem}
\newtheorem{lemma}[theorem]{Lemma}

\newtheorem{proposition}[theorem]{Proposition}

\labelformat{hypothesis}{\textbf{M\kern-0.1mm#1}}

\newtheorem{condition}{Condition}
\newtheorem{conditionA}{A\kern-0.1mm}
\labelformat{conditionA}{\textbf{A\kern-0.1mm#1}}

\renewcommand\dots{\hbox to 1em{.\hss.\hss.}}

\theoremstyle{definition}

\def \eref#1{\hbox{(\ref{#1})}}

\numberwithin{equation}{section}

\def\bb#1{\mathbb{#1}}

\def\scr#1{\mathscr{#1}}

\def\geq{\geqslant}
\def\leq{\leqslant}

\newcommand{\ee}{\varepsilon}

\DeclareMathOperator{\supp}{supp}

\DeclarePairedDelimiter\floor{\lfloor}{\rfloor}

\begin{document}

\title[Limit theorems for products of random matrices]
{Limit theorems for the coefficients of random walks on the general linear group}


\author{Hui Xiao}
\author{Ion Grama}
\author{Quansheng Liu}

\curraddr[Xiao, H.]{ Universit\'{e} de Bretagne-Sud, LMBA UMR CNRS 6205, Vannes, France}
\email{hui.xiao@univ-ubs.fr}
\curraddr[Grama, I.]{ Universit\'{e} de Bretagne-Sud, LMBA UMR CNRS 6205, Vannes, France}
\email{ion.grama@univ-ubs.fr}
\curraddr[Liu, Q.]{ Universit\'{e} de Bretagne-Sud, LMBA UMR CNRS 6205, Vannes, France}
\email{quansheng.liu@univ-ubs.fr}


\begin{abstract}
Let $(g_n)_{n\geq 1}$ be a sequence of independent and identically distributed 
random elements with law $\mu$ on the general linear group $\textup{GL}(V)$, where $V=\bb R^d$. 
Consider the random walk $G_n : = g_n \ldots g_1$, $n \geq 1$, and the coefficients $\langle f, G_n v \rangle$, 
where $v \in V$ and $f \in V^*$. 
Under suitable moment assumptions on $\mu$, 
we prove the strong and  weak laws of large numbers and the central limit theorem for $\langle f, G_n v \rangle$, 
which improve the previous results established under the exponential moment condition on $\mu$.
We further demonstrate the Berry-Esseen bound, the Edgeworth expansion,
the Cram\'{e}r type moderate deviation expansion
and the local limit theorem with moderate deviations for $\langle f, G_n v \rangle$ under the exponential moment condition. 
Under a subexponential moment condition on $\mu$, 
we also show a Berry-Esseen type bound
and the moderate deviation principle for $\langle f, G_n v \rangle$. 
Our approach is based on various versions of the H\"older regularity of 
the invariant measure of the Markov chain $G_n \!\cdot \! x = \bb R G_n v$ on the projective space
of $V$ 
with the starting point $x = \bb R v$. 
\end{abstract}

\date{\today}

\maketitle

\tableofcontents


\section{Introduction} 
\subsection{Background and main objectives}
Since the pioneering work of Furstenberg and Kesten \cite{FK60}, 
the theory of random walks on linear groups (also called products of random matrices)
has attracted a great deal of attention of many mathematicians for several decades, 
see for instance the influential work of Le Page \cite{LeP82}, 
Guivarc'h and Raugi \cite{GR85},  Bougerol and Lacroix \cite{BL85}, Goldsheid and Margulis \cite{GM89},
Benoist and Quint \cite{BQ16b}, and the references therein. 
This theory has important applications in a number of research areas such as 
spectral theory \cite{BL85, CL90, BG12},  
geometric measure theory \cite{PT08, HS17, GK21}, 
statistical physics \cite{CPV93},   
homogeneous dynamics \cite{BFLM11, BQ13},
stochastic recursions and smoothing transforms \cite{Kes73, GL16, Men16},
and branching processes in random environment \cite{LPP18, GLP20a}.
Of particular interest is 
the study of asymptotic properties of the random walk
\begin{align*}
G_n : = g_n \ldots g_1,  \quad  n \geq 1, 
\end{align*}
where $(g_n)_{n \geq 1}$ is a sequence of independent and identically distributed
(i.i.d.) random elements with law $\mu$ on the general linear group $\textup{GL}(V)$ with $V = \bb R^d$. 
One natural and important way to describe the random walk $(G_n)_{n\geq 1}$ is to  
investigate the growth rate  
of the coefficients $\langle f, G_n v \rangle$, 
where $v \in V$ and $f \in V^*$, 
and $\langle \cdot, \cdot \rangle$ is the duality bracket: $\langle f, v \rangle = f(v)$. 
Bellman \cite{Bel54} conjectured that the classical central limit theorem 
should hold true for $\langle f, G_n v \rangle$ 
in the case when the matrices $(g_n)$ are positive. 
This conjecture was proved by Furstenberg and Kesten \cite{FK60},
who established the strong law of large numbers and central limit theorem 
under the condition that the matrices $g_n$ are strictly positive and that all the coefficients of $g_n$ are comparable. 
This condition was then relaxed 
by Kingman \cite{Kin73}, 
Cohn, Nerman and Peligrad \cite{CNP93} and Hennion \cite{Hen97}.

As noticed  by Furstenberg \cite{Fur63}, 
the analysis developed in \cite{FK60} for positive matrices breaks down for invertible matrices.  
It turns out that the situation of invertible matrices is much more complicated and delicate. 
Guivarc'h and Raugi \cite{GR85} established the strong law of large numbers for the coefficients of products 
of invertible matrices
under an exponential moment condition: 
for any $v \in V \setminus \{0\}$ and $f \in V^* \setminus \{0\}$,  a.s.
\begin{align}\label{Ch7_SLLN_Entry0a}
\lim_{n\to\infty} \frac{1}{n} \log | \langle f, G_n v \rangle | = \lambda_1,  
\end{align}
where $\lambda_1 \in \bb R$ is a constant (independent of $f$ and $v$) called the first Lyapunov exponent of $\mu$.  
It is worth mentioning that the result \eqref{Ch7_SLLN_Entry0a} does not follow from
the classical subadditive ergodic theorem of Kingman \cite{Kin73}, nor from the recent version by Gou\"ezel and Karlsson \cite{GK20}.  
The central limit theorem for the coefficients has also been established in \cite{GR85}  under the exponential moment condition: 
for any $t \in \bb R$, 
\begin{align}\label{Ch7_CLT_Entry0a}
\lim_{n \to \infty} \bb{P} \left( \frac{ \log | \langle f, G_n v \rangle | 
 - n \lambda_1}{ \sigma \sqrt{n} } \leq t  \right)  =  \Phi(t), 
\end{align}
where $\Phi$ is the standard normal distribution function on $\bb R$
and $\sigma^2 > 0$ is the asymptotic variance of $ \frac{1}{\sqrt n} \log | \langle f, G_n v \rangle |.$  
Recently, Benoist and Quint \cite{BQ16b} have extended \eqref{Ch7_SLLN_Entry0a} and \eqref{Ch7_CLT_Entry0a}
to the framework of the general linear group $\textup{GL}(V)$ with $V = \bb K^d$,
where $\bb{K}$ is a local field. 
Moreover, they also established 
the law of iterated logarithm and the large deviations bounds.


The first objective of this paper is to prove the weak law of large numbers  
under the first moment condition: we prove that, if 
$\int_{ \textup{GL}(V) } \log N(g)  \mu(dg) < \infty$ with 
$N(g) = \max \{ \|g\|, \| g^{-1} \| \}$, then \eqref{Ch7_SLLN_Entry0a}
holds in probability.
Moreover, under the second moment condition that 
$\int_{ \textup{GL}(V) } \log^{2} N(g)  \mu(dg) < \infty$, 
we prove the strong law of large numbers, namely that \eqref{Ch7_SLLN_Entry0a} holds a.s. 
Under the same second moment condition we also prove
the central limit theorem: we show that the Gaussian approximation \eqref{Ch7_CLT_Entry0a} holds. 

Our second objective is to investigate the rate of convergence in the central limit theorem 
\eqref{Ch7_CLT_Entry0a}.
To this end, we first establish a Berry-Esseen bound 
under the exponential moment condition: we prove that if 
 $\int_{ \textup{GL}(V) }  N(g)^{\ee} \mu(dg) < \infty$ for some $\ee > 0$, 
then there exists a constant $c > 0$ such that for all $n \geq 1$, $t \in \bb R$, 
$v \in V$ and $f \in V^*$  with $\|v\| = \|f\| =1$, 
 and any $\gamma$-H\"older continuous function $\varphi$ on $\bb P(V)$, 
\begin{align}\label{BerryEsseen_Coeffaa-Intro}
\left|  \bb{E} \left[ \varphi(G_n \!\cdot\! x) 
\mathds{1}_{ \big\{ \frac{\log |\langle f, G_n v \rangle| - n \lambda_1 }{ \sigma \sqrt{n} } \leq t  \big\}  }   \right]
- \nu (\varphi) \Phi(t)  \right|  \leq  \frac{c}{\sqrt{n}} \|\varphi\|_{\gamma}, 
\end{align}
where $\|\varphi\|_{\gamma}$ is the $\gamma$-H\"older norm of the function $\varphi$. 
Our result \eqref{BerryEsseen_Coeffaa-Intro} improves a very recent Berry-Esseen bound with $\varphi = 1$ in \cite{DKW21},
where a different approach is applied. 
In fact we will establish a much stronger result, that is, the first-order Edgeworth expansion 
(cf.\ Theorem \ref{Thm-Edge-Expan-Coeff001}).
Moreover, under a subexponential moment condition, we prove that a weaker version of the Berry-Esseen bound 
holds true. Namely, the bound \eqref{BerryEsseen_Coeffaa-Intro} holds with the rate $\frac{c \log^{\beta} n }{\sqrt{n}}$
(for some $\beta > 0$)
instead of $\frac{c}{\sqrt{n}}$.


We next establish the  moderate deviation principle under the subexponential moment condition
that $\int_{ \textup{GL}(V) } e^{\log^{\alpha} N(g)} \,  \mu(dg)  < \infty$
for some constant $\alpha \in (0,1)$:
for any Borel set $B \subseteq \bb{R}$ and any sequence $(b_n)_{n\geq 1}$ of positive real numbers satisfying
$\frac{ b_n }{ n } \to 0$ and $b_n = o(n^{\frac{1}{2 - \alpha}})$ as $n \to \infty$, 
we have
\begin{align}\label{MDP-Coeff00a}
- \inf_{t \in B^{\circ}} \frac{t^2}{2 \sigma^2} & \leq
\liminf_{n\to \infty} \frac{n}{b_n^{2}}
\log \bb{P} \left(  \frac{ \log | \langle f, G_n v \rangle | - n \lambda_1 }{b_n} \in B  \right)  \nonumber\\
 & \leq   \limsup_{n\to \infty}\frac{n}{b_n^{2}}
\log  \bb{P}  \left( \frac{ \log | \langle f, G_n v \rangle | - n \lambda_1 }{b_n} \in B  \right) 
\leq - \inf_{t \in \bar{B}} \frac{t^2}{2\sigma^2},
\end{align}
where $B^{\circ}$ and $\bar{B}$ are respectively the interior and the closure of $B$.
%

We then reinforce the moderate deviation principle \eqref{MDP-Coeff00a} to the Cram\'er type moderate deviation expansion 
under the exponential moment condition: we prove that if  
$\int_{ \textup{GL}(V) }  N(g)^{\ee} \mu(dg) < \infty$ for some $\ee>0$, then  
uniformly in $t \in [0, o(\sqrt{n} )]$, $v \in V$ and $f \in V^*$ with $\|v\| = \|f\| =1$, as $n \to \infty$,  
\begin{align}
\frac{\bb{P} \Big( \frac{\log | \langle f, G_n v \rangle | 
- n\lambda_1}{ \sigma \sqrt{n} } \geq t \Big)} {1-\Phi(t)}
& =  e^{ \frac{t^3}{\sqrt{n}} \zeta ( \frac{t}{\sqrt{n}} ) } \big[ 1 +  o(1) \big], \label{Ch7_Cramer_Entry_0a}
\end{align}
where $\zeta$ is the Cram\'{e}r series (cf.\ \eqref{Ch7Def-CramSeri}). A similar expansion for the lower tail is also obtained.
More generally, we prove the Cram\'{e}r type moderate deviation expansion for the couple 
$(G_n \!\cdot\! x, \log |\langle f, G_n v \rangle|)$ with a target function $\varphi$ on the Markov chain $(G_n \!\cdot\! x)$
on the projective space $\bb P(V)$; see Theorem \ref{Thm-Cram-Entry_bb}. 
We mention that Bahadur-Rao type and Petrov type large deviation asymptotics 
for the coefficients $\langle f, G_n v \rangle$ have been recently established in \cite{XGL19d},
which give precise estimation of the rate of convergence 
in the weak law of large numbers \eqref{Ch7_SLLN_Entry0a}.

Our third objective is to establish the  
local limit theorem with moderate deviations for the coefficients $\langle f, G_n v \rangle$: 
for any real numbers $a_1 < a_2$, 
we have, as $n \to \infty$, uniformly in $|t| = o(\sqrt{n})$, 
\begin{align}\label{LLT-MD-Intro}
&\mathbb{P} \Big( \log| \langle f, G_n v \rangle | - n\lambda_1 \in [a_1, a_2] + \sqrt{n}\sigma t \Big) \notag\\
&= \frac{a_2 - a_1}{ \sigma \sqrt{2 \pi n} } 
 e^{ - \frac{t^2}{2} + \frac{t^3}{\sqrt{n}}\zeta(\frac{t}{\sqrt{n}} ) } [1 + o(1)]. 
\end{align}
See Theorem \ref{ThmLocal02} for a more general statement 
where a local limit theorem with moderate deviations for the couple $(G_n \!\cdot\! x, \log |\langle f, G_n v \rangle|)$
with target functions is given. 
We would like to mention that using the approach developed in this paper,
it is possible to establish some new and interesting limit theorems
for the Gromov product of random walks on hyperbolic groups;
we refer to Gou\"{e}zel \cite{Gou09, Gou14} on this topic. 

Finally we would like to mention that all the results of the paper remain valid when $V$ is $\bb C^d$ or $\bb K^d$, where 
$\bb K$ is any local field.

\subsection{Proof strategy} 
Our starting point is the following decomposition
which relates the coefficients $\langle f, G_n v \rangle$ to the cocycles $\sigma (G_n, x)$: 
for any $x = \bb R v \in \bb P(V)$ and $y = \bb R f \in \bb P(V^*)$ with $\|f\|=1$, 
\begin{align}\label{Ch7_Intro_Decom0a}
\log |\langle f, G_n v \rangle| = \sigma (G_n, x) +  \log \delta(y, G_n \!\cdot\! x),  
\end{align}
where $\delta(y, x) = \frac{|\langle f, v \rangle|}{\|f\| \|v\|}$. 
The proof of \eqref{Ch7_SLLN_Entry0a} 
given in \cite{GR85, BQ16b}
is based on the 
exponential H\"older regularity of the invariant measure $\nu$ of the Markov chain $(G_n\kern-0.05em \cdot\kern-0.05em x)$
on the projective space $\bb P(V)$, which requires the exponential moment condition. 
We refer to Bougerol and Lacroix \cite{BL85} for a comprehensive account of the 
proof strategy developed in \cite{GR85}.  
In order to relax the exponential moment condition,
we will establish a weaker form of the H\"older regularity of $\nu$:
we prove that for any $\ee >0$,
\begin{align}\label{Regularity-Firstmom}
\lim_{n \to \infty} \bb P \left( \delta (y, G_n \!\cdot\! x) \leq e^{- \ee n} \right) = 0,  
\end{align}
provided that the first moment condition holds. See Proposition \ref{Prop-rates-delta01}.  
The weak law of large numbers for the coefficients 
follows from this and the law of large numbers of Furstenberg \cite{Fur63}
 for the norm cocycle $\sigma (G_n, x)$. 

To prove the strong law of large numbers, 
we show the following stronger regularity of $\nu$: if the second moment condition holds,
then for any $\ee >0$ there exists a sequence of  positive numbers $(a_k)$ such that $\sum_{k=1}^{\infty} a_k < \infty$ 
and that for all $n \geq k \geq 1$, 
\begin{align}\label{logRegularity-Secondmom}
\bb{P} \left( \delta(y, G_n \!\cdot\! x) \leq e^{- \ee k} \right) \leq  a_k. 
\end{align}
See Lemma \ref{Lem-ScaNorm_second}. 
This permits us to conclude the strong law of large numbers for the coefficients, using that of Furstenberg \cite{Fur63} for the norm cocycle. 

Using again \eqref{logRegularity-Secondmom}, 
together with the central limit theorem for the norm cocycle $\sigma(G_n, x)$ due to Benoist and Quint \cite{BQ16},
allows us to prove the Gaussian approximation \eqref{Ch7_CLT_Entry0a} under the optimal second moment condition.

For the proof of the Edgeworth expansion and the Berry-Esseen bound \eqref{BerryEsseen_Coeffaa-Intro}, 
we first use a partition $(\chi_{n,k}^y)_{k \geq 1}$ of the unity  to discretise the component $\log \delta(y, G_n \!\cdot\! x)$ in \eqref{Ch7_Intro_Decom0a}.
This allows to reduce the study of the coefficients to that of the norm cocycle $\sigma(G_n, x)$ 
jointly with the target function $\chi_{n,k}^y(G_n \!\cdot\! x)$. 
Then our strategy is to
make use of the Edgeworth expansion for the couple $(G_n \!\cdot\! x, \sigma(G_n, x))$
with target functions on the Markov chain $(G_n \!\cdot\! x)$.  
Finally, the result \eqref{BerryEsseen_Coeffaa-Intro} is obtained by patching up these expansions 
by means of the exponential H\"older regularity of the invariant measure $\nu$. 
For the proof of the Berry-Esseen type bound (Theorem \ref{Thm_BerryEsseen}) under a subexponential moment condition, 
we establish a subexponential regularity of the invariant measure $\nu$ (Theorem \ref{Thm-Regularity-Subex})
and make use of the previous result in \cite{CDMP21} 
on Berry-Esseen bounds for the norm cocycle $\sigma(G_n, x)$. 

To prove the moderate deviation principle \eqref{MDP-Coeff00a}, we use 
again the subexponential regularity of the invariant measure $\nu$ together with 
the moderate deviation principle for the norm cocycle $\sigma(G_n, x)$ established by in \cite{CDM17}. 

To establish the Cram\'er type moderate deviation expansion \eqref{Ch7_Cramer_Entry_0a}, 
our approach is different from the standard one which is based on 
performing a change of measure and proving 
a Berry-Esseen bound under the changed measure; see for example Cram\'{e}r \cite{Cra38} and Petrov \cite{Pet75}. 
However, even with a Berry-Esseen bound under the changed measure at hands,
we do not know how to obtain \eqref{Ch7_Cramer_Entry_0a} 
using this strategy. 
Our approach is to use again a partition  of the unity
and for each piece to pass to the Fourier transforms under the changed measure,
and then to establish exact asymptotic expansions.  
We then patch up these expansions by using the exponential H\"older regularity of the invariant measure $\nu$. 

The proof of the local limit theorem with moderate deviations \eqref{LLT-MD-Intro}
follows the same lines as that of \eqref{Ch7_Cramer_Entry_0a}, 
together with the uniform version of the exponential H\"older regularity of the invariant measure of the Markov chain $(G_n \!\cdot\! x)$
under the changed measure.

\section{Main results}

\subsection{Notation and conditions}
For any integer $d \geq 1$, denote by $V = \bb R^d$ the $d$-dimensional Euclidean space.
We fix a basis $e_1, \ldots, e_d$ of $V$ and the associated norm on $V$ is defined by $\|v\|^2 = \sum_{i=1}^d |v_i|^2$ for $v = \sum_{i=1}^d v_i e_i \in V$. 
Let $V^*$ be the dual vector space of $V$ and its dual basis is denoted by $e_1^*, \ldots, e_d^*$ 
so that $e_i^*(e_j)= 1$ if $i=j$ and  $e_i^*(e_j)= 0$ if $i\neq j$. 
For any integer $2 \leq p \leq d$, let $\wedge^p V$ be the $p$-th exterior product of $V$
and we use the same symbol $\| \cdot \|$ for the norms induced on $\wedge^p V$ and $V^*$.  
We equip $\bb P (V)$ with the angular distance 
\begin{align}\label{Angular-distance}
d(x, x') = \frac{\| v \wedge v' \|}{ \|v\| \|v'\| }  \quad \mbox{for} \  x= \bb R v \in \bb P(V), \  x' = \bb R v' \in \bb P(V). 
\end{align}
We use the symbol $\langle \cdot, \cdot \rangle$ to denote the dual bracket defined by $\langle f, v \rangle = f(v)$ for any $f \in V^*$ and $v \in V$. 
Set 
\begin{align*}
\delta(x,y) = \frac{| \langle f, v \rangle |}{\|f\| \|v\| }  \quad \mbox{for} \  x= \bb R v \in \bb P(V),  \  y = \bb R f \in \bb P(V^*).
\end{align*}
Denote by $\mathscr{C}(\bb{P}(V) )$ the space of complex-valued continuous functions on $\bb{P}(V)$, 
equipped with the norm $\|\varphi\|_{\infty}: =  \sup_{x\in \bb{P}(V) } |\varphi(x)|$ for $\varphi \in \mathscr{C}(\bb{P}(V) )$.
Let $\gamma>0$ be a fixed small enough constant and set
\begin{align*}
\|\varphi\|_{\gamma}: =  \|\varphi\|_{\infty} + 
  \sup_{x, x' \in \bb{P}(V): x \neq x'} \frac{|\varphi(x)-\varphi(x')|}{ d(x, x')^{\gamma} }.  
\end{align*}
Consider the Banach space 
\begin{align*}
\mathscr{B}_{\gamma}:= \left\{ \varphi\in \mathscr{C}(\bb P(V)): \|\varphi\|_{\gamma}< \infty \right\},   
\end{align*}
which consists of complex-valued $\gamma$-H\"older continuous functions on $\bb P(V)$. 
Denote by $\mathscr{L(B_{\gamma},B_{\gamma})}$ 
the set of all bounded linear operators from $\mathscr{B}_{\gamma}$ to $\mathscr{B}_{\gamma}$
equipped with the operator norm
$\left\| \cdot \right\|_{\mathscr{B}_{\gamma} \to \mathscr{B}_{\gamma}}$. 
The topological dual of $\mathscr B_\gamma$ endowed with the induced norm 
is denoted by $\mathscr B'_\gamma$.
Let $\mathscr B_\gamma^*$ be the Banach space of $\gamma$-H\"older continuous functions
on $\mathbb P(V^*)$ endowed with the norm
\begin{align*} 
\| \varphi \|_{\mathscr B_\gamma^*} = 
 \sup_{y\in \mathbb P(V^*) } |\varphi(x)| 
  + \sup_{  y,y'\in \bb P(V^*): \, y \neq y' } \frac{ |\varphi(y)-\varphi(y')| }{ d(y,y')^{\gamma} }, 
\end{align*}
where $d(y,y') = \frac{\| f \wedge f' \|}{ \|f\| \|f'\| }$ for $y= \bb R f \in \bb P(V^*)$ and $y' = \bb R f' \in \bb P(V^*)$.


Let $\textup{GL}(V)$ be the general linear group of the vector space $V$.
The action of $g \in \textup{GL}(V)$ on a vector $v \in V$ is denoted by $gv$, 
and the action of $g \in \textup{GL}(V)$ on a projective line $x = \bb R v \in \bb P(V)$ is denoted by $g \cdot x = \bb R gv$. 
For any $g \in \textup{GL}(V)$, let $\| g \| = \sup_{v \in V \setminus \{0\} } \frac{\| g v \|}{\|v\|}$
and denote $N(g) = \max \{ \|g\|, \| g^{-1} \| \}$. 
Let $\mu$ be a Borel probability measure on $\textup{GL}(V)$. 
We shall use the following exponential moment condition.




\begin{conditionA}[Exponential moment condition]\label{Ch7Condi-Moment} 
There exists a constant $\ee >0$ such that $\int_{ \textup{GL}(V) } N(g)^{\ee} \mu(dg) < \infty$. 
\end{conditionA}


Let $\Gamma_{\mu}$ be the smallest closed subsemigroup
generated by the support of the measure $\mu$. 
An endomorphism $g$ of $V$ is said to be proximal 
if it has an eigenvalue $\lambda$ with multiplicity one and all other eigenvalues of $g$ have modulus strcitly less than $|\lambda|$.
We shall need the following strong irreducibility and proximality condition. 

\begin{conditionA}\label{Ch7Condi-IP}
{\rm (i)(Strong irreducibility)} 
No finite union of proper subspaces of $V$ is $\Gamma_{\mu}$-invariant.

{\rm (ii)(Proximality)}
$\Gamma_{\mu}$ contains a proximal endomorphism. 
\end{conditionA}

%
%
%

Define the norm cocycle $\sigma: \textup{GL}(V) \times \bb P(V) \to \bb R$ as follows: 
\begin{align*}
\sigma (g, x) = \log \frac{\|gv\|}{\|v\|}, \quad   \mbox{for any} \  g \in \textup{GL}(V)  \  \mbox{and}  \   x = \bb R v \in \bb P(V). 
\end{align*}
By \cite[Proposition 3.15]{XGL19b}, under \ref{Ch7Condi-Moment} and \ref{Ch7Condi-IP},
 the following limit exists and is independent of $x \in \bb P(V)$:   
\begin{align}\label{Def-sigma}
\sigma^2: = \lim_{n \to \infty} \frac{1}{n} \bb E \left[ (\sigma (G_n, x) - n \lambda_1)^2 \right] \in (0, \infty).  
\end{align}
For any $s \in (-s_0, s_0)$ with $s_0 >0$ small enough, 
we define the transfer operator $P_s$ as follows: for any bounded measurable function $\varphi$ on $\bb P(V)$, 
\begin{align}\label{Def_Ps001}
P_s \varphi(x) = \int_{ \textup{GL}(V) } e^{s \sigma(g, x)} \varphi(g \!\cdot\! x) \mu(dg),  
\quad  x \in \bb P(V). 
\end{align}
It will be shown in Lemma \ref{Ch7transfer operator} that 
there exists a constant $s_0 >0$ such that for any $s \in (-s_0, s_0)$,
the operator $P_s \in \mathscr{L(B_{\gamma},B_{\gamma})}$ 
has 
a unique dominant eigenvalue $\kappa(s)$ 
with $\kappa(0) = 1$ and the mapping $s \mapsto \kappa(s)$ being analytic.

We denote $\Lambda = \log \kappa$.
Set $\gamma_m = \Lambda^{(m)}(0)$ for any $m \geq 1$.
In particular, $\gamma_1 = \lambda_1$ and $\gamma_2 = \sigma^2$. 
Throughout the paper, we write $\zeta$ for the Cram\'{e}r series \cite{Pet75}: 
\begin{align}\label{Ch7Def-CramSeri}
\zeta(t)=\frac{\gamma_3}{6\gamma_2^{3/2} } + \frac{\gamma_4\gamma_2-3\gamma_3^2}{24\gamma_2^3}t
+ \frac{\gamma_5\gamma_2^2-10\gamma_4\gamma_3\gamma_2 + 15\gamma_3^3}{120\gamma_2^{9/2}}t^2 + \cdots, 
\end{align}
which converges for  $|t|$ small enough.

Under \ref{Ch7Condi-Moment} and \ref{Ch7Condi-IP},
the Markov chain $(G_n \!\cdot\! x)_{n \geq 0}$ has a unique invariant probability measure 
$\nu$ on $\bb P(V)$ such that for any bounded measurable function $\varphi$ on $\bb P(V)$,
\begin{align} \label{Ch7mu station meas}
\int_{\bb P(V)} \int_{\textup{GL}(V)} \varphi(g \!\cdot\! x) \mu(dg) \nu(dx) 
 = \int_{ \bb P(V) } \varphi(x) \nu(dx)
= : \nu(\varphi).
 \end{align}



\subsection{Law of large numbers and central limit theorem}
In this subsection we present the law of large numbers and the central limit theorem
for the coefficients $\langle f, G_n v \rangle$.

We first present the law of large numbers for $\langle f, G_n v \rangle$: namely
we state a weak law of large numbers under the existence of the first moment and a 
strong law of large numbers under the second moment assumption. 
Denote by $\lambda_1$ and $\lambda_2$ the first and second Lyapunov exponents of $\mu$:
\begin{align} \label{def-lambda2}
\lambda_1 : = \lim_{n \to \infty} \frac{1}{n} \bb E \log \|G_n\| , \quad 
\lambda_2 : = \lim_{n \to \infty} \frac{1}{n} \bb E \log \frac{\|\wedge^2 G_n\|}{\| G_n \|}. 
\end{align}
We say that $\Gamma_{\mu}$ is irreducible if no proper subspace of $V$ is $\Gamma_{\mu}$-invariant.

\begin{theorem}\label{Thm_LLN}
Assume that $\int_{ \textup{GL}(V) } \log N(g)  \mu(dg) < \infty.$
Assume also that $\lambda_1 > \lambda_2$ and that $\Gamma_{\mu}$ is irreducible. 
Then, we have, as $n \to \infty$, uniformly in $v \in V$ and $f \in V^*$ with $\|v\| = \|f\| =1$,   
\begin{align}\label{Weak_LLN_Coeff}
\frac{1}{n}  \log |\langle f, G_n v \rangle|  \to \lambda_1   \quad  \mbox{in probability and in  }  L^1. 
\end{align}
Moreover, if 
$\int_{ \textup{GL}(V) } \log^2 N(g) \,  \mu(dg) < \infty$
and  
condition \ref{Ch7Condi-IP} holds, then  
the convergence in probability in \eqref{Weak_LLN_Coeff} can be improved to the a.s.\ convergence.
\end{theorem}

Notice that, by a theorem of Guivarc'h \cite{Gui80}, 
condition \ref{Ch7Condi-IP} implies that $\lambda_1 > \lambda_2$ and that $\Gamma_{\mu}$ is irreducible. 
Under the exponential moment condition \ref{Ch7Condi-Moment} 
 together with the strong irreducibility  and proximality  condition \ref{Ch7Condi-IP}, 
Guivarc'h and Raugi \cite{GR85}  proved the a.s. convergence in \eqref{Weak_LLN_Coeff} 
 for the special linear group $\textrm{SL} (d, \mathbb R)  $ (the set of $d \times d$ matrices with determinant $1$);     
 Benoist and Quint \cite{BQ16b}  extended  it to the general linear group $\textup{GL}(V)$. 
Our Theorem \ref{Thm_LLN} states  the weak law of large numbers 
for $\textup{GL}(V)$ under the first moment condition, and the strong 
law  of large numbers under the second moment condition. 
The question remains open whether the strong law of large numbers 
holds true under the first moment condition. 

It is worth mentioning that Aoun and Sert \cite{AS19} proved the weak law of large numbers for the spectral radius of $G_n$
under the first moment condition, and the strong law of large numbers
under the second moment condition.


In the next theorem we formulate the central limit theorem for the coefficients $\langle f, G_n v \rangle$ under the second moment condition. 
Recall that 
 $$\Phi(t) = \frac{1}{\sqrt{2 \pi}} \int_{- \infty}^t e^{- u^2/2} du, \quad  t \in \bb R$$
 is the standard normal distribution function, and $\sigma^2$ is the asymptotic variation given by \eqref{Def-sigma}.

\begin{theorem}\label{Thm_CLT}
If $\int_{ \textup{GL}(V) } \log^2 N(g) \,  \mu(dg) < \infty$
and  condition \ref{Ch7Condi-IP} holds,
then for any $t \in \bb R$, uniformly in $v \in V$ and $f \in V^*$ with $\|v\| = \|f\| =1$, 
\begin{align}\label{CLT_Coeff}
\lim_{n \to \infty}
\bb{P} \left( \frac{\log |\langle f, G_n v \rangle| - n \lambda_1}{ \sigma \sqrt{n} } \leq t \right) 
= \Phi(t). 
\end{align}
\end{theorem}

Theorem \ref{Thm_CLT} improves the results of Guivarc'h and Raugi \cite{GR85} on $\textrm{SL} (d, \mathbb R)$ 
and those of Benoist and Quint \cite{BQ16b} on $\textup{GL}(V)$
by relaxing the exponential moment condition \ref{Ch7Condi-Moment}  
to the optimal second moment condition. 

We mention that the central limit theorem for the spectral radius of $G_n$ has been obtained by Aoun \cite{Aou21}. 

\subsection{Edgeworth expansion and Berry-Esseen bound}
In many applications it is of primary interest to give an estimation of the rate of convergence in the Gaussian approximation \eqref{CLT_Coeff}.  
In this direction we prove a first-order Edgeworth expansion.
For any $\varphi \in \mathscr{B}_{\gamma}$, define the functions
\begin{align} \label{drift-b001}
b_{\varphi}(x): = \lim_{n \to \infty}
   \mathbb{E} \big[ ( \sigma(G_n, x) - n \Lambda'(s) ) \varphi(G_n \!\cdot\! x) \big], 
\quad   x \in \bb P(V) 
\end{align}
and
\begin{align} \label{drift-d001}
d_{\varphi}(y) :=  \int_{\bb P(V)} \varphi(x) \log \delta(y,x)\nu(dx), \quad  y \in \bb P(V^*). 
\end{align}
It is shown in Lemmas \ref{Lem-Bs} and \ref{Lem-ds} that
both functions $b_{\varphi}$ and $d_{\varphi}$ are well-defined and $b_{\varphi} \in \scr B_{\gamma}$, 
$d_{\varphi} \in \scr B_{\gamma}^*$.
 
\begin{theorem}\label{Thm-Edge-Expan-Coeff001}
Assume \ref{Ch7Condi-Moment} and \ref{Ch7Condi-IP}. 
Then, for any $\ee > 0$, uniformly in 
 $t \in \bb R$, $x=\bb R v \in \bb P(V)$, $ y = \bb R f \in \bb P(V^*)$ with $\|v\| = \|f\| =1$, and
  $\varphi \in \mathscr{B}_{\gamma}$,  as $n\to \infty$,
\begin{align}\label{EdgeworthExpan}
&  \mathbb{E}
   \Big[  \varphi(G_n \!\cdot\! x) \mathds{1}_{ \big\{ \frac{\log |\langle f, G_n v \rangle| - n \lambda_1  }{\sigma \sqrt{n}} \leq t \big\} } \Big]
   \notag \\
 &  \qquad  = \nu(\varphi) \Big[  \Phi(t) + \frac{\Lambda'''(0)}{ 6 \sigma^3 \sqrt{n}} (1-t^2) \phi(t) \Big]
  - \frac{ b_{\varphi}(x) + d_{\varphi}(y) }{ \sigma \sqrt{n} } \phi(t)    \notag\\
& \qquad\quad  +  \nu(\varphi)  o \Big( \frac{ 1 }{\sqrt{n}} \Big)  +  \lVert \varphi \rVert_{\gamma} O \Big( \frac{ 1 }{n^{1 - \ee} } \Big).     
\end{align}
\end{theorem}

As a consequence of Theorem \ref{Thm-Edge-Expan-Coeff001} we get the following Berry-Esseen bound 
with the optimal convergence rate, under the exponential moment condition.

\begin{theorem}\label{Thm-BerryEsseen-optimal}
Under \ref{Ch7Condi-Moment} and \ref{Ch7Condi-IP}, 
there exist constants $\gamma >0$ and $c > 0$  such that for any 
$n \geq 1$, $t \in \bb R$, $x=\bb R v \in \bb P(V)$ and $f \in V^*$ with $\|v\| = \|f\| =1$, 
 and $\varphi \in \mathscr{B}_{\gamma}$,  
\begin{align}\label{BerryEsseen_Coeffbb}
\left|  \bb{E} \left[ \varphi(G_n \!\cdot\! x) 
\mathds{1}_{ \big\{ \frac{\log |\langle f, G_n v \rangle| - n \lambda_1 }{ \sigma \sqrt{n} } \leq t  \big\}  }   \right]
- \nu (\varphi) \Phi(t)  \right|  
\leq  \frac{c}{\sqrt{n}} \|\varphi\|_{\gamma}. 
\end{align}
\end{theorem}

Under the same conditions, our result \eqref{BerryEsseen_Coeffbb}
improves on two very recent results in \cite{DKW21} and \cite{CDMP21b}:
it is proved in \cite{DKW21} that \eqref{BerryEsseen_Coeffbb} holds for the particular case when $\varphi = 1$, 
and in \cite{CDMP21b} the authors also consider the case where $\varphi = 1$ but the convergence rate $\frac{1}{\sqrt{n}}$ in \eqref{BerryEsseen_Coeffbb} is replaced by $\frac{\log n}{\sqrt{n}}$.

It is an open question how to relax the exponential moment condition \ref{Ch7Condi-Moment}
 in the Edgeworth expansion and in the Berry-Esseen bound.
For positive matrices, the Edgeworth expansion \eqref{EdgeworthExpan}
and the Berry-Esseen bound \eqref{BerryEsseen_Coeffbb} have been recently obtained
using a different approach in a forthcoming paper \cite{XGL21b}
under optimal moment conditions.  
In the following theorem we get a Berry-Esseen type bound for invertible matrices under the sub-exponential moment condition,
with an extra $\log^{\frac{1}{\alpha}} n$ factor. 
\begin{theorem}\label{Thm_BerryEsseen}
Assume \ref{Ch7Condi-IP} and that there exists a constant $\alpha \in (0, 1)$ 
such that $\mu(\log N(g) > u) \leq \exp \{- u^{\alpha} a(u) \}$ for any $u >0$
and for some function $a(u) >0$ satisfying $a(u) \to \infty$ as $u \to \infty$. 
Then, there exists a constant $c > 0$ such that for any $n \geq 2$, $t \in \bb R$, 
 $v \in V$ and $f \in V^*$ with $\|v\| = \|f\| =1$, 
\begin{align}\label{BerryEsseen_Coeffaa}
\left|  \bb{P} \left( \frac{\log |\langle f, G_n v \rangle| - n \lambda_1 }{ \sigma \sqrt{n} } \leq t \right) 
- \Phi(t)  \right|  \leq  \frac{c \log^{\frac{1}{\alpha}} n }{\sqrt{n}}. 
\end{align}
\end{theorem}

Note that the condition $\mu(\log N(g) > u) \leq \exp \{- u^{\alpha} a(u) \}$ 
holds true
 if $\int_{ \textup{GL}(V) } e^{\log^{(\alpha + \ee)} N(g)} \,  \mu(dg)  < \infty$ for some $\ee >0$. 
 
Recently in \cite{CDMP21b}, 
under the polynomial moment condition of order $p \geq 3$, that is, $\int_{ \textup{GL}(V) } \log^p N(g) \,  \mu(dg) < \infty$,
a Berry-Esseen type bound is  obtained with the convergence rate $n^{-\frac{p-1}{2p}}$. 

The proof of Theorem \ref{Thm_BerryEsseen} is based on the Berry-Esseen type bound recently established in \cite{CDMP21}
and on the subexponential H\"older regularity of the invariant measure $\nu$ which is the subject of Section \ref{Sect-Regularity}.

\subsection{Subexponential H\"older regularity}\label{Sect-Regularity}
We shall establish the following subexponential H\"older regularity of the invariant measure $\nu$ 
under only subexponential moments condition. 
This turns out to be an important step for proving the Berry-Esseen type bound \eqref{BerryEsseen_Coeffaa}
as well as the moderate deviation principle for the coefficients $\langle f, G_n v \rangle$ (see Theorem \ref{Thm-MDP} below). 

\begin{theorem}\label{Thm-Regularity-Subex}
Assume \ref{Ch7Condi-IP} and that there exists a constant $\alpha \in (0, 1)$ 
such that $\mu(\log N(g) > u) \leq \exp \{- u^{\alpha} a(u) \}$ for any $u >0$
and for some function $a(u) >0$ satisfying $a(u) \to \infty$ as $u \to \infty$. 
Then, there exist constants $c >0$ and $k_0 \in \bb{N}$ such that for all $n \geq k \geq k_0$, 
$x \in \bb P(V)$ and $y \in \bb P(V^*)$, 
\begin{align}\label{thmRegularity01}
 \bb P \left( \delta (G_n \!\cdot\! x, y) \leq  e^{  - k } \right) 
    \leq  \exp \left( -c k^{\alpha} \right).   
\end{align}
Moreover, there exists a constant $\eta >0$ such that 
\begin{align}\label{Regu-Subexponen01}
\sup_{y \in \bb P(V^*)} 
\int_{\bb P(V)} \exp \left( \eta |\log \delta (x,y)|^{ \alpha} \right)  \nu(dx) < + \infty. 
\end{align}
In particular, there exist constants $c>0$ and $\eta >0$ such that for any $y \in \bb P(V^*)$ and $r >0$,
\begin{align}\label{Regu-Subexponen02}
\nu (B(y,r)) \leq c  \exp \left( \eta |\log r|^{ \alpha} \right),
\end{align}
where $B(y,r) = \{ x \in \bb P(V): \delta(y,x) \leq r \}$. 
\end{theorem}

Notice that if \ref{Ch7Condi-IP} holds and that subexponential moment condition in Theorem \ref{Thm-Regularity-Subex}
is strengthed to the exponential moment condition \ref{Ch7Condi-Moment}, 
then by a theorem due to Guivarc'h \cite{Gui90}, 
the invariant measure $\nu$ satisfies a stronger regularity property: 
there exists a constant $\eta >0$ such that
\begin{align}\label{HolderRegu-InvarMea}
\sup_{y \in \bb P(V^*)} 
\int_{\bb P(V)}  \frac{1}{ \delta(x,y)^{\eta} }  \nu(dx) < + \infty. 
\end{align}
In this case, we say that the invariant measure $\nu$ is exponentially H\"older regular. 

Under the $p$-th moment condition that $\int_{ \textup{GL}(V) }  \log^p N(g) \, \mu(dg) < \infty$ for some $p>1$, 
Benoist and Quint \cite[Proposition 4.5]{BQ16} have recently established the log-regularity of $\nu$:
under \ref{Ch7Condi-IP}, 
\begin{align}\label{LogRegu-InvarMea}
\sup_{y \in \bb P(V^*)} 
\int_{\bb P(V)}  |\log \delta(x,y) |^{p-1} \nu(dx) < + \infty. 
\end{align}
This result is one of the crucial points for establishing the central limit theorem for the norm cocycle $\sigma(G_n, x)$ 
under the optimal second moment condition $\int_{ \textup{GL}(V) }  \log^2 N(g) \, \mu(dg) < \infty$, see \cite{BQ16}. 


\subsection{Moderate deviation principles and expansions}
In this subsection we state a moderate deviation principle and Cram\'{e}r type moderate deviation expansions for 
the coefficients $\langle f, G_n v \rangle$.

We first present the moderate deviation principle under a subexponential moment condition, 
whose proof is based upon the regularity of the invariant measure $\nu$ shown in Theorem \ref{Thm-Regularity-Subex}. 

\begin{theorem} \label{Thm-MDP}
Assume \ref{Ch7Condi-IP} and that $\int_{ \textup{GL}(V) } e^{\log^{\alpha} N(g)} \,  \mu(dg)  < \infty$
for some constant $\alpha \in (0,1)$. 
Then, for any Borel set $B \subseteq \mathbb{R}$
and  any sequence $(b_n)_{n\geq 1}$ of positive numbers satisfying 
$\frac{b_n}{\sqrt{n}} \to \infty$ and $b_n = o(n^{\frac{1}{2 - \alpha}})$ as $n \to \infty$,
 we have, uniformly in $v \in V$ and $f \in V^*$ with $\|v\| = \|f\| =1$, 
\begin{align*}
 -\inf_{y\in B^{\circ}} \frac{y^2}{2\sigma^2} & \leq 
\liminf_{n\to \infty} \frac{n}{b_n^{2}}
\log  \bb{P} \left(   \frac{\log |\langle f, G_n v \rangle| - n\lambda_1 }{b_n} \in B   \right)   \nonumber\\
 &  \leq  \limsup_{n\to \infty}\frac{n}{b_n^{2}}
\log  \bb{P}  \left(  \frac{\log |\langle f, G_n v \rangle| - n\lambda_1 }{b_n} \in B   \right)
\leq - \inf_{y\in \bar{B}} \frac{y^2}{2\sigma^2}, 
\end{align*}
where $B^{\circ}$ and $\bar{B}$ are respectively the interior and the closure of $B$. 
\end{theorem}

When the sub-exponential moment condition is strengthened to the exponential moment condition \ref{Ch7Condi-Moment},
we are able to establish the following Cram\'{e}r type moderate deviation expansions for 
the coefficients $\langle f, G_n v \rangle$, and more generally, 
for the couple $(G_n \!\cdot\! x, \log |\langle f, G_n v \rangle|)$ with a target function $\varphi$ 
on the Markov chain $(G_n \!\cdot\! x)_{n \geq 0}$. 

\begin{theorem}\label{Thm-Cram-Entry_bb}
Assume \ref{Ch7Condi-Moment} and \ref{Ch7Condi-IP}. Then, 
we have, as $n \to \infty$,  
uniformly in $v \in V$ and $f \in V^*$ with $\|v\| = \|f\| =1$, 
and  $t \in [0, o(\sqrt{n} )]$,  
\begin{align}
\frac{\bb{P} \left( \frac{\log |\langle f, G_n v \rangle| 
- n \lambda_1 }{ \sigma \sqrt{n} } \geq t \right)} {1-\Phi(t)}
& =  e^{ \frac{t^3}{\sqrt{n}}\zeta ( \frac{t}{\sqrt{n}} ) }  \big[ 1 +  o(1) \big],   \label{MD-Expan01}\\
\frac{\bb{P} \left( \frac{\log |\langle f, G_n v \rangle|
 - n \lambda_1 }{ \sigma \sqrt{n} }\leq -t \right)}{ \Phi(-t) }
& =  e^{ - \frac{t^3}{\sqrt{n}}\zeta (-\frac{t}{\sqrt{n}} ) }  \big[ 1 +  o(1) \big].     \label{MD-Expan02}
\end{align}
More generally, for any $\varphi \in \mathscr{B}_{\gamma}$ with $\gamma >0$ sufficiently small, 
we have, as $n \to \infty$, 
uniformly in 
$x = \bb R v \in \bb P(V)$ and $f \in V^*$ with $\|v\| = \|f\| =1$, and $t \in [0, o(\sqrt{n} )]$,
\begin{align}
\frac{\bb{E} \left[ \varphi(G_n \!\cdot\! x) \mathds{1}_{ 
\left\{ \log| \langle f,  G_n v \rangle | - n \lambda_1 \geq \sqrt{n} \sigma t \right\} } 
\right] }  { 1-\Phi(t) }    
& =  e^{ \frac{t^3}{\sqrt{n}} \zeta(\frac{t}{\sqrt{n}} ) } \big[ \nu(\varphi) +  o(1) \big],  
\label{LD general upper001} \\
\frac{\bb{E} \left[ \varphi(G_n \!\cdot\! x) \mathds{1}_{ 
\left\{ \log| \langle f,  G_n v \rangle | - n \lambda_1 \leq - \sqrt{n} \sigma t  \right\} } \right] } { \Phi(-t)  }   
& =  e^{ - \frac{t^3}{\sqrt{n}} \zeta(-\frac{t}{\sqrt{n}} ) } \big[ \nu(\varphi) +  o(1) \big]. 
\label{LD general lower001}  
\end{align}
\end{theorem}

Note that the rate $o(1)$ in \eqref{MD-Expan01}, \eqref{MD-Expan02},
 \eqref{LD general upper001} and \eqref{LD general lower001} depends on 
the rate  $o(\sqrt{n})$ in $t \in [0, o(\sqrt{n} )]$.

Theorem \ref{Thm-Cram-Entry_bb} clearly implies the following moderate deviation principle
for the couple $(G_n \!\cdot\! x, \log |\langle f, G_n v \rangle|)$ 
with a target function $\varphi$ on the Markov chain $(G_n \!\cdot\! x)$:
under \ref{Ch7Condi-Moment} and \ref{Ch7Condi-IP}, 
for any sequence of positive numbers $(b_n)_{n\geq 1}$ satisfying
$\frac{b_n}{n} \to 0$ and $\frac{b_n}{\sqrt{n}} \to \infty$, 
any Borel set $B \subseteq \bb R$
and real-valued function $\varphi \in \mathscr{B}_{\gamma}$ satisfying $\nu(\varphi) > 0$,  
we have that 
uniformly in $x = \bb R v \in \bb P(V)$, $v \in V$ and $f \in V^*$ with $\|v\| = \|f\| =1$, 
\begin{align*} 
-\inf_{t \in B^{\circ}} \frac{t^2}{2\sigma^2} 
& \leq \liminf_{n\to \infty} \frac{n}{b_n^{2}}
\log  \bb{E} \Big[  \varphi( G_n \!\cdot\! x )  
\mathds{1}_{  \big\{ \frac{ \log |\langle f, G_n v \rangle| - n\lambda_1 }{b_n} \in B  \big\} }  \Big]
\nonumber\\
& \leq   \limsup_{n\to \infty}\frac{n}{b_n^{2}}
\log  \bb{E}  \Big[   \varphi( G_n \!\cdot\! x ) 
\mathds{1}_{ \big\{ \frac{ \log |\langle f, G_n v \rangle| - n\lambda_1 }{b_n} \in B  \big\} }   \Big] 
\leq - \inf_{t \in \bar{B}} \frac{t^2}{2\sigma^2},
\end{align*}
where $B^{\circ}$ and $\bar{B}$ are respectively the interior and the closure of $B$.
This moderate deviation principle is new even for $\varphi = 1$.

\subsection{Local limit theorem with moderate deviations}
In this subsection we state 
the local limit theorem with moderate deviations and target functions
for the coefficients $\langle f, G_n v \rangle$.

\begin{theorem}\label{ThmLocal02}
Assume \ref{Ch7Condi-Moment} and \ref{Ch7Condi-IP}.
Then, for any real numbers $-\infty < a_1 < a_2 < \infty$,  
we have, 
uniformly in $v \in V$ and $f \in V^*$ with $\|v\| = \|f\| =1$,  and $|t| = o(\sqrt{n})$, 
\begin{align} \label{ModerDevXX001}
\mathbb{P} \Big( \log| \langle f, G_n v \rangle | - n\lambda_1 \in [a_1, a_2] + \sqrt{n}\sigma t \Big)
= \frac{a_2 - a_1}{ \sigma \sqrt{2 \pi n} } 
 e^{ - \frac{t^2}{2} + \frac{t^3}{\sqrt{n}} \zeta(\frac{t}{\sqrt{n}} ) } [1 + o(1)]. 
\end{align}
More generally, for any $\varphi \in \mathscr{B}_{\gamma}$ with $\gamma >0$ sufficiently small, 
and any directly Riemann integrable function $\psi$ with compact support on $\bb R$, we have, as $n \to \infty$, 
uniformly in $x = \bb R v$, $v \in V$ and  $f \in V^*$ with $\|v\| = \|f\| =1$, 
and $|t| = o(\sqrt{n})$, 
\begin{align}\label{LLT-Moderate-01}
& \mathbb{E}  \Big[ \varphi(G_n \!\cdot\! x)
 \psi \Big( \log| \langle f, G_n v \rangle | - n\lambda_1 - \sqrt{n}\sigma t \Big) 
    \Big]   \nonumber\\ 
& \qquad\qquad = \frac{e^{ -\frac{t^2}{2} + \frac{t^3}{\sqrt{n}}
  \zeta(\frac{t}{\sqrt{n}} ) }}{\sigma \sqrt{2 \pi n}}  
  \left[ \nu(\varphi) \int_{\bb R} \psi(u) du  +  o(1) \right].
\end{align}
\end{theorem}

The asymptotic \eqref{ModerDevXX001} improves the recent results obtained in \cite{GQX20} and  \cite{DKW21}: 
the result in \cite{GQX20} corresponds to the case when $t=0$, 
and that in \cite{DKW21} to the case when $t = o(1)$.

\section{Proofs of the laws of large numbers} 

The following strong law of large numbers for the norm cocycle $\sigma(G_n, x)$ is due to Furstenberg \cite{Fur63}. 
Recall that $\Gamma_{\mu}$ is irreducible if no proper subspace of $V$ is $\Gamma_{\mu}$-invariant.

\begin{lemma}[\cite{Fur63}]  \label{Lem_LLN}
Assume $\int_{ \textup{GL}(V) } \log N(g)  \mu(dg) < \infty$
and that $\Gamma_{\mu}$ is irreducible.
Then, uniformly in $x \in \bb P(V)$,   
\begin{align*}
\lim_{n \to \infty} \frac{ \sigma(G_n, x) }{n}  = \lambda_1  \quad a.s.,
\end{align*}
where $\lambda_1 \in \bb R$ is the first Lyapunov exponent of $\mu$. 
\end{lemma}



%


In the 
next lemma, we state two    laws  of large numbers  for $\|G_n\|$  and  $\| \wedge^2 G_n\|$.  The first law is due to 
 Furstenberg-Kesten \cite{FK60}, which can be proved by  Kingman's subadditive ergodic theorem \cite{Kin73};  
 the second one is also an easy  consequence  of Kingman's ergodic theorem \cite{Kin73} using 
 the definition of $\lambda_2$ given in  \eqref{def-lambda2}. 

\begin{lemma}[\cite{FK60, Kin73}] \label{Lem_LLN001}
Assume $\int_{ \textup{GL}(V) } \log N(g)  \mu(dg) < \infty$.
Then, 
\begin{align*}
\lim_{n \to \infty}  \frac{1}{n} \log \| G_n \|  = \lambda_1
\quad \mbox{and} \quad 
\lim_{n \to \infty}  \frac{1}{n} \log \| \wedge^2 G_n \|  =  \lambda_1 + \lambda_2   \quad  a.s.
\end{align*}
\end{lemma}

Denote by $K$ the group of isometries of $(V, \|\cdot\|)$
and by $A^+$ the semigroup $A^+ = \{ \textrm{diag}(a_1, \ldots, a_d): a_1 \geq \ldots \geq a_d >0 \}$,
where $\textrm{diag}(a_1, \ldots, a_d)$ is a diagonal endomorphism under the basis $e_1, \ldots, e_d$ of $V$. 
The well known Cartan decomposition states that $\textup{GL}(V) = K A^+ K$. 
For every $g \in \textup{GL}(V)$, we choose a decomposition (which is not necessarily unique)
\begin{align*}
g = k_g a_g l_g,
\end{align*}
where $k_g, l_g \in K$  and $a_g \in A^+$. 
Recall that 
$(e_i^*)_{1 \leq i \leq d}$ is the dual basis in $V^*$. 
Set
\begin{align*}
x_g^M = \bb R k_g e_1  \quad  \mbox{and} \quad  y_g^m = \bb R l_g e_1^*. 
\end{align*}
Let $g^*$ denote the adjoint automorphism of $g \in \textup{GL}(V)$. 
Following \cite{BQ16, BQ16b}, 
$x_g^M \in \bb P(V)$ and $y_g^m \in \bb P(V^*)$ are called respectively the density point of $g$ and $g^*$.  

 The next result is taken from  \cite[Lemma 4.7]{BQ16}.

\begin{lemma}[\cite{BQ16}]  \label{Lem_delta_d}
For any $g \in \textup{GL}(V)$, $x = \bb R v$ and $y = \bb R f$ 
with $v \in V \setminus \{0\}$ and $f \in V^* \setminus \{0\}$, we have 
\begin{enumerate}
\item  
$\delta(x, y_g^m) \leq \frac{\|gv\|}{\|g\| \|v\|} \leq \delta(x, y_g^m) + \frac{\|\wedge^2 g\|}{\|g\|^2}$, 

\item
$\delta(x_g^M, y) \leq \frac{\|g^* f\|}{\|g\| \|f\|} \leq \delta(x_g^M, y) + \frac{\|\wedge^2 g\|}{\|g\|^2}$, 

\item
$d(g \!\cdot\! x, x_g^M) \delta(x, y_g^m) \leq \frac{\|\wedge^2 g\|}{\|g\|^2}$.  
\end{enumerate}
\end{lemma}

Using Lemmas \ref{Lem_LLN}-\ref{Lem_delta_d}, we get the following: 

\begin{proposition}\label{Prop-rates-delta01}
Assume that $\int_{ \textup{GL}(V) } \log N(g)  \mu(dg) < \infty$.
Assume also that $\lambda_1 > \lambda_2$ and that $\Gamma_{\mu}$ is irreducible.   
Then, for any $\ee >0$, 
we have that uniformly in $x \in \bb P(V)$ and $y \in \bb P(V^*)$, 
\begin{align}
& \lim_{n \to \infty} 
   \bb P \left( d (G_n \!\cdot\! x, x_{G_n}^M) \geq e^{- (\lambda_1 - \lambda_2 - \ee) n} \right) = 0,  \label{d_Gn_xM} \\
& \lim_{n \to \infty} \bb P \left( \delta(x_{G_n}^M, y) \leq  e^{- \ee n}  \right) = 0,  \label{xGn_y}  \\
& \lim_{n \to \infty} \bb P \left( \delta (G_n \!\cdot\! x, y) \leq e^{- \ee n} \right) = 0.    \label{Gnx_y}
\end{align}
\end{proposition}

\begin{proof}
We first prove \eqref{d_Gn_xM}. 
For every $g \in \textup{GL}(V)$, using the Cartan decomposition $g = k_g a_g l_g$ and the fact that $k_g$ is an isometry of $(V, \|\cdot\|)$,
we get that for any $x = \bb R v \in \bb P(V)$, 
\begin{align*}
d (g \!\cdot\! x, x_{g}^M) = \frac{\| gv \wedge k_g e_1 \|}{ \|gv\| \|k_g e_1\| }
=  \frac{\| k_g a_g l_g v \wedge k_g e_1 \|}{ \|gv\| }  
= \frac{\| a_g l_g v \wedge e_1 \|}{ \|gv\| }.
\end{align*}
Since $a_g = \textrm{diag}(a_1, \ldots, a_d)$, we have $\| a_g l_g v \wedge e_1 \| \leq a_2 \|v\|$.
This, together with the fact that $a_2 = \frac{\| \wedge^2 g \|}{\|g\|}$, implies that 
$d (g \!\cdot\! x, x_{g}^M) \leq \frac{\| \wedge^2 g \|  \|v\|}{\|g\|  \|g v\|}$. 
Hence, for any $x = \bb R v \in \bb P(V)$ and $n \geq 1$, 
\begin{align}\label{Inequ_First001}
d (G_n \!\cdot\! x, x_{G_n}^M) \leq  \frac{\| \wedge^2 G_n \|  \|v\|}{\|G_n\|  \|G_n v\|}. 
\end{align}
By Lemmas \ref{Lem_LLN} and \ref{Lem_LLN001},  we get that for any $\ee >0$,  
uniformly in $v \in V \setminus \{0\}$, 
\begin{align*}
\lim_{n \to \infty} \bb P 
\left( \left| \frac{1}{n} \log \frac{\| \wedge^2 G_n \|  \|v\|}{\|G_n\|  \|G_n v\|}
   - (\lambda_2 - \lambda_1) \right| \geq \ee \right) = 0,  
\end{align*}
which implies that uniformly in $v \in V \setminus \{0\}$, 
\begin{align*}
\lim_{n \to \infty} \bb P \left( \frac{\| \wedge^2 G_n \|  \|v\|}{\|G_n\|  \|G_n v\|} 
    \geq  e^{- (\lambda_1 - \lambda_2 - \ee) n}   \right)   =0. 
\end{align*}
Hence, using \eqref{Inequ_First001}, we obtain \eqref{d_Gn_xM}. 

We next prove \eqref{xGn_y}.  
By Lemma \ref{Lem_delta_d} (2), 
we have that for any $\ee' > 0$ and any $y = \bb R f$ with  $f \in V^* \setminus \{0\}$,
\begin{align}\label{Pf_LLN_Ine001}
\bb P \left( \delta(x_{G_n}^M, y) \leq  e^{- \ee' n}  \right)
\leq  \bb P \left(  \frac{ \|G_n^* f\|}{\|G_n^*\|  \|f\|} - \frac{\| \wedge^2 G_n \|}{\|G_n\|^2}  \leq  e^{- \ee'  n}  \right).  
\end{align}
Applying Lemma \ref{Lem_LLN001}  and Lemma \ref{Lem_LLN} to the measure $\mu^*$
($\mu^*$ is the image of the measure $\mu$ by the map $g \mapsto g^*$, 
where $g^*$ is the adjoint automorphism of $g \in \textup{GL}(V)$), 
we get that for any $\ee >0$, 
\begin{align}\label{Pf_LLN_Ine002}
\lim_{n \to \infty} 
\bb P \left(  \frac{\| \wedge^2 G_n \|}{\|G_n\|^2}   \geq  e^{- (\lambda_1 - \lambda_2 - \ee) n}  \right)  = 0, 
\end{align}
and that for any $y = \bb R f$ with  $f \in V^* \setminus \{0\}$,
\begin{align}\label{Pf_LLN_Ine003}
\lim_{n \to \infty} \bb P \left(  \frac{ \| G_n^* f\| }{\|G_n^*\|  \|f\|}   \leq  e^{- \ee n}  \right)  = 0.  
\end{align}
From \eqref{Pf_LLN_Ine002} we derive that as $n \to \infty$, 
\begin{align*}
& \bb P \left(  \frac{ \| G_n^* f\| }{\|G_n^*\|  \|f\|} - \frac{\| \wedge^2 G_n \|}{\|G_n\|^2}  \leq  e^{-\ee' n}  \right)  \notag\\
& \leq  \bb P \left(  \frac{ \|G_n^* f\| }{\|G_n^*\|  \|f\|} - \frac{\| \wedge^2 G_n \|}{\|G_n\|^2}  \leq  e^{- \ee' n},
    \frac{\| \wedge^2 G_n \|}{\|G_n\|^2}   <  e^{- (\lambda_1 - \lambda_2 - \ee) n}  \right)  + o(1)  \notag\\
 & \leq  \bb P \left(  \frac{ \|G_n^* f\| }{\|G_n^*\|  \|f\|}   \leq  e^{- \ee' n} + e^{- (\lambda_1 - \lambda_2 - \ee) n}  \right)  
   + o(1)  \notag\\
& =  o(1),
\end{align*}
where in the last line we used \eqref{Pf_LLN_Ine003}.  
This, together with \eqref{Pf_LLN_Ine001}, gives \eqref{xGn_y}.  

We finally prove \eqref{Gnx_y}.  
Since, for any $a\in \bb P(V)$ and $y\in \bb P(V^*)$ it holds that $\delta(a,y)= d(a,z)$, where $z=y^{\bot}$ is the element in $\bb P(V)$ orthogonal to $y$.
By triangular inequality, we have, for all $a,b \in \bb P(V)$ and $ y\in \bb P(V^*)$,
\begin{align*} 
\delta(a,y) = d(a,z) \leq  d(a,b) +  d(b,z)  
= d(a,b) + \delta(b,y).   
\end{align*} 
It follows that
\begin{align*}
\delta (G_n \!\cdot\! x, y) \geq  \delta(x_{G_n}^M, y) - d (G_n \!\cdot\!  x, x_{G_n}^M). 
\end{align*}
Therefore, using \eqref{d_Gn_xM}, we get that, as $n \to \infty$, 
\begin{align*}
& \bb P \left( \delta (G_n \!\cdot\! x, y) \leq e^{-\ee' n} \right)   \notag\\
& \leq   \bb P \left( \delta(x_{G_n}^M, y) - d (G_n \!\cdot\! x, x_{G_n}^M)  \leq e^{- \ee'  n} \right)   \notag\\
& \leq  \bb P \left( \delta(x_{G_n}^M, y) - d (G_n \!\cdot\! x, x_{G_n}^M)  \leq e^{- \ee' n},
       d (G_n \!\cdot\! x, x_{G_n}^M) < e^{- (\lambda_1 - \lambda_2 - \ee) n} \right) 
    + o(1) \notag\\
& \leq  \bb P \left( \delta(x_{G_n}^M, y)   \leq e^{- \ee' n} + e^{- (\lambda_1 - \lambda_2 - \ee) n}  \right) 
    + o(1)  \notag\\
& =  o(1), 
\end{align*}
where in the last line we used  \eqref{xGn_y}. 
This ends the proof of \eqref{Gnx_y}. 
\end{proof}

The following result is a direct consequence of \cite[Lemma 4.8]{BQ16}. 

%

\begin{lemma}\label{Lem-ScaNorm_second}
Assume that $\int_{ \textup{GL}(V) } \log^2 N(g)  \mu(dg) < \infty$
and that condition \ref{Ch7Condi-IP} holds. 
Then, for any $\ee > 0$, there exist constants $a_k > 0$ with $\sum_{k = 1}^{\infty} a_k < \infty$, 
such that for all $n \geq k \geq 1$, $x \in \bb P(V)$ and $y \in \bb P(V^*)$, 
\begin{align*} 
\bb{P} \left( \delta(y, G_n \!\cdot\! x) \leq e^{- \ee k} \right) \leq  a_k. 
\end{align*}
\end{lemma}

\begin{proof}
Denote $G_{n, k} = g_n \ldots g_{n - k + 1}$ for $n \geq k$. 
Notice that 
\begin{align*}
& \bb{P} \left( \delta(y, G_n \!\cdot\! x) \leq e^{- \ee k} \right)  \nonumber\\
& = \int_{\textup{GL}(V)^{n-k}} \bb{P} \left( \delta(y, G_{n,k} g_{n-k} \ldots g_1 \!\cdot\! x) \leq e^{- \ee k} \right)
   \mu(dg_{n-k}) \ldots \mu(dg_1).  
\end{align*}
By \cite[Lemma 4.8]{BQ16}, for any $\ee > 0$, there exist constants $a_k > 0$ with $\sum_{k = 1}^{\infty} a_k < \infty$, 
such that for all $n \geq k \geq 1$, $x \in \bb P(V)$ and $y \in \bb P(V^*)$, 
\begin{align*} 
\bb{P} \left( \delta(y, G_n \!\cdot\! x) \leq e^{- \ee k} \right) \leq  a_k. 
\end{align*}
The desired result follows. 
\end{proof}

\begin{proof}[Proof of Theorem \ref{Thm_LLN}]
We first prove the assertion \eqref{Weak_LLN_Coeff}. 
Note that for any $x = \bb R v \in \bb P(V)$ and $y = \bb R f \in \bb P(V^*)$ with 
$\|v\| = \|f\| =1$,   
\begin{align}\label{Pf_LLN_Equality}
\log |\langle f, G_n v \rangle| = \sigma(G_n, x) + \log \delta(y, G_n \!\cdot\! x). 
\end{align}
Using Lemma \ref{Lem_LLN} and  \eqref{Gnx_y},
we get that, as $n \to \infty$,  $\frac{1}{n}  \log |\langle f, G_n v \rangle|$ converges to $\lambda_1$ in probability,
uniformly in $v \in V$ and $f \in V^*$ with $\|v\| = \|f\| =1$. 
This, together with the fact that the sequence $\{ \frac{1}{n}  \log |\langle f, G_n v \rangle| \}_{n\geq 1}$
is uniformly integrable, proves \eqref{Weak_LLN_Coeff}.

Now we prove the almost sure convergence based on Lemmas \ref{Lem_LLN} and \ref{Lem-ScaNorm_second}. 
By Lemma \ref{Lem-ScaNorm_second} and Borel-Cantelli's lemma, 
we get that for any $\ee >0$, 
\begin{align*}
\liminf_{n \to \infty} \frac{1}{n} \log \delta(y, G_n \!\cdot\! x) >  - \ee   \quad  \mbox{a.s.}  
\end{align*}
Together with Lemma \ref{Lem_LLN} and the fact that $\delta(y, G_n \!\cdot\! x) \leq 1$,
this yields that  $\lim_{n \to \infty} \frac{1}{n}  \log |\langle f, G_n v \rangle| = \lambda_1$, a.s.
\end{proof}

In the proof of Theorem \ref{Thm_CLT}, we will use the central limit theorem for the norm cocycle $\sigma(G_n, x)$. 
Under the exponential moment condition \ref{Ch7Condi-Moment}, this result is due to Le Page \cite{LeP82}. 
Recently, using the martingale approximation approach and the log-regularity of the invariant measure $\nu$, 
Benoist and Quint \cite{BQ16} relaxed condition \ref{Ch7Condi-Moment} 
to the optimal second moment condition. 

\begin{lemma}[\cite{BQ16}] \label{Lem_CLT}
Assume that $\int_{ \textup{GL}(V) } \log^2 N(g) \, \mu(dg) < \infty$
and that condition \ref{Ch7Condi-IP} holds. 
Then, for any $t \in \bb R$, it holds that uniformly in $x \in \bb P(V)$,   
\begin{align*}
\lim_{n \to \infty}
\bb{P} \left( \frac{\sigma(G_n, x) - n\lambda}{ \sigma \sqrt{n} } \leq t \right) 
= \Phi(t). 
\end{align*}
\end{lemma}

Now we prove Theorem \ref{Thm_CLT} using Lemmas \ref{Lem-ScaNorm_second} and \ref{Lem_CLT}. 

\begin{proof}[Proof of Theorem \ref{Thm_CLT}]
By Lemma \ref{Lem_CLT} and Slutsky's theorem, 
it suffices to prove that $\frac{1}{\sigma \sqrt{n}} \log \delta(y, G_n \!\cdot\! x)$ converges to $0$ in probability, as $n \to \infty$.
Namely, we need to show that, for any $\ee >0$, uniformly in $x \in \bb P(V)$ and $y \in \bb P(V^*)$,
\begin{align*}
\lim_{n \to \infty} \bb{P} \left( \frac{\log \delta(y, G_n \!\cdot\! x)}{ \sigma \sqrt{n} } < - \ee \right)  = 0. 
\end{align*}
To prove this, taking $k = \floor[]{\sqrt{n}}$ in Lemma \ref{Lem-ScaNorm_second},
we get that uniformly in $x \in \bb P(V)$ and $y \in \bb P(V^*)$,
\begin{align*}
\bb{P} \left( \frac{\log \delta(y, G_n \!\cdot\! x)}{ \sigma \sqrt{n} } < - \ee \right)
=  \bb{P} \left(  \delta(y, G_n \!\cdot\! x)  < e^{- \ee \sigma \sqrt{n}} \right)
\leq a_{\floor[]{\sqrt{n}}}, 
\end{align*}
where $a_{\floor[]{\sqrt{n}}}$ converges to $0$ as $n \to \infty$, since $\sum_{k = 1}^{\infty} a_k < \infty$. 
\end{proof}

\section{Proof of the Edgeworth expansion} 

\subsection{Spectral gap properties and a change of measure}\label{subsec-Pz}

For any $z \in \bb{C}$ with $|\Re z|$ small enough, 
we define the complex transfer operator $P_z$ as follows: 
for any bounded measurable function $\varphi$ on $\bb P(V)$, 
\begin{align}\label{Def_Pz_Ch7}
P_z \varphi(x) = \int_{\textup{GL}(V)} e^{z \sigma(g, x)} \varphi(g \!\cdot\! x) \mu(dg),  
\quad  x \in \bb P(V). 
\end{align}
Throughout this paper let $B_{s_0}(0): = \{ z \in \bb{C}: |z| < s_0 \}$
be the open disc with center $0$ and radius $s_0 >0$ in the complex plane $\bb C$. 
The following result  
shows that the operator $P_z$ has spectral gap properties 
when $z \in B_\eta(0)$;  
we refer to \cite{LeP82, HH01, GL16, BQ16b, XGL19b} for the proof 
based on the perturbation theory of linear operators. 
Recall that $\mathscr{B}_{\gamma}'$ is the topological dual space of the Banach space $\mathscr{B}_{\gamma}$, 
and that $\mathscr{L(B_{\gamma},B_{\gamma})}$ 
is the set of all bounded linear operators from $\mathscr{B}_{\gamma}$ to $\mathscr{B}_{\gamma}$
equipped with the operator norm
$\left\| \cdot \right\|_{\mathscr{B}_{\gamma} \to \mathscr{B}_{\gamma}}$.


\begin{lemma}[\cite{BQ16b, XGL19b}]  \label{Ch7transfer operator}
Assume \ref{Ch7Condi-Moment} and \ref{Ch7Condi-IP}.  
Then, there exists a constant $s_0 >0$ such that for any $z \in B_{s_0}(0)$ and $n \geq 1$, 
\begin{align}\label{Ch7Pzn-decom}
P_z^n = \kappa^n(z) \nu_z \otimes r_z + L_z^n, 
\end{align}
where 
\begin{align*}
z \mapsto \kappa(z) \in \bb{C}, \quad   z \mapsto r_z \in \mathscr{B}_{\gamma} , 
\quad   z \mapsto \nu_z \in \mathscr{B}_{\gamma}' , 
\quad   z \mapsto  L_z \in \mathscr{L(B_{\gamma},B_{\gamma})}
\end{align*}
are analytic mappings which satisfy, for any $z \in B_{s_0}(0)$, 

\begin{itemize}
\item[{\rm(a)}]
    the operator $M_z: = \nu_z \otimes r_z$ is a rank one projection on $\mathscr{B}_{\gamma}$,
    i.e. $M_z \varphi = \nu_z(\varphi) r_z$ for any $\varphi \in \mathscr{B}_{\gamma}$; 

\item[{\rm(b)}]
 $M_z L_z = L_z M_z =0$,  $P_z r_z = \kappa(z) r_z$ with $\nu(r_z) = 1$, and  $\nu_z P_z = \kappa(z) \nu_z$;

\item[{\rm(c)}]
    $\kappa(0) = 1$, $r_0 = 1$, $\nu_0 = \nu$ with $\nu$ defined by \eqref{Ch7mu station meas}, and 
    $\kappa(z)$ and $r_z$ are strictly positive for real-valued $z \in (-s_0, s_0)$.    

\end{itemize}
\end{lemma}


Using Lemma \ref{Ch7transfer operator}, 
a change of measure can be performed below.
Specifically, for any $s \in (-s_0, s_0)$ with $s_0>0$ sufficiently small, 
 any $x \in \bb P(V)$ and $g \in \textup{GL}(V)$, denote
\begin{align*}
q_n^s(x, g) = \frac{ e^{s \sigma(g, x) } }{ \kappa^{n}(s) } \frac{ r_s(g \!\cdot\! x) }{ r_s(x) },
\quad  n \geq 1. 
\end{align*}
Since the eigenvalue $\kappa(s)$ and the eigenfunction $r_s$ are strictly positive for $s \in (-s_0, s_0)$, 
using $P_s r_s = \kappa(s) r_s$ we get that 
\begin{align*}
\bb Q_{s,n}^x (dg_1, \ldots, dg_n) = q_n^s(x, G_n) \mu(dg_1) \ldots \mu(dg_n),  \quad  n \geq 1, 
\end{align*}
are probability measures and 
form a projective system on $\textup{GL}(V)^{\bb{N}}$. 
By the Kolmogorov extension theorem, 
there is a unique probability measure  $\bb Q_s^x$ on $\textup{GL}(V)^{\bb{N}}$ with marginals $\bb Q_{s,n}^x$. 
We write $\bb{E}_{\bb Q_s^x}$ for the corresponding expectation 
and the change of measure formula holds: 
for any $s \in (-s_0, s_0)$, $x \in \bb P(V)$, 
$n\geq 1$ and bounded measurable function $h$ on $(\bb P(V) \times \bb R)^{n}$,  
\begin{align}\label{Ch7basic equ1}
&  \frac{1}{ \kappa^{n}(s) r_{s}(x) } \bb{E}  \Big[   
r_{s}(G_n \!\cdot\! x) e^{s \sigma(G_n, x) } h \Big( G_1 \!\cdot\! x, \sigma(G_1, x), \dots, G_n \!\cdot\! x, \sigma(G_n, x) \Big)
  \Big]   \nonumber\\
&  =   \bb{E}_{\bb{Q}_{s}^{x}} \Big[ h \Big( G_1 \!\cdot\! x, \sigma(G_1, x), \dots, G_n \!\cdot\! x, \sigma(G_n, x) \Big) \Big].
\end{align}
Under the changed measure $\bb Q_s^x$, the process $(G_n \!\cdot\! x)_{n \geq 0}$ is a Markov chain 
with the transition operator $Q_s$ given as follows: for any $\varphi \in \mathscr{C}(\bb P(V))$, 
\begin{align*}
Q_{s}\varphi(x) = \frac{1}{\kappa(s)r_{s}(x)}P_s(\varphi r_{s})(x),  \quad   x \in \bb P(V).
\end{align*}
%
Under \ref{Ch7Condi-Moment} and \ref{Ch7Condi-IP}, 
it was shown in \cite{XGL19b} that 
the Markov operator $Q_s$ has a unique invariant probability measure $\pi_s$ given by 
\begin{align}\label{ExpCon-Qs}
\pi_s(\varphi) = \frac{ \nu_s( \varphi r_s) }{ \nu_s(r_s) }  \quad  \mbox{for any } \varphi \in \mathscr{C}(\bb P(V)).
\end{align}
By \cite[Proposition 3.13]{XGL19b}, the following strong law of large numbers 
for the norm cocycle under the changed measure $\bb Q_s^x$ holds: under \ref{Ch7Condi-Moment} and \ref{Ch7Condi-IP}, 
for any $s\in (-s_0, s_0)$ and $x \in \bb P(V)$, 
\begin{align*}
\lim_{n \to \infty} \frac{ \sigma(G_n, x) }{n} = \Lambda'(s),  \quad  \bb Q_s^x\mbox{-a.s.}
\end{align*}
where $\Lambda(s) = \log \kappa(s)$. 
For any $s\in (-s_0, s_0)$ and $u \in \bb R$, 
we define the perturbed operator $R_{s, iu}$ as follows: for $\varphi \in \mathscr{C}(\bb P(V))$, 
\begin{align} \label{Def_Rsz_Ch7}
R_{s, iu} \varphi(x)
= \bb{E}_{\bb{Q}_{s}^{x}} \left[ e^{iu( \sigma(g, x) - \Lambda'(s) )} \varphi(g \!\cdot\! x) \right],  
\quad   x \in \bb P(V). 
\end{align}
It follows from the cocycle property of $\sigma(\cdot, \cdot)$ that for any $n \geq 1$, 
\begin{align*} 
R^{n}_{s, iu} \varphi(x)
= \bb{E}_{\bb{Q}_{s}^{x}} \left[ e^{iu( \sigma(G_n, x) - n\Lambda'(s) )} \varphi(G_n \!\cdot\! x) \right],
\quad    x \in \bb P(V).
\end{align*}
The next result gives spectral gap properties of the perturbed operator $R_{s, iu}$. 

\begin{lemma}[\cite{XGL19b}] \label{Ch7perturbation thm}
Assume \ref{Ch7Condi-Moment} and \ref{Ch7Condi-IP}.  
Then, there exist constants $s_0 >0$ and $\delta > 0$ 
such that for any $s \in (-s_0, s_0)$ and $u \in  (-\delta, \delta)$,
\begin{align} \label{Ch7perturb001}
R^{n}_{s, iu} = \lambda^{n}_{s, iu} \Pi_{s, iu} + N^{n}_{s, iu},  
\end{align}
where
\begin{align}\label{relationlamkappa001}
\lambda_{s, iu} = e^{ \Lambda(s + iu) - \Lambda(s) - iu \Lambda'(s)}, 
\end{align}
and for fixed $s \in (-s_0, s_0)$, the mappings $u \mapsto \Pi_{s, iu}: (-\delta, \delta) \to \mathscr{L(B_{\gamma},B_{\gamma})}$,
$u \mapsto N_{s, iu}: (-\delta, \delta) \to \mathscr{L(B_{\gamma},B_{\gamma})}$
and $u \mapsto \lambda_{s, iu}: (-\delta, \delta) \to \bb{R}$ are analytic. 
In addition, for fixed $s$ and $u$, the operator $\Pi_{s,iu}$ is a rank-one projection with 
$\Pi_{s, 0}(\varphi)(x) = \pi_{s}(\varphi)$ for any $\varphi \in \mathscr{B}_{\gamma}$ and $x\in \bb P(V)$,
and $\Pi_{s,iu} N_{s, iu} = N_{s,iu} \Pi_{s,iu} = 0$.  
  
Moreover,  for any $k \in  \bb{N}$,  there exist constants $c > 0$ and $0< a <1$ such that
\begin{align} \label{Ch7SpGapContrN}
\sup_{|s| < s_0} \sup_{|u| < \delta}
\Big\| \frac{d^{k}}{du^{k}} \Pi^{n}_{s, iu}  \Big\|_{\mathscr{B}_{\gamma} \to \mathscr{B}_{\gamma}} 
\leq c,  
\
\sup_{|s| < s_0} \sup_{|u| < \delta}
\Big\| \frac{d^{k}}{du^{k}}N^{n}_{s, iu}  \Big\|_{\mathscr{B}_{\gamma} \to \mathscr{B}_{\gamma}} 
\leq  c  a^n.    
\end{align}
\end{lemma} 

We end this subsection by giving the non-arithmetic property of the perturbed operator $R_{s,iu}$. 

\begin{lemma}[\cite{XGL19b}] \label{Lem-St-NonLatt}
Assume \ref{Ch7Condi-Moment} and \ref{Ch7Condi-IP}.  
Then, for any compact set $K \subseteq \mathbb{R}\backslash\{0\}$,
there exist constants $s_0, c, C_{K}>0$ such that 
for any $n \geq 1$ and $\varphi\in \mathscr{B}_{\gamma}$,
\begin{align*} 
\sup_{s \in (-s_0, s_0)} \sup_{u\in K} \sup_{x\in \bb P(V)}
|R^{n}_{s, iu}\varphi(x)| \leq c e^{-nC_{K}} \|\varphi \|_{\gamma}.
\end{align*}
\end{lemma}

\subsection{Exponential H\"older regularity of the invariant measure $\pi_s$}

We need the following exponential regularity of the invariant measure $\pi_s$ from \cite{GQX20}.
\begin{lemma}[\cite{GQX20}]  \label{Lem_Regu_pi_s00}
Assume \ref{Ch7Condi-Moment} and \ref{Ch7Condi-IP}. Then there exist constants $s_0 >0$ and $\eta > 0$ such that
\begin{align} \label{Regu_pi_s_00}
\sup_{ s\in (-s_0, s_0) } \sup_{y \in \bb P(V^*) } 
  \int_{\bb P(V) } \frac{1}{ \delta(y, x)^{\eta} } \pi_s(dx) < + \infty. 
\end{align}
\end{lemma}

We also need the following property: 

\begin{lemma}[\cite{GQX20}]  \label{Lem_Regu_pi_s}
Assume \ref{Ch7Condi-Moment} and \ref{Ch7Condi-IP}. 
Then, for any $\ee >0$, there exist constants $s_0 >0$ and $c, C >0$ such that
for all $s \in (-s_0, s_0)$, $n \geq k \geq 1$, $x \in \bb P(V)$ and $y \in \bb P(V^*)$, 
\begin{align}\label{Regu_pi_s}
\bb Q_s^x \Big( \log \delta(y, G_n \!\cdot\! x) \leq -\ee k  \Big) \leq C e^{- ck}. 
\end{align}
\end{lemma} 

Note that \eqref{Regu_pi_s} is stronger than the exponential H\"{o}lder regularity of the invariant measure $\pi_s$
stated in Lemma  \ref{Lem_Regu_pi_s00}.


\subsection{Proof of Theorem \ref{Thm-Edge-Expan-Coeff001}} \label{sec-proof of Edgeworth exp}
We shall prove a more general version of 
Theorem \ref{Thm-Edge-Expan-Coeff001} 
under the changed measure $\mathbb{Q}_{s}^{x}$.
For any $s \in (-s_0, s_0)$ and $\varphi \in \mathscr{B}_{\gamma}$,   define 
\begin{align}\label{Def-bsvarphi}
b_{s, \varphi}(x): = \lim_{n \to \infty}
   \mathbb{E}_{\mathbb{Q}_{s}^{x}} \big[ ( \sigma(G_n, x) - n \Lambda'(s) ) \varphi(G_n \!\cdot\! x) \big],
\quad   x \in \bb P(V) 
\end{align}
and
\begin{align} \label{drift-d001bis}
d_{s,\varphi}(y)=  \int_{\bb P(V)} \varphi(x) \log \delta(y,x) \pi_s(dx), \quad  y \in \bb P(V^*). 
\end{align}
These functions are well-defined as shown in Lemmas \ref{Lem-Bs} and \ref{Lem-ds} below.
In particular, we have $b_{0,\varphi} = b_{\varphi}$ and $d_{0,\varphi} = d_{\varphi}$,
 where $b_{\varphi}$ and $d_{\varphi}$ are defined in \eqref{drift-b001} and \eqref{drift-d001}, respectively. 

Our goal of this subsection is to establish the following first-order Edgeworth expansion for the coefficients $\langle f, G_n v \rangle$
 under the changed measure $\mathbb{Q}_{s}^x$. 
 Note that $\sigma_s = \sqrt{\Lambda''(s)}$, which is strictly positive under \ref{Ch7Condi-Moment} and \ref{Ch7Condi-IP}. 

\begin{theorem}\label{Thm-Edge-Expan-Coeff001extended}
Assume \ref{Ch7Condi-Moment} and \ref{Ch7Condi-IP}. 
Then, 
for any $\ee > 0$,
 there exist constants $\gamma >0$ and $s_0 > 0$ such that uniformly in $s \in (-s_0, s_0)$, 
 $t \in \bb R$, $x=\bb R v \in \bb P(V)$, $ y = \bb R f \in \bb P(V^*)$ with $\|v\| = \|f\| =1$, and
  $\varphi \in \mathscr{B}_{\gamma}$,  as $n\to \infty$,
\begin{align*}
&  \mathbb{E}_{\mathbb{Q}_{s}^x}
   \Big[  \varphi(G_n \!\cdot\! x) \mathds{1}_{ \big\{ \frac{\log |\langle f, G_n v \rangle| - n \Lambda'(s)  }{\sigma_s \sqrt{n}} \leq t \big\} } \Big]
   \notag \\
 &  \qquad  = \pi_s(\varphi) \Big[  \Phi(t) + \frac{\Lambda'''(s)}{ 6 \sigma_s^3 \sqrt{n}} (1-t^2) \phi(t) \Big]
  - \frac{ b_{s,\varphi}(x) + d_{s,\varphi}(y) }{ \sigma_s \sqrt{n} } \phi(t)    \notag\\
& \qquad\quad  +  \pi_s(\varphi)  o \Big( \frac{ 1 }{\sqrt{n}} \Big)  +  \lVert \varphi \rVert_{\gamma} O \Big( \frac{ 1 }{n^{1-\ee} }\Big).     
\end{align*}
\end{theorem}
Theorem \ref{Thm-Edge-Expan-Coeff001} 
follows from Theorem \ref{Thm-Edge-Expan-Coeff001extended} by taking $s=0$.

We begin with some properties of the function $b_{s, \varphi}$ (cf.\ \eqref{Def-bsvarphi}) proved recently in \cite{XGL19b}. 
\begin{lemma}[\cite{XGL19b}] \label{Lem-Bs}
Assume \ref{Ch7Condi-Moment} and \ref{Ch7Condi-IP}. 
Then, there exists $s_0 >0$ such that for any $s\in (-s_0, s_0)$, the function $b_{s,\varphi}$  
is well-defined and
\begin{align} \label{Def2-bs}
b_{s,\varphi}(x) = \frac{ d \Pi_{s,z} }{ dz } \Big|_{z=0} \varphi(x),  \quad  x \in \bb P(V). 
\end{align}
Moreover, there exist constants $\gamma>0$ and $c>0$ such that $b_{s,\varphi}  \in \mathscr{B}_{\gamma}$ 
and $\| b_{s,\varphi}\|_{\gamma}  \leq  c \| \varphi \| _{\gamma}$ for any $s\in (-s_0, s_0)$.
\end{lemma}

%
%
%
%

In addition to Lemma \ref{Lem-Bs}, we prove the following result on the function $d_{s,\varphi}$ defined in \eqref{drift-d001bis}. 
\begin{lemma} \label{Lem-ds}
Assume \ref{Ch7Condi-Moment} and \ref{Ch7Condi-IP}.
Then, there exists $s_0 >0$ such that for any $s\in (-s_0, s_0)$, the function 
$d_{s,\varphi}$ is well-defined. 
Moreover, there exist constants $\gamma>0$ and $c>0$ such that $d_{s,\varphi}  \in \mathscr{B}^*_{\gamma}$ 
and $\| d_{s,\varphi}\|_{\gamma}  \leq  c \| \varphi \|_{\infty}$ for any $s\in (-s_0, s_0)$.
\end{lemma}
\begin{proof}
Since $a\leq e^{a}$ for  any $a\geq 0$, for any $\eta\in (0,1)$ we have 
$- \eta \log \delta(y,x) \leq \delta(y,x)^{-\eta}$, so that
\begin{align*} 
-d_{s,\varphi}(y) 
\leq 
\frac{\| \varphi \| _{\infty}}{\eta} \int_{\bb P(V)} \frac{1}{ \delta(y,x)^{\eta} } \pi_s(dx). 
\end{align*}
Choosing $\eta$ small enough,  by Lemma \ref{Lem_Regu_pi_s00}, the latter integral is bounded uniformly in $y\in \bb P(V^*)$
and $s \in (-s_0, s_0)$,
which proves that $d_{s,\varphi}$ is well defined and that $\|d_{s,\varphi}\|_{\infty} \leq c \| \varphi \|_{\infty}$.

Note that $|\log(1+a)| \leq  c |a|$ for any $|a| \leq \frac{1}{2}$.  Then, using this, for any 
$y'=\bb R f'\in \bb P(V^*)$, $y''=\bb R f''\in \bb P(V^*)$ and any $\gamma>0$, we deduce that
\begin{align*} 
&\left| \log \delta(y',x)-\log \delta(y'',x) \right| \\
&= \left| \log \delta(y',x)-\log \delta(y'',x) \right| 
     \mathds 1_{ \big\{ \big| \frac{\delta(y',x)-\delta(y'',x)}{\delta(y'',x)} \big| > \frac{1}{2} \big\} } \\
& \quad + \left| \log \delta(y',x)-\log \delta(y'',x) \right|
    \mathds 1_{ \big\{ \big|\frac{\delta(y',x)-\delta(y'',x)}{\delta(y'',x)} \big| \leq \frac{1}{2} \big\} }  \\
&\leq  2^{\gamma}\left( \left| \log \delta(y',x) \right| + \left| \log \delta(y'',x) \right| \right) 
    \left|\frac{\delta(y',x)-\delta(y'',x)}{\delta(y'',x)} \right|^{\gamma} \\
& \quad +  c^{\gamma} \left| \log \delta(y',x)-\log \delta(y'',x) \right|^{1-\gamma} 
   \left|\frac{\delta(y',x)-\delta(y'',x)}{\delta(y'',x)} \right|^{\gamma}.
\end{align*}
Taking into account the fact that   
$- \eta \log \delta(y,x) \leq \delta(y,x)^{-\eta}$,
from the previous bound, we obtain  
\begin{align*} 
&|\log \delta(y',x)-\log \delta(y'',x)| \\
&\leq c_{\eta} \left(\delta(y',x)^{-\eta} \delta(y'',x)^{-\gamma} + \delta(y'',x)^{-\eta-\gamma} \right) \left|\delta(y',x)-\delta(y'',x)\right|^ {\gamma}\\
&+ c_{\eta,\gamma} \left(\delta(y',x)^{-\eta (1-\gamma) } \delta(y'',x)^{-\gamma} + \delta(y'',x)^{-\eta (1-\gamma)-\gamma} \right) \left|\delta(y',x)-\delta(y'',x)\right|^ {\gamma}\\
&\leq c_{\eta,\gamma} \left(\delta(y',x)^{-\eta} \delta(y'',x)^{-\gamma} + \delta(y'',x)^{-\eta-\gamma} \right) \left|\delta(y',x)-\delta(y'',x)\right|^ {\gamma}\\
&\leq c_{\eta,\gamma} \left(\delta(y',x)^{-2\eta} +  \delta(y'',x)^{-2\gamma} + \delta(y'',x)^{-\eta-\gamma} \right) \left|\delta(y',x)-\delta(y'',x)\right|^ {\gamma}.
\end{align*}
Since $\| \frac{f'}{\| f' \|} - \frac{f''}{\| f \|} \| \leq  d(y',y'')$ where $d(y',y'')$ 
is the angular distance on $\bb P(V^*)$,
it follows that
$$ 
\left|\delta(y',x)-\delta(y'',x)\right|= \left|\frac{v(f')}{\|v\| \|f'\|} - \frac{v (f'')}{\|v\| \|f''\|}\right| 
\leq \Big\| \frac{f'}{\| f' \|} - \frac{f''}{\| f \|} \Big\| 
\leq d(y',y''). 
$$
By the definition of the function $d_{s,\varphi}$, using the above bounds, we obtain
\begin{align*} 
&\frac{| d_{s,\varphi}(y')-d_{s,\varphi}(y'')|}{d(y',y'')^{\gamma}} \\
&\leq c_{\eta,\gamma} \| \varphi \|_{\infty}  
\int_{\bb P(V)} \left(\delta(y',x)^{-2\eta} + \delta(y'',x)^{-2\gamma}  + \delta(y'',x)^{-\eta-\gamma} \right) \pi_s(dx).
\end{align*}
By choosing $\eta$ and $\gamma$ sufficiently small,
the last integral is bounded uniformly in $y',y''\in \bb P(V^*)$ and $s \in (-s_0, s_0)$ by Lemma \ref{Lem_Regu_pi_s00}. 
This proves that 
$ d_{s,\varphi} \in \scr B_{\gamma}^*$.
\end{proof}

In the proof of Theorem \ref{Thm-BerryEsseen-optimal} 
we make use of the following Edgeworth expansion for the couple $(G_n \cdot x, \sigma(G_n, x))$
with a target function on $G_n \cdot x$, which slightly improves \cite[Theorem 5.3]{XGL19b} 
by giving more accurate reminder terms. 
This improvement will be important for establishing Theorem \ref{Thm-BerryEsseen-optimal}. 

\begin{theorem}\label{Thm-Edge-Expan}
Assume \ref{Ch7Condi-Moment} and \ref{Ch7Condi-IP}. 
Then, there exists  $s_0 >0$ such that, 
as $n \to \infty$, uniformly in   $s \in (-s_0, s_0)$, 
$x \in \bb P(V)$, $t \in \bb R$ and $\varphi \in \mathscr{B}_{\gamma}$,  
\begin{align*}
&  \mathbb{E}_{\mathbb{Q}_s^x}
   \Big[  \varphi(G_n \cdot x) \mathds{1}_{ \big\{ \frac{\sigma(G_n, x) - n \Lambda'(s) }{\sigma_s \sqrt{n}} \leq t \big\} } \Big]
   \notag \\
 &  =  \pi_s(\varphi) \Big[  \Phi(t) + \frac{\Lambda'''(s)}{ 6 \sigma_s^3 \sqrt{n}} (1-t^2) \phi(t) \Big]
    -  \frac{ b_{s,\varphi}(x) }{ \sigma_s \sqrt{n} } \phi(t)     \notag\\
&  \quad  +  \pi_s(\varphi)  o \Big( \frac{ 1 }{\sqrt{n}} \Big)  +  \lVert \varphi \rVert_{\gamma} O \Big( \frac{ 1 }{n} \Big).     
\end{align*}
\end{theorem}

\begin{proof}
For any $x \in \bb P(V)$, define
\begin{align*}
F(t) & = \mathbb{E}_{\mathbb{Q}_s^x}
\Big[  \varphi(G_n \!\cdot\! x) \mathds{1}_{ \big\{ \frac{\sigma(G_n, x) - n \Lambda'(s) }{\sigma_s \sqrt{n}} \leq t \big\} } \Big]
  +  \frac{ b_{s,\varphi}(x) }{ \sigma_s \sqrt{n} } \phi(t),  
  \quad t \in \mathbb{R},   \notag\\
H(t) & =   \mathbb{E}_{\mathbb{Q}_s^x} [ \varphi(G_n \!\cdot\! x) ]
 \Big[ \Phi(t) + \frac{\Lambda'''(s)}{ 6 \sigma_s^3 \sqrt{n}} (1-t^2) \phi(t) \Big],  \quad  t\in \mathbb{R}.
\end{align*}
Since $F(-\infty) = H(-\infty) = 0$ and $F(\infty) = H(\infty)$, 
applying Proposition 4.1 of \cite{XGL19b} we get that 
\begin{align}\label{BerryEsseen001}
  \sup_{t \in \mathbb{R}}  \big| F(t) - H(t)  \big|
\leq  \frac{1}{\pi }  ( I_1 + I_2 + I_3 + I_4),
\end{align}
where
\begin{align*}
I_1  & =   o \Big( \frac{ 1 }{\sqrt{n}} \Big) \sup_{t \in \mathbb{R}} |H'(t)|,  
\quad    I_2   \leq C e^{-cn} \|\varphi \|_{\gamma},  
\quad  I_3    \leq  \frac{c}{n} \|\varphi \|_{\gamma},   \quad 
I_4   \leq  \frac{c}{n} \|\varphi \|_{\gamma}. 
\end{align*}
Here the bounds for $I_2$, $I_3$ and $I_4$ are obtained in \cite{XGL19b}. 
It is easy to see that 
\begin{align*}
I_1 =   o \Big( \frac{ 1 }{\sqrt{n}} \Big)  \mathbb{E}_{\mathbb{Q}_s^x} \Big[  \varphi(G_n \!\cdot\! x)  \Big]. 
\end{align*}
This, together with the fact that 
\begin{align*}
\mathbb{E}_{\mathbb{Q}_s^x} \Big[  \varphi(G_n \!\cdot\! x) \Big] \leq  \pi_s(\varphi) +  C e^{-cn} \|\varphi \|_{\gamma}
\end{align*}
(cf.\ Lemma \ref{Ch7perturbation thm}), 
proves the theorem. 
\end{proof}

In the following we shall construct a partition $(\chi_{n,k}^y)_{k \geq 0}$ of the unity on the projective space $\bb P(V)$,
which is similar to the partitions in \cite{XGL19d, GQX20, DKW21}.  
In contrast to \cite{XGL19d, GQX20}, there is no escape of mass in our partition, 
which simplifies the proofs. 
Our partition becomes finer when $n \to \infty$, which allows us to obtain precise expressions for remainder terms
in the central limit theorem 
and thereby to establish the Edgeworth expansion for the coefficients. 

Let $U$ be the uniform distribution function on the interval $[0,1]$: $U(t)=t$ 
for $t\in [0,1]$, $U(t)=0$ for $t <0$ 
and $U(t)=1$ for $t > 1$. 
 Let $a_n=\frac{1}{\log n}$.
 Here and below we assume that $n \geq 18$ so that $a_n e^{a_n} \leq \frac{1}{2}$.
For any integer $k\geq 0$, define 
\begin{align*} 
U_{n,k}(t)= U\left(\frac{t-(k-1) a_n}{a_n}\right),  \qquad 
h_{n,k}(t)=U_{n,k}(t) - U_{n,k+1}(t),  \quad  t \in \bb R. 
\end{align*}
It is easy to see that $U_{n,m} = \sum_{k=m}^\infty h_{n,k}$ for any $m\geq 0$. 
Therefore, for any $t\geq 0$ and $m\geq0$, we have
\begin{align} \label{unity decomposition h-001}
\sum_{k=0}^{\infty} h_{n,k} (t) =1, \quad \sum_{k=0}^{m} h_{n,k} (t) + U_{n,m+1} (t) =1.
\end{align}
Note that, for any $k\geq 0$, 
\begin{align} \label{h_kLip001}
\sup_{s,t\geq 0: s \neq t} \frac{ | h_{n,k}(s) - h_{n,k}(t) |}{|s-t|} \leq \frac{1}{a_{n}}. 
\end{align}
For any $x=\bb R v \in \bb P(V)$ and $y=\bb R f \in \bb P(V^*)$, 
set 
\begin{align}\label{Def-chi-nk}
\chi_{n,k}^y(x)=h_{n,k}(-\log \delta(y, x))  \quad  \mbox{and}  \quad 
\overline \chi_{n,k}^y(x)= U_{n,k} ( -\log \delta(y, x) ), 
\end{align}
where we recall that $-\log\delta(y, x) \geq 0$ for any $x\in \bb P (V)$ and $y \in \bb P(V^*)$. 
From \eqref{unity decomposition h-001}
 we have the following partition of the unity  on $\bb P(V)$: for any $x\in \bb P (V)$, $y \in \bb P(V^*)$ and $m\geq 0$,
\begin{align} \label{Unit-partition001}
\sum_{k=0}^{\infty} \chi_{n,k}^y (x) =1, \quad 
\sum_{k=0}^{m} \chi_{n,k}^y (x) + \overline \chi_{n,m+1}^y (x) =1.
\end{align}
Denote by $\supp (\chi_{n,k}^y)$ the support of the function $\chi_{n,k}^y$.  
It is easy to see that for any $k\geq 0$ and $y\in \bb P(V^*)$,
\begin{align} \label{on the support on chi_k-001}
 -\log \delta(y, x) \in [a_n (k-1), a_n(k+1)] \quad \mbox{for any}\ x\in \supp (\chi_{n,k}^y). 
\end{align}

\begin{lemma} \label{lemmaHolder property001}
There exists a constant $c>0$ such that 
for any $\gamma\in(0,1]$, 
 $k\geq 0$ and $y\in \bb P(V^*)$, it holds
  $\chi_{n,k}^y\in \scr B_{\gamma}$ and, moreover, 
\begin{align} \label{Holder prop ohCHI_k-001}
\| \chi_{n,k}^y \|_{\gamma} \leq \frac{c e^{\gamma k a_n}}{a_{n}^\gamma}.
\end{align}
\end{lemma}

\begin{proof} 
Since $\|\chi_{n,k}^y\|_{\infty} \leq 1$,
it is enough to give a bound for the modulus of continuity: 
\begin{align*} 
[\chi_{n,k}^y]_{\gamma} = \sup_{x',x''\in \bb P(V): x' \neq x''}\frac{|\chi_{n,k}^y(x') - \chi_{n,k}^y(x'')|}{d(x',x'')^{\gamma}},
\end{align*}
where $d$ is the angular distance on $\bb P(V)$ defined by \eqref{Angular-distance}. 
Asume that $x'=\bb R v'\in \bb P(V)$ and $x''=\bb R v''\in \bb P(V)$ are such that $\|v'\|=\|v''\|=1$. 
We note that 
\begin{align} \label{angular dist-bound001}
\| v'-v''\| \leq \sqrt{2}d(x',x''). 
\end{align}
For short denote $B_k=((k-1)a_n,ka_n]$. Remark that the function $h_{n,k}$ is increasing on $B_k$ and 
decreasing on $B_{k+1}$. 
Set for brevity
$t'=-\log \delta(y,x')$ and $t''=-\log \delta(y,x'')$. 
First we consider the case when $t'$ and $t''$ are such that $t',t''\in B_{k}$. 
Then, using \eqref{Def-chi-nk} and \eqref{h_kLip001}, for any $\gamma \in (0,1]$,  we have 
\begin{align} \label{bounddCHI-001}
&|\chi_{n,k}^y(x')- \chi_{n,k}^y(x'')| = |h_{n,k}(t')- h_{n,k}(t'')|^{1-\gamma}  |h_{n,k}(t')- h_{n,k}(t'')|^{\gamma} \notag\\ 
&\leq |h_{n,k}(t')- h_{n,k}(t'')|^{\gamma} \leq \frac{|t'-t''|^{\gamma}}{a_{n}^{\gamma}} 
= \frac{1}{a_{n}^{\gamma}} |\log u'-\log u''|^{\gamma},
\end{align}
where we set for brevity $u'=\delta(y,x')$ and $u''= \delta(y,x'')$. 
Since $u'=e^{-t'}$, $u''=e^{-t''}$ and $t,t'\in B_{k}$, we have
$u''\geq e^{-k a_n}$ and $| u' -u''| \leq e^{-(k-1) a_n }-e^{-ka_n}$. 
Therefore, 
\begin{align*} 
\Big| \frac{u'}{u''}-1 \Big|= \Big| \frac{u'-u''}{u''} \Big| 
\leq \frac{ e^{-(k-1)a_n} -e^{-ka_n} }{e^{-ka_n}}
=e^{a_n}-1\leq a_n e^{a_n} 
\leq \frac{1}{2}, 
\end{align*}
which, together with the inequality $|\log(1+a)| \leq  2 |a|$ for any $|a| \leq \frac{1}{2}$, implies
\begin{align} \label{h_kLip002}
|\log u' - \log u''|= \left| \log\left(1+ \frac{u'}{u''}-1 \right) \right| \leq 2 \frac{|u'-u''|}{u''}.
\end{align}
Since $u''\geq e^{-k a_n}$, using the fact that $\|v'\|=\|v''\|=1$ and \eqref{angular dist-bound001}, we get
\begin{align} \label{Lipschitz-u}
\frac{|u'-u''|}{u''}
&\leq  e^{k a_n}  |\delta(y,x')- \delta(y,x'')| = e^{k a_n} \frac{|f(v') - f(v'') |}{\|f\|} \notag\\
& =  e^{k a_n} \frac{|f (v'-v'')|}{\| f \|} \leq \sqrt{2} e^{ ka_n} d(x',x'').
\end{align}
Therefore, from \eqref{bounddCHI-001}, \eqref{h_kLip002} and \eqref{Lipschitz-u}, it follows that for $\gamma \in (0,1]$, 
\begin{align} \label{bounddCHI-010}
|\chi_{n,k}^y(x')- \chi_{n,k}^y(x'')|  \leq  3 \frac{e^{\gamma k a_n}}{a_{n}^\gamma}d(x',x'')^\gamma.
\end{align}
The case $t',t''\in B_k$ is treated in the same way.


To conclude the proof we shall consider the case when
$t'=-\log\delta(y,x')\in B_{k-1}$ and $t''=-\log\delta(y,x'')\in B_{k}$; 
the other cases can be handled in the same way.
We shall reduce this case to the previous ones.
Let $x^*\in \bb P(V)$ be the point on the geodesic line $[x',x'']$ on $\bb P(V)$ 
such that $d(x',x'')=d(x',x^*)+d(x^*,x'')$ and $t^*=-\log\delta(y,x^*)=ka_n.$
Then 
\begin{align} \label{bounddCHI-011}
|\chi_{n,k}^y(x')- \chi_{n,k}^y(x'')| 
&\leq |\chi_{n,k}^y(x')- \chi_{n,k}^y(x^*)| + |\chi_{n,k}^y(x'')- \chi_{n,k}^y(x^*)| \notag\\
&\leq 
3 \frac{e^{\gamma k a_n}}{ a_{n}^\gamma }d(x',x^*)^{\gamma} 
+ 3 \frac{e^{\gamma k a_n}}{ a_{n}^\gamma }d(x'',x^*)^{\gamma} \notag\\
&\leq 
6 \frac{e^{\gamma k a_n}}{ a_{n}^\gamma }d(x',x'')^{\gamma}. 
\end{align}
From \eqref{bounddCHI-010} and \eqref{bounddCHI-011} we conclude that 
$[\chi_{n,k}^y]_{\gamma}\leq  6 \frac{e^{ \gamma k a_n }}{a_{n}^{\gamma}}$, which
shows \eqref{Holder prop ohCHI_k-001}.
\end{proof}

We need the following bounds.
Let $M_n=\floor{A\log^2 n}$, where $A>0$ is a constant and $n$ is large enough. 
It is convenient to denote 
\begin{align} \label{varphi-nk-001}
\varphi_{n,k}^y=\varphi \chi_{n,k}^y\quad\mbox{for}\quad  0 \leq k \leq M_n-1,\quad 
\varphi_{n,M_n}^y=\varphi \overline \chi_{n,M_n}^y.
\end{align}

\begin{lemma} \label{new bound for delta020} 
Assume \ref{Ch7Condi-Moment} and \ref{Ch7Condi-IP}. 
Then there exist constants $s_0 >0$ and $c >0$ such that for any $s \in (-s_0, s_0)$, 
$y\in \bb P(V^*)$ and any non-negative bounded measurable function $\varphi$ on $\bb P(V)$, 
\begin{align*} 
 \sum_{k=0}^{M_n} (k+1) a_n \pi_s(\varphi_{n,k}^y) \leq  -d_{s,\varphi}(y) + 2 a_n \pi_s(\varphi)
\end{align*}
and
\begin{align*} 
\sum_{k=0}^{M_n} (k-1) a_n \pi_s(\varphi_{n,k}^y) \geq -d_{s,\varphi}(y) - 2 a_n \pi_s(\varphi) -  c \frac{ \|\varphi \|_{\infty}}{n^2}.
\end{align*}
\end{lemma}
\begin{proof}
Recall that $d_{\varphi}(y)$ is defined in \eqref{drift-d001}.
Using \eqref{Unit-partition001} we deduce that
\begin{align*} 
-d_{s,\varphi}(y) 
& = - \sum_{k=0}^{M_n} \int_{\bb P (V)} \varphi_{n,k}^y(x) \log\delta(y,x) \pi_s(dx) \notag \\
&\geq \sum_{k=0}^{M_n} (k-1) a_n \pi_s(\varphi_{n,k}^y)   \notag\\
& = \sum_{k=0}^{M_n} (k+1) a_n \pi_s(\varphi_{n,k}^y) - 2 a_n \pi_s(\varphi),
\end{align*}
which proves the first assertion of the lemma.

Using the Markov inequality and the exponential H\"older regularity of the invariant measure $\pi_s$ (Lemma \ref{Lem_Regu_pi_s00}), 
 we get that there exists a small $\eta >0$ such that 
\begin{align*} 
&- \int_{\bb P(V)} \varphi_{n,M_n}^y(x)\log \delta(y,x) \pi_s(dx)  \notag\\
& \leq c \|\varphi \|_{\infty} \int_{\bb P(V)} \frac{e^{-\eta A\log n}}{\delta(y,x)^{\eta}} \delta(y,x)^{-\eta} \pi_s(dx) \notag \\
& = e^{-\eta A \log n} \|\varphi \|_{\infty} \int_{\bb P(V)}  \delta(y,x)^{-2\eta} \pi_s(dx) 
\leq  c \frac{ \|\varphi \|_{\infty}}{ n^2 },
\end{align*}
where in the last inequality we choose $A >0$ to be sufficiently large so that $\eta A \geq 2$. 
Therefore, 
\begin{align*} 
-d_{s,\varphi}(y) 
&= - \sum_{k=0}^{M_n} \int_{\bb P (V)} \varphi_{n,k}^y(x) \log\delta(y,x) \pi_s(dx) \notag \\
&\leq \sum_{k=0}^{M_n-1} (k+1) a_n \pi_s(\varphi_{n,k}^y) +  c \frac{ \|\varphi \|_{\infty}}{ n^2 } \notag  \\
& \leq \sum_{k=0}^{M_n-1} (k-1) a_n \pi_s(\varphi_{n,k}^y) + 2 a_n \pi_s(\varphi) +  c \frac{ \|\varphi \|_{\infty}}{ n^2 }  \notag\\
& \leq \sum_{k=0}^{M_n} (k-1) a_n \pi_s(\varphi_{n,k}^y) + 2 a_n \pi_s(\varphi) +  c \frac{ \|\varphi \|_{\infty}}{ n^2 }. 
\end{align*}
This proves the second assertion of the lemma.
\end{proof}

%
%
%
%

\begin{proof}[Proof of Theorem \ref{Thm-Edge-Expan-Coeff001extended}]
Without loss of generality, we assume that the target function $\varphi$ is non-negative. 
With the notation in \eqref{varphi-nk-001}, we have
\begin{align}\label{Initial decompos-001aa}
I_n(t) &: =\bb{E}_{\mathbb{Q}_s^x} \left[ \varphi(G_n \!\cdot\! x) 
\mathds{1}_{ \big\{ \frac{\log |\langle f, G_n v \rangle| - n\lambda_1 }{ \sigma_s \sqrt{n} } \leq t  \big\}  } \right] \notag \\ 
& =  \sum_{k=0}^{M_n} 
\bb{E}_{\mathbb{Q}_s^x} \left[ \varphi_{n,k}^y (G_n \!\cdot\! x) 
\mathds{1}_{ \big\{ \frac{\log |\langle f, G_n v \rangle| - n\lambda_1 }{ \sigma_s \sqrt{n} } \leq t  \big\}  }\right] 
=: \sum_{k=0}^{M_n}  F_{n,k}(t).
\end{align}
For $0\leq k\leq M_n - 1$, using \eqref{Ch7_Intro_Decom0a} 
and the fact that $-\log \delta(y, x) \leq (k+1)a_n$ when $x \in \supp \varphi_{n,k}^y$, we get
\begin{align} \label{boundFfbar-001}
F_{n,k}(t)
\leq
\bb{E}_{\mathbb{Q}_s^x} \left[ \varphi_{n,k}^y (G_n \!\cdot\! x) 
\mathds{1}_{ \big\{ \frac{\sigma(G_n, x) - n\lambda_1 }{ \sigma_s \sqrt{n} } \leq t + \frac{(k+1) a_n}{\sigma_s \sqrt{n}} \big\}  }   \right] 
=:  H_{n,k}(t).
\end{align}
For 
$k=M_n$, 
we have
\begin{align} \label{boundFfbar-002}
F_{n,M_n}(t)
 &\leq 
 \bb{E}_{\mathbb{Q}_s^x} \left[ \varphi_{n,M_n}^y (G_n \!\cdot\! x) 
\mathds{1}_{ \big\{ \frac{\log \sigma(G_n,x) - n\lambda_1 }{ \sigma_s \sqrt{n} } \leq t +\frac{(M_n+1) a_n}{\sigma_s \sqrt{n}}  \big\}  }\right] 
\notag \\
&\quad +\bb{E}_{\mathbb{Q}_s^x} \left[ \varphi_{n,M_n}^y (G_n \!\cdot\! x) 
\mathds{1}_{ \big\{ -\log \delta(y, G_n \cdot x)  \geq (M_n +1)a_n  \big\}  }\right] \notag \\
&=: H_{n,M_n}(t) +  W_{n}(t), 
\end{align}
where, by Lemma \ref{Lem_Regu_pi_s}, choosing $A$ large enough,
\begin{align} \label{W_n001}
W_{n}(t) &\leq  \|\varphi \|_{\infty}  \mathbb{Q}_s^x ( -\log \delta(y, G_n \cdot x)  \geq A \log n )  \notag\\
&\leq  \|\varphi \|_{\infty} \frac{c_0}{n^{c_1 A}}  
\leq  \|\varphi \|_{\infty} \frac{c_0}{n^{2}}.
\end{align}
Now we deal with $H_{n,k}(t)$ for $0 \leq k \leq M_n$.  
Denote for short $t_{n,k} = t +\frac{(k+1) a_n}{\sigma_s \sqrt{n}}$. 
Applying the Edgeworth expansion (Theorem \ref{Thm-Edge-Expan}) we obtain
that, uniformly in $s \in (-s_0, s_0)$, $x \in \bb P(V)$, 
$t \in \bb R$, $0 \leq k \leq M_n$ and 
$\varphi \in \mathscr{B}_{\gamma}$, as $n \to \infty$,
\begin{align*}
  H_{n,k}(t)
 &   =   \pi_s(\varphi_{n,k}^y) \Big[  \Phi(t_{n,k}) + \frac{\Lambda'''(s)}{ 6 \sigma_s^3 \sqrt{n}} (1-t_{n,k}^2) \phi(t_{n,k}) \Big]
 \notag\\ 
 & \quad - \frac{ b_{s, \varphi_{n,k}^y}(x) }{ \sigma_s \sqrt{n} } \phi(t_{n,k})    
     +  \pi_s(\varphi_{n,k}^y)  o \Big( \frac{ 1 }{\sqrt{n}} \Big)  
  +  \lVert \varphi_{n,k}^y \rVert_{\gamma} O \Big( \frac{ 1 }{n} \Big).   
\end{align*}
By the Taylor expansion we have, 
uniformly in $s \in (-s_0, s_0)$, $x \in \bb P(V)$,  $t \in \bb R$ and $0 \leq k \leq M_n$, 
\begin{align*} 
\Phi(t_{n,k}) 
& =\Phi(t) + \phi(t) \frac{(k+1) a_n}{\sigma_s \sqrt{n}} + O\Big(\frac{\log^2 n}{n} \Big)
\end{align*}
and 
\begin{align*} 
(1-t_{n,k}^2) \phi(t_{n,k}) 
&= (1-t^2) \phi(t) + O\left( \frac{\log n }{\sqrt{n}} \right).
\end{align*}
Moreover, using Lemma \ref{Lem-Bs},
\begin{align*} 
\frac{ b_{s, \varphi_{n,k}^y}(x) }{ \sigma_s \sqrt{n} } \phi(t_{n,k}) 
= \frac{ b_{s, \varphi_{n,k}^y}(x) }{ \sigma_s \sqrt{n} } \phi(t) 
+ \| \varphi_{n,k}^y \|_{\gamma} O \Big(\frac{\log n}{n}\Big). 
\end{align*}
Using these expansions and \eqref{boundFfbar-001}, \eqref{boundFfbar-002} and \eqref{W_n001}, we get
that there exists a sequence $(\beta_n)_{n \geq 1}$ of positive numbers
satisfying $\beta_n \to 0$ as $n \to \infty$, such that for any $0 \leq k \leq M_n$, 
\begin{align}\label{F_nkbound001}
   F_{n,k}(t)
 &   \leq   \pi_s(\varphi_{n,k}^y) \Big[  \Phi(t) + \frac{\Lambda'''(s)}{ 6 \sigma_s^3 \sqrt{n}} (1-t^2) \phi(t) \Big]\notag\\
  &\quad - \frac{ b_{s, \varphi_{n,k}^y}(x) }{ \sigma_s \sqrt{n} } \phi(t) 
  + \frac{\phi(t)}{\sigma_s \sqrt{n}} \pi_s(\varphi_{n,k}^y)(k+1) a_n   \notag\\
  & \quad  +  \pi_s(\varphi_{n,k}^y)  \frac{ \beta_n }{\sqrt{n}}  
  +   \lVert \varphi_{n,k}^y \rVert_{\gamma}  \frac{ c \log n }{n}.   
\end{align}
By Lemma \ref{lemmaHolder property001}, it holds that for any $\gamma \in (0, 1]$ and $0 \leq k \leq M_n$, 
\begin{align}\label{HolderNorm-varphik}
\lVert \varphi_{n,k}^y \rVert_{\gamma} 
\leq  c \lVert \varphi \rVert_{\infty} n^{\gamma A} \log^{\gamma} n  +  \lVert \varphi \rVert_{\gamma}. 
\end{align}
Summing up over $k$ in \eqref{F_nkbound001}, using \eqref{HolderNorm-varphik} 
and taking $\gamma >0$ to be sufficiently small such that $\gamma A < \ee/2$, 
we obtain
\begin{align}\label{Edgeworth-Fn}
 I_n(t) = \sum_{k= 0}^{M_n}  F_{n,k} & \leq    \pi_s(\varphi) 
   \Big[  \Phi(t) + \frac{\Lambda'''(s)}{ 6 \sigma_s^3 \sqrt{n}} (1-t^2) \phi(t) \Big]  \notag\\
& \quad -  \frac{ b_{s,\varphi}(x) }{ \sigma_s \sqrt{n} } \phi(t)
+ \frac{\phi(t)}{\sigma_s \sqrt{n}} \sum_{k = 0}^{M_n} \pi_s(\varphi_{n,k}^y)(k+1)  a_n \notag\\
&  \quad + \pi_s(\varphi)   \frac{ \beta_n }{\sqrt{n}}  +  
    \lVert \varphi \rVert_{\gamma}  \frac{ c }{n^{1 - \ee}}. 
\end{align}
Combining \eqref{Edgeworth-Fn} 
and Lemma \ref{new bound for delta020}, 
together with the fact that $a_n \to 0$ as $n \to \infty$, 
yields the desired upper bound.

The lower bound is established in the same way. 
Instead of \eqref{F_nkbound001} we use the following bound, which is obtained
using \eqref{Ch7_Intro_Decom0a} 
and the fact that $-\log \delta(y, x) \geq (k-1)a_n$ for $x \in \supp \varphi_{n,k}^y$ and $0\leq k\leq M_n$, 
\begin{align} \label{boundFfbar-003}
F_{n,k}(t)
\geq
\bb{E}_{\mathbb{Q}_s^x} \left[ \varphi_{n,k}^y (G_n \!\cdot\! x) 
\mathds{1}_{ \big\{ \frac{\sigma(G_n, x) - n\lambda_1 }{ \sigma_s \sqrt{n} } \leq t +\frac{(k-1) a_n}{\sigma_s\sqrt{n}} \big\}  }   \right]. 
\end{align}
Proceeding in the same way as in the proof of the upper bound, 
 using \eqref{boundFfbar-003}  
instead of \eqref{boundFfbar-001} and \eqref{boundFfbar-002},  we get
\begin{align*} 
 I_n(t) = \sum_{k= 0}^{M_n}  F_{n,k}  & \geq  \pi_s(\varphi) 
   \Big[  \Phi(t) + \frac{\Lambda'''(s)}{ 6 \sigma_s^3 \sqrt{n}} (1-t^2) \phi(t) \Big]  \notag\\
&\quad -    \frac{ b_{s, \varphi}(x) }{ \sigma_s \sqrt{n} } \phi(t)
+ \frac{\phi(t)}{\sigma_s \sqrt{n}} \sum_{k = 0}^{M_n}  \pi_s(\varphi_{n,k}^y)(k-1) a_n \notag\\
&  \quad + \pi_s(\varphi)  o \Big( \frac{ 1 }{\sqrt{n}} \Big)  +  
    \lVert \varphi \rVert_{\gamma} O \Big( \frac{1 }{n^{1 - \ee}} \Big). 
\end{align*}
The lower bound is obtained using again Lemma \ref{new bound for delta020}.
\end{proof}

\section{Proof of the subexponential regularity and applications}

In this section we first establish Theorem \ref{Thm-Regularity-Subex} on the regularity of the invariant measure $\nu$
under a sub-exponential moment condition.
Then we shall apply Theorem \ref{Thm-Regularity-Subex} to prove 
Theorem \ref{Thm_BerryEsseen} on the Berry-Esseen type bound for the coefficients $\langle f, G_n v \rangle$,
and Theorem \ref{Thm-MDP} on the moderate deviation principle for $\langle f, G_n v \rangle$ also under a sub-exponential moment condition.

\subsection{Proof of Theorem \ref{Thm-Regularity-Subex}}

The following moderate deviation principle is proved in \cite{CDM17} under a subexponential moment condition.  

\begin{lemma}[\cite{CDM17}]  \label{Lem_MDP_VecrorNorm}
Assume \ref{Ch7Condi-IP}. 
Assume also that $\int_{ \textup{GL}(V) } \log^2N(g)  \,  \mu(dg) < \infty$ and 
\begin{align}\label{Mom-Condi-MD}
\limsup_{n \to \infty} \frac{n}{b_n^2}  \log  \Big[ n \mu \big( \log N(g) > b_n \big) \Big]  = - \infty, 
\end{align}
where $(b_n)_{n \geq 1}$ is a sequence of positive numbers satisfying $\frac{b_n}{n} \to 0$ 
and $\frac{b_n}{\sqrt{n}} \to \infty$ as $n \to \infty$. 
Then, for any Borel set $B \subset \bb R$, 
\begin{align*}
- \inf_{t \in B^{\circ}} \frac{t^2}{2 \sigma^2}  
& \leq  \liminf_{n \to \infty} \frac{n}{b_n^2} 
   \log \inf_{x \in \bb P(V)} \bb P \left( \frac{\sigma(G_n, x) - n \lambda_1}{b_n}  \in B \right)  \notag\\
& \leq  \limsup_{n \to \infty} \frac{n}{b_n^2} 
   \log \sup_{x \in \bb P(V)} \bb P \left( \frac{\sigma(G_n, x) - n \lambda_1}{b_n}  \in B \right) 
 \leq  - \inf_{t \in \bar{B} } \frac{t^2}{2 \sigma^2},
\end{align*}
where $B^{\circ}$ and $\bar{B}$ are respectively the interior and the closure of $B$. 
\end{lemma}

The following result is an easy consequence of Lemma \ref{Lem_MDP_VecrorNorm}. 

\begin{lemma}\label{Lem_Cor_MDBounds}
Assume the conditions of Lemma \ref{Lem_MDP_VecrorNorm}.
Then, there exist constants $c, c' >0$ and $n_0 \in \bb N$ such that for any
$n \geq n_0$ and  $v \in V \setminus \{0\}$, 
\begin{align}
& e^{- c' \frac{b_n^2}{n} }
  \leq  \bb P \left( \frac{ \|G_n v\|}{\|v\|}  \geq  e^{ \lambda_1  n  +  b_n}   \right)  
  \leq  e^{- c \frac{b_n^2}{n} },      \label{MDPCor001}  \\
& e^{- c' \frac{b_n^2}{n} }
\leq  \bb P \left( \frac{\|G_n v\|}{\|v\|}  \leq  e^{ \lambda_1  n  -  b_n}   \right)  
\leq  e^{- c \frac{b_n^2}{n} },    \label{MDPCor002}  \\
& e^{- c' \frac{b_n^2}{n} }
  \leq  \bb P \left( \|G_n  \|  \geq  e^{ \lambda_1  n  +  b_n}   \right)  
  \leq  e^{- c \frac{b_n^2}{n} },      \label{MDPCor003}  \\
& e^{- c' \frac{b_n^2}{n} }
\leq  \bb P \left( \| G_n \|  \leq  e^{ \lambda_1  n  -  b_n}   \right)  
\leq  e^{- c \frac{b_n^2}{n} }.     \label{MDPCor004}
\end{align}
\end{lemma}

\begin{proof}
Inequalities \eqref{MDPCor001} and \eqref{MDPCor002} 
are direct consequences of Lemma \ref{Lem_MDP_VecrorNorm}. 
By \eqref{MDPCor001} and the fact that $\|G_n  \| \geq \frac{\|G_n v\|}{\|v\|}$,
the first inequality in \eqref{MDPCor003} follows. 
To show the second inequality in \eqref{MDPCor003},  
since all matrix norms on $V$ are equivalent, 
we get 
\begin{align*}
\bb P \left( \|G_n  \|  \geq  e^{n \lambda_1 +  b_n}   \right)
&  \leq  \bb P \left( \max_{1 \leq i \leq d}  \| G_n e_i \|  \geq  e^{n \lambda_1 +  b_n}   \right)  \notag\\
&  \leq  \sum_{i=1}^d  \bb P \left(  \| G_n e_i \|  \geq  e^{n \lambda_1 +  b_n}   \right)
\leq  e^{- c \frac{b_n^2}{n} },
\end{align*}
where the last inequality holds due to \eqref{MDPCor001}.
This concludes the proof of \eqref{MDPCor003}.
Using \eqref{MDPCor002}, 
the proof of \eqref{MDPCor004} can be carried out in the same way as that of \eqref{MDPCor003}. 
\end{proof}

The following purely algebraic result is due to Chevalley \cite{Che51}; see also Bougerol and Lacroix \cite[page 125, Lemma 5.5]{BL85}. 

\begin{lemma}[\cite{Che51, BL85}] \label{LemChevalley}
Let $G$ be an irreducible subgroup of $\textup{GL}(V)$. 
Then, for any integer $1 \leq p \leq d$,
there exists a direct-sum decomposition of the $p$-th exterior power: 
$\wedge^p (V)  = V_1  \oplus \ldots  \oplus  V_k$
such that $(\wedge^p g) V_j = V_j$ for any $g \in G$ and $1 \leq j \leq k$. 
Moreover, $\wedge^p (G) := \{ \wedge^p g: g \in G \} $ is irreducible on each subspace $V_j, j=1, \cdots, k$. 
%
\end{lemma}

Using Lemma \ref{LemChevalley} and the strategies from \cite{BL85, XGL21}, 
we extend \eqref{MDPCor003} and \eqref{MDPCor004} to the exterior product $\wedge^p  G_n$.

\begin{lemma}
Assume the conditions of Lemma \ref{Lem_MDP_VecrorNorm}.
Let $1 \leq p \leq d$ be an integer. 
Then there exist constants $c, c' >0$ and $n_0 \in \bb N$ such that
for any $n \geq n_0$, 
\begin{align}
& e^{- c' \frac{b_n^2}{n} }
  \leq  \bb P \left( \| \wedge^p  G_n  \|  \geq  e^{ n \sum_{i = 1}^p \lambda_i  +  b_n}   \right)  
  \leq  e^{- c \frac{b_n^2}{n} },      \label{MDPCor005}  \\
& e^{- c' \frac{b_n^2}{n} }
\leq  \bb P \left( \| \wedge^p  G_n \|  \leq  e^{ n \sum_{i = 1}^p \lambda_i  -  b_n}   \right)  
\leq  e^{- c \frac{b_n^2}{n} }.       \label{MDPCor006}  
\end{align}
\end{lemma}

\begin{proof}

We first prove \eqref{MDPCor005}. 
Since the Lyapunov exponents $(\lambda_p)_{1 \leq p \leq d}$ of $\mu$ are given by
$\lambda_1 > \lambda_2 \geq \ldots \geq \lambda_d,$
the two largest Lyapunov exponents of $\wedge^p G_n$
satisfy $\sum_{i = 1}^p \lambda_i >  \sum_{i = 2}^{p+1} \lambda_i$
(see \cite{BL85}), where we use the convention that $\lambda_{d+1} = \lambda_d$. 
Without loss of generality, we can assume that $\sum_{i = 1}^p \lambda_i = 0$;
otherwise we replace $\wedge^p  G_n$ by $e^{- \sum_{i = 1}^p \lambda_i } \wedge^p  G_n$. 

Since the action of $\wedge^p G_n$ on $\wedge^p V$ is in general not irreducible, 
we need to consider a decomposition of $\wedge^p V$. 
Applying Lemma \ref{LemChevalley} to $G = G_{\mu}$ 
(the smallest closed subgroup of $\textup{GL}(V)$ generated by the support of $\mu$), we get
the following direct-sum decomposition of the $p$-th exterior power $\wedge^p V$: 
$\wedge^p V = V_1 \oplus V_2 \oplus \ldots  \oplus  V_k$,  
where $V_j$ are subspaces of $\wedge^p V$
such that $(\wedge^p g) V_j = V_j$ for any $g \in G_{\mu}$ and $1 \leq j \leq k$.
Moreover, $\wedge^p (G_{\mu}) := \{ \wedge^p g: g \in G_{\mu} \}$  is irreducible on each subspace $V_j$. 
Since the set of all Lyapunov exponents of $\wedge^p G_n$ on the space $\wedge^p V$
coincides with the union of all the Lyapunov
exponents of $(\wedge^p G_n)$ restricted to each subspace $V_j$, $1 \leq j \leq k$, 
we can choose $V_1$ in such a way that
the restrictions of $\wedge^p G_n$ to $V_1$ and $V_1' : = V_2 \oplus \ldots  \oplus  V_k$, 
denoted  respectively by $G_n'$ and $G_n''$, 
  satisfy: 
\begin{align}
\| \wedge^p G_n \|   =  \max \{ \|G_n'\|, \|G_n''\| \},   \label{Pf_Ch6_GAB}
\end{align}
and a.s., 
\begin{align}
&  \lim_{n \to \infty} \frac{1}{n} \log \|G_n'\|  = \sum_{i = 1}^p \lambda_i = 0  \quad  \mbox{and}  \quad
  \lim_{n \to \infty} \frac{1}{n} \log \|G_n''\|    =  \sum_{i = 2}^{p+1} \lambda_i  < 0.   \label{Pf_Ch6_Lya_Bn}   
\end{align}
Here, $G_n'$ and $G_n''$ are products of i.i.d. random variables of the form 
$G_n'= g_n' \cdots g_1'$ and  $G_n''= g_n'' \cdots g_1''$.  
We denote by $\mu_1$ the law of the random variable $g_1$, by $d_1$ the dimension of the vector space $V_1$, 
and by $\Gamma_{\mu_1}$ the smallest closed subsemigroup of $\textup{GL}(V_1)$ generated by the support of $\mu_1$. 
Then, following the analogous argument used in the proof of the central limit theorem for $\|G_n\|$ 
(see \cite[Theorem V.5.4]{BL85}),
one can verify 
that the semigroup $\Gamma_{\mu_1}$ is strongly irreducible and proximal on $V_1$. 
Therefore, 
we can apply the moderate deviation bounds \eqref{MDPCor003} and \eqref{MDPCor004}
 (with $G_n$ replaced by $G_n'$) to get 
\begin{align}
& e^{- c' \frac{b_n^2}{n} }
  \leq  \bb P \left( \| G_n'  \|  \geq  e^{ b_n}   \right)  
  \leq  e^{- c \frac{b_n^2}{n} },   \quad  
 e^{- c' \frac{b_n^2}{n} }
\leq  \bb P \left( \| G_n' \|  \leq  e^{ -  b_n}   \right)  
\leq  e^{- c \frac{b_n^2}{n} }.       \label{MDPCor005_Inte} 
\end{align}
From \eqref{Pf_Ch6_GAB} and \eqref{MDPCor005_Inte}, the lower bound of \eqref{MDPCor005} easily follows: 
\begin{align}\label{Pf_LowBound_Gn}
e^{- c' \frac{b_n^2}{n} }
  \leq  \bb P \left( \| \wedge^p  G_n  \|  \geq  e^{ b_n}   \right).  
\end{align}

Now we prove the upper bound of \eqref{MDPCor005}. 
Since the first Lyapunov exponent of the sequence $(G_n'')_{n \geq 1}$ is strictly less than $0$
(see \eqref{Pf_Ch6_Lya_Bn}), 
we have $$\frac{1}{m} \bb{E} (\log \| G_m'' \|) < 0$$ for sufficiently large integer $m \geq 1$. 
If we write $n = km + r$ with $k \geq 1$ and $0 \leq r < m$, 
then we have the identity 
$$G_n'' =  [G_n'' (G_{km}'')^{-1} ]  \,  [G_{km}''  (G_{(k-1)m}'')^{-1}]  \ldots  [G_{2m}'' (G_m'')^{-1} ]  \,  G_m'',$$
and hence
\begin{align}\label{Ch6_Pf_Bn_aa}
\log \|G_n''\| \leq \log \| G_n'' (G_{km}'')^{-1} \| + \log \| G_{km}'' (G_{(k-1)m}'')^{-1} \|
+ \ldots + \log \|G_m''\|.  
\end{align}
For fixed integer $m \geq 1$, we denote $u_m : = - \bb{E} (\log \|G_m''\|) > 0$. 
Then,  
\begin{align}\label{Pf_MDP_lll}
& \bb P (\log \|G_n''\| \geq 0) 
  \leq \bb P \Big(\log \| G_n'' (G_{km}'')^{-1} \| \geq  k \frac{u_m}{2} \Big)  \nonumber\\
& \qquad  +  \bb P \Big( \log \| G_{km}'' (G_{(k-1)m}'')^{-1} \| 
     + \cdots + \log \|G_m''\| + k u_m \geq  k \frac{u_m}{2} \Big). 
\end{align}
For the first term, using \eqref{Pf_Ch6_GAB} and the inequality $\| \wedge^p g \| \leq \| g \|^p$ for any $g \in \Gamma_{\mu}$, 
we get
\begin{align}\label{Pf_MDP_bbb}
& \bb P \Big(\log \| G_n'' (G_{km}'')^{-1} \| \geq  k \frac{u_m}{2} \Big) 
  = \bb P \Big(\log \| G_r'' \| \geq  k \frac{u_m}{2} \Big)  \notag\\
  & \leq  \bb P \Big(\log \| \wedge^p G_r \| \geq  k \frac{u_m}{2} \Big)
   =   \bb P \Big( \sum_{i=1}^r \log \| \wedge^p  g_i \| \geq  k \frac{u_m}{2} \Big)  \notag\\
 & \leq  \sum_{i=1}^r \bb P \Big(  \log \| \wedge^p  g_1 \| \geq  k \frac{u_m}{2r} \Big) 
   \leq  \sum_{i=1}^r \bb P \Big(  \log \| g_1 \| \geq  k \frac{u_m}{2rp} \Big)  \notag\\
& \leq  \frac{c}{k} e^{- c \frac{b_k^2}{k} },
\end{align}
where in the last inequality we use condition \eqref{Mom-Condi-MD}. 
The second term on the right hand side of \eqref{Pf_MDP_lll} is dominated by $e^{- c b_k^2/k }$, 
by using the moderate deviation bounds for sums of i.i.d. real-valued random variables. 
This, together with \eqref{Pf_MDP_bbb}, \eqref{Pf_MDP_lll} and the fact that $k \geq n/(m+1)$, yields that
\begin{align}\label{Pf_MDP_Bn_kk}
\bb{P} ( \log \| G_n'' \| \geq  0 ) \leq  e^{- c \frac{b_n^2}{n} }. 
\end{align}
From \eqref{Pf_Ch6_GAB}, \eqref{MDPCor005_Inte} and \eqref{Pf_MDP_Bn_kk}, 
we derive that
\begin{align*}
 \bb P \left( \| \wedge^p  G_n  \|  \geq  e^{b_n}   \right)  
\leq  \bb P \left( \| G_n'  \|  \geq  e^{ b_n}   \right)  +  \bb P \left( \| G_n''  \|  \geq  e^{ b_n}   \right)  
\leq  2 e^{- c \frac{b_n^2}{n} }. 
\end{align*}
Combining this with \eqref{Pf_LowBound_Gn}, we get \eqref{MDPCor005}. 

The proof of \eqref{MDPCor006} can be carried out in a similar way. 
\end{proof}

Our next result is an analog of Proposition \ref{Prop-rates-delta01} under a subexponential  moment assumption.  

\begin{proposition}\label{Prop-Regularity}
Assume the conditions of Lemma \ref{Lem_MDP_VecrorNorm}.
Then, 
there exist constants $c>0$ and $n_0 \in \bb N$ such that 
for any $n \geq n_0$,  $x \in \bb P(V)$ and $y \in \bb P(V^*)$, 
\begin{align}
& \bb P \left( d (G_n \!\cdot\! x, x_{G_n}^M) \geq e^{- (\lambda_1 - \lambda_2 ) n + b_n } \right) 
   \leq  e^{-c \frac{b_n^2}{n}},      \label{d_Gn_xM_MDP} \\
&  \bb P \left( \delta(x_{G_n}^M, y) \leq  e^{  -  b_n }  \right) 
    \leq  e^{-c \frac{b_n^2}{n}},     \label{xGn_y_MDP}  \\
&  \bb P \left( \delta (G_n \!\cdot\! x, y) \leq  e^{  -  b_n } \right) 
    \leq  e^{-c \frac{b_n^2}{n}}.    \label{Gnx_y_MDP}
\end{align}
\end{proposition}

\begin{proof}
We first prove \eqref{d_Gn_xM_MDP}. 
By \eqref{Inequ_First001}, we have 
\begin{align*}
\bb P \left( d (G_n \!\cdot\! x, x_{G_n}^M) \geq e^{- (\lambda_1 - \lambda_2 ) n  +  b_n }  \right)
\leq   \bb P \left( \frac{\| \wedge^2 G_n \| \|v\|}{\|G_n\|  \|G_n v\|} \geq  e^{- (\lambda_1 - \lambda_2 ) n  + b_n } \right) =: I_n.  
\end{align*}
Using \eqref{MDPCor002} gives
\begin{align*}
I_n & \leq  \bb P \left( \frac{\| \wedge^2 G_n \| \|v\|}{\|G_n\|  \|G_n v\|} \geq  e^{- (\lambda_1 - \lambda_2 ) n + b_n },
    \frac{\|G_n v\|}{\|v\|}  >  e^{ \lambda_1  n -  \frac{b_n}{3} }  \right)  +  e^{- c \frac{b_n^2}{n} }   \notag\\
& \leq  \bb P \left( \frac{\| \wedge^2 G_n \| }{\|G_n\| } \geq  e^{ \lambda_2  n + \frac{2}{3} b_n } \right)  
     +  e^{- c \frac{b_n^2}{n} }. 
\end{align*}
In the same way, from \eqref{MDPCor004} and \eqref{MDPCor005}  it follows that 
\begin{align}\label{Upper-Gn-2}
I_n  
& \leq  \bb P \left( \frac{\| \wedge^2 G_n \| }{\|G_n\| } \geq  e^{ \lambda_2  n + \frac{2}{3} b_n },
    \|G_n \|  >  e^{ \lambda_1  n - \frac{1}{3} b_n}     \right)  
     +  2 e^{- c \frac{b_n^2}{n} }   \notag\\
& \leq   \bb P \left(  \| \wedge^2 G_n \|  \geq  e^{(\lambda_1 + \lambda_2 ) n + \frac{1}{3} b_n }    \right)  
     +  2 e^{- c \frac{b_n^2}{n} }   \notag\\
& \leq  3 e^{- c \frac{b_n^2}{n} }, 
\end{align}
which concludes the proof of \eqref{d_Gn_xM_MDP}. 

We next prove \eqref{xGn_y_MDP}.  
By Lemma \ref{Lem_delta_d} (2), we have that for any $y = \bb R f$ 
with  $f \in V^* \setminus \{0\}$,
\begin{align}\label{Pf_LLN_Ine001_MDP}
\bb P \left( \delta(x_{G_n}^M, y) \leq  e^{  -  b_n }   \right)
\leq  \bb P \left(  \frac{ \|G_n^* f\| }{\|G_n^*\|  \|f\|} - \frac{\| \wedge^2 G_n \|}{\|G_n\|^2}  \leq  e^{ -  b_n }  \right) =: J_n.  
\end{align}
Following the proof of \eqref{Upper-Gn-2}, one has
\begin{align*}
\bb P \left(  \frac{\| \wedge^2 G_n \|}{\|G_n\|^2}   \geq  e^{- (\lambda_1 - \lambda_2 ) n + b_n}  \right)  
\leq  e^{- c \frac{b_n^2}{n} }. 
\end{align*}
Using this together with \eqref{MDPCor002} and \eqref{MDPCor003} applied to the measure $\mu^*$
($\mu^*$ is the image of the measure $\mu$ by the map $g \mapsto g^*$, 
where $g^*$ is the adjoint automorphism of $g \in \textup{GL}(V)$ acting on the dual space $V^*$), we get
\begin{align*}
J_n 
& \leq  \bb P \left(  \frac{\|G_n^* f \|}{\|G_n^*\|  \|f\|} - \frac{\| \wedge^2 G_n \|}{\|G_n\|^2}  \leq  e^{ -  b_n },  
      \frac{\| \wedge^2 G_n \|}{\|G_n\|^2}  <  e^{- (\lambda_1 - \lambda_2 ) n + b_n}  \right)  +  e^{- c \frac{b_n^2}{n} }  \notag\\
& \leq  \bb P \left(  \frac{\|G_n^* f \|}{\|G_n^*\| \|f\|}  \leq  2 e^{ -  b_n }   \right)  +  e^{- c \frac{b_n^2}{n} }  \notag\\
& \leq  \bb P \left(  \frac{ \|G_n^* f\|}{\|G_n^*\|  \|f\|}  \leq  2 e^{ -  b_n },  \|G_n^*\| \leq  e^{\lambda_1 n + \frac{b_n}{2} }  \right)  
   +  2 e^{- c \frac{b_n^2}{n} }   \notag\\
&  \leq    \bb P \left(  \frac{\|G_n^* f\|}{ \|f\|}  \leq  2 e^{ \lambda_1 n -  \frac{b_n}{2} }  \right)  
   +  2 e^{- c \frac{b_n^2}{n} }   \notag\\
& \leq  3 e^{- c \frac{b_n^2}{n} }, 
\end{align*}
which ends the proof of \eqref{xGn_y_MDP}. 

We finally prove \eqref{Gnx_y_MDP}.  Since 
\begin{align*}
\delta (G_n \!\cdot\! x, y) \geq  \delta(x_{G_n}^M, y) - d (G_n \!\cdot\! x, x_{G_n}^M), 
\end{align*}    
using \eqref{d_Gn_xM_MDP}, \eqref{xGn_y_MDP} and the fact that $\lambda_1 > \lambda_2$, we get
\begin{align*}
& \bb P \left( \delta (G_n \!\cdot\! x, y) \leq  e^{  -  b_n }  \right)   \notag\\
& \leq   \bb P \left( \delta(x_{G_n}^M, y) - d (G_n \!\cdot\! x, x_{G_n}^M)  \leq  e^{  -  b_n }   \right)   \notag\\
& \leq  \bb P \left( \delta(x_{G_n}^M, y) - d (G_n \!\cdot\! x, x_{G_n}^M)  \leq   e^{  -  b_n },
       d (G_n \!\cdot\! x, x_{G_n}^M) <  e^{- (\lambda_1 - \lambda_2 ) n + b_n }   \right) 
    + e^{-c \frac{b_n^2}{n}}   \notag\\
& \leq  \bb P \left( \delta(x_{G_n}^M, y)   \leq  e^{  -  b_n }  +  e^{- (\lambda_1 - \lambda_2 ) n + b_n }  \right) 
    +  e^{-c \frac{b_n^2}{n}}  \notag\\
& \leq  e^{-c \frac{b_n^2}{n}}, 
\end{align*}
which concludes the proof of \eqref{Gnx_y_MDP}. 
\end{proof}

As a consequence of Proposition \ref{Prop-Regularity}, we get the following: 

\begin{proposition}\label{Prop-Regularity02}
Assume the conditions of Lemma \ref{Lem_MDP_VecrorNorm}. 
Then there exist constants $c >0$ and $k_0 \in \bb{N}$ such that for all $n \geq k \geq k_0$, $x \in \bb P(V)$
and $y \in \bb P(V^*)$, 
\begin{align}\label{Regularity01}
 \bb P \left( \delta (G_n \!\cdot\! x, y) \leq  e^{  -  b_k } \right) 
    \leq  e^{-c \frac{b_k^2}{k}}.   
\end{align}
Moreover,  there exist constants $c>0$ and $n_0 >0$ such that 
for any $n \geq n_0$ and $y \in \bb P(V^*)$, 
\begin{align}\label{Regularity02}
\nu \left( x \in \bb P(V):  \delta (x, y) \leq  e^{  -  b_n } \right)  \leq  e^{-c \frac{b_n^2}{n}}.  
\end{align}
\end{proposition}

\begin{proof}
By using \eqref{Gnx_y_MDP}, 
the proof of \eqref{Regularity01} is similar to that of Lemma \ref{Lem-ScaNorm_second}. 
Inequality \eqref{Regularity02} is a consequence of \eqref{Gnx_y_MDP} 
and the fact that $\nu$ is the unique invariant measure of the Markov chain $(G_n \!\cdot\! x)_{n \geq 0}$. 
\end{proof}

\begin{proof}[Proof of Theorem \ref{Thm-Regularity-Subex}]
As a particular case of \eqref{Regularity01}, by taking $b_k = k^{\beta}$ with $\beta \in (\frac{1}{2}, 1)$, we get 
that if \ref{Ch7Condi-IP} holds and $\mu(\log N(g) > u) \leq \exp \{- u^{\frac{2 \beta - 1}{\beta}} a(u) \}$ for any $u >0$
and for some function $a(u) > 0$ satisfying $a(u) \to \infty$ as $u \to \infty$,    
then there exist constants $c >0$ and $k_0 \in \bb{N}$ such that for all $n \geq k \geq k_0$, $x \in \bb P(V)$
and $y \in \bb P(V^*)$, 
\begin{align*}
 \bb P \left( \delta (G_n \!\cdot\! x, y) \leq  e^{  -  k^{\beta} } \right) 
    \leq  \exp \left\{ -c k^{2\beta - 1} \right\}.   
\end{align*}
Letting $\alpha = \frac{2 \beta - 1}{\beta}$, 
this implies that if \ref{Ch7Condi-IP} holds and $\mu(\log N(g) > u) \leq \exp \{- u^{ \alpha } l_u \}$ 
for some $\alpha \in (0, 1)$ and any $u >0$,
then there exist constants $c >0$ and $k_0 \in \bb{N}$ such that for all $n \geq k \geq k_0$, $x \in \bb P(V)$
and $y \in \bb P(V^*)$, 
\begin{align*}
 \bb P \left( \delta (G_n \!\cdot\! x, y) \leq  e^{  -  k } \right) 
    \leq  \exp \left\{ -c k^{\alpha} \right\}.   
\end{align*}
This proves \eqref{thmRegularity01}. 

Now we prove \eqref{Regu-Subexponen01}. 
Taking $k = n$ in \eqref{thmRegularity01}, 
we get that there exist constants $c_0 >0$ and $n_0 >0$ such that 
for any $n \geq n_0$ and $y \in \bb P(V^*)$, 
\begin{align}\label{RegularityMea6}
\nu \left( x \in \bb P(V):  \delta (x, y) \leq  e^{- n} \right)  \leq  e^{-c_0 n^{\alpha} }.  
\end{align}
Set $B_{n,y}:= \{x \in \bb P(V): e^{- (n+1)} \leq  \delta (x, y)  \leq e^{ - n } \}$ for $y \in \bb P(V^*)$ and $n \geq 1$. 
From \eqref{RegularityMea6} we deduce that for any $\eta \in (0, c_0)$, uniformly in $y \in \bb P(V^*)$, 
\begin{align}\label{Regu-Subex66}
& \int_{\bb P(V)} \exp \left(  \eta |\log \delta (x,y)|^{\alpha} \right)  \nu(dx)   \notag\\
& \leq  e^{\eta n_0^{\alpha}}  
  +  \sum_{n=n_0}^{\infty}   \int_{ B_{n,y} } \exp \left( \eta |\log \delta (x,y)|^{\alpha} \right)  \nu(dx)  \notag\\
& \leq  e^{\eta n_0^{\alpha}}  +  \sum_{n=n_0}^{\infty}  e^{ - (c_0 - \eta) n^{\alpha} }  <  \infty,
\end{align}
since $\alpha > 0$. 
This concludes the proof of \eqref{Regu-Subexponen01}. 

Using \eqref{Regu-Subexponen01} and the Markov inequality, one can easily get \eqref{Regu-Subexponen02}. 
\end{proof}

\subsection{Proof of Theorem \ref{Thm_BerryEsseen}}

To prove Theorem \ref{Thm_BerryEsseen}, we shall use Theorem \ref{Thm-Regularity-Subex} and the following Berry-Esseen bound
for the norm cocycle $\sigma (G_n, x)$. 
 

\begin{lemma}[\cite{CDMP21}]\label{Lem_BerryE_NormPoly}
Assume \ref{Ch7Condi-IP} and that $\int_{ \textup{GL}(V) } \log^4 N(g) \, \mu(dg) < \infty$. 
Then, there exists a constant $c > 0$ such that for any $n \geq 1$, $t \in \bb R$ and $x \in \bb P(V)$,  
\begin{align*}
\left|  \bb{P} \left(  \frac{\sigma(G_n, x) - n \lambda_1 }{ \sigma \sqrt{n} } \leq t     \right)
-  \Phi(t)  \right|  \leq  \frac{c}{\sqrt{n}}. 
\end{align*}
\end{lemma}

The same result has been obtained in \cite{Jir20} under the slightly stronger moment condition
 $\int_{ \textup{GL}(V) } \log^8 N(g)  \, \mu(dg) < \infty$.

\begin{proof}[Proof of Theorem \ref{Thm_BerryEsseen}]
We first prove \eqref{BerryEsseen_Coeffaa}. 
The lower bound is a direct consequence of Lemma \ref{Lem_BerryE_NormPoly}
and the fact that $\log |\langle f, G_n v \rangle| \leq \log |G_n v|$: 
there exists a constant $c>0$ such that for any $n \geq 1$, $t \in \bb R$, 
$v \in V$ and $f \in V^*$ with $\|v\| = \|f\| =1$, 
\begin{align}\label{Pf_BE_norm_Lower}
I_n: = \bb{P} \left(  \frac{\log |\langle f, G_n v \rangle| - n\lambda_1}{ \sigma \sqrt{n} } \leq t   \right)
 \geq   \Phi(t)  -  \frac{c}{\sqrt{n}}.
\end{align} 
The upper bound follows from Proposition \ref{Prop-Regularity02} together with Lemma \ref{Lem_BerryE_NormPoly}. 
Specifically, from the identity \eqref{Pf_LLN_Equality} and Theorem \ref{Thm-Regularity-Subex}, 
we get that for $x = \bb R v$ and $y = \bb R f$ with $\|v\| = \|f\| =1$, 
\begin{align}\label{Pf_BE_NormIn}
I_n  &  = \bb{P} \left(  \frac{\log |\langle f, G_n v \rangle| - n\lambda_1}{ \sigma \sqrt{n} } \leq t,   \  
     \delta(y, G_n \cdot x) > e^{- k}  \right)   \nonumber\\
&  \quad  +  \bb{P} \left(  \frac{\log |\langle f, G_n v \rangle| - n\lambda_1}{ \sigma \sqrt{n} } \leq t,   \
   \delta(y, G_n \cdot x) \leq  e^{- k}  \right)  \nonumber\\  
&  \leq   \bb{P} \left(  \frac{\log \|G_n v\| - n\lambda_1}{ \sigma \sqrt{n} } 
                   \leq t  + \frac{k}{\sigma \sqrt{n}}   \right)
    +  e^{- c_1 k^{\alpha } }. 
\end{align}
Taking $k = \floor{ (\frac{1}{c_1} \log n)^{ \frac{1}{\alpha } } }$, 
we get that there exists a constant $c>0$ such that $e^{- c_1 k^{\alpha} }  \leq  \frac{c}{\sqrt{n}}$. 
By Theorem \ref{Thm-Edge-Expan},  
we get 
\begin{align}\label{Pf_BE_UpperIn_2}
\bb{P} \left(  \frac{\log \|G_n v\| - n\lambda_1}{ \sigma \sqrt{n} } 
                   \leq t  + \frac{k}{\sigma \sqrt{n}}   \right)
& \leq   \Phi \left( t  +  \frac{ k }{\sigma \sqrt{n}}  \right)  +  \frac{c}{\sqrt{n}}  \nonumber\\
& \leq  \Phi(t)  +  \frac{c \log^{\frac{1}{ \alpha }} n }{\sqrt{n}}. 
\end{align}
The desired bound \eqref{BerryEsseen_Coeffaa} follows by combining \eqref{Pf_BE_norm_Lower}, \eqref{Pf_BE_NormIn} and \eqref{Pf_BE_UpperIn_2}. 
\end{proof}

\subsection{Proof of Theorem \ref{Thm-MDP}}
In this subsection we establish Theorem \ref{Thm-MDP} using Theorem \ref{Thm-Regularity-Subex} and Lemma \ref{Lem_MDP_VecrorNorm}.

\begin{proof}[Proof of Theorem \ref{Thm-MDP}]
By Lemma 4.4 of \cite{HL12}, 
it suffices to prove the following moderate deviation asymptotics: 
for any $t >0$, uniformly in $v \in V$ and $f \in V^*$ with $\|v\| = \|f\| =1$, 
\begin{align}
&  \lim_{n\to \infty} \frac{n}{b_n^{2}}
\log \bb{P}  \left(   \frac{\log |\langle f, G_n v \rangle|  - n\lambda_1 }{b_n} \geq t    \right)
=  - \frac{t^2}{2\sigma^2},    \label{MDP_Entry_y_Posi}   \\
&   \lim_{n\to \infty} \frac{n}{b_n^{2}}
\log \bb{P}  \left( \frac{\log |\langle f, G_n v \rangle| - n\lambda_1 }{b_n} \leq - t    \right) 
=  - \frac{t^2}{2\sigma^2}.   \label{MDP_Entry_y_Neg}
\end{align}
We first prove \eqref{MDP_Entry_y_Posi}. 
The upper bound is an easy consequence of Lemma \ref{Lem_MDP_VecrorNorm}.
Since there exists a constant $\alpha \in (0, 1)$ 
such that $\mu(\log N(g) > u) \leq c \exp \{- u^{\alpha} \}$ for any $u >0$, 
clearly $\int_{ \textup{GL}(V) } \log^2N(g)  \,  \mu(dg) < \infty$. 
Moreover, by taking $(b_n)_{n \geq 1}$ 
such that $\frac{b_n}{\sqrt{n}} \to \infty$ and $b_n = o(n^{\frac{1}{2 - \alpha}})$ as $n \to \infty$,
 we get 
\begin{align}\label{Mom-Condi-MD-bis}
 \limsup_{n \to \infty} \frac{n}{b_n^2}  \log  \Big[ n \mu \big( \log N(g) > b_n \big) \Big]  
& = \limsup_{n \to \infty} \frac{n}{b_n^2}  \log  \Big[ \mu \big( \log N(g) > b_n \big) \Big]  \notag\\
& \leq  \limsup_{n \to \infty} \left(  - \frac{n}{ b_n^{2 - \alpha} }   \right) = - \infty. 
\end{align}
Therefore, from Lemma \ref{Lem_MDP_VecrorNorm} it follows that 
for any $t >0$, 
\begin{align}\label{MDP-bn}
-  \frac{t^2}{2 \sigma^2}  
& \leq  \liminf_{n \to \infty} \frac{n}{b_n^2} 
   \log \inf_{x \in \bb P(V)} \bb P \left( \frac{\sigma(G_n, x) - n \lambda_1}{b_n}  \geq t  \right)  \notag\\
& \leq  \limsup_{n \to \infty} \frac{n}{b_n^2} 
   \log \sup_{x \in \bb P(V)} \bb P \left( \frac{\sigma(G_n, x) - n \lambda_1}{b_n}  \geq  t  \right) 
 \leq  - \frac{t^2}{2 \sigma^2}.
\end{align}
Hereafter, $(b_n)_{n \geq 1}$ is any sequence of positive numbers 
satisfying $\frac{b_n}{\sqrt{n}} \to \infty$ and $b_n = o(n^{\frac{1}{2 - \alpha}})$ as $n \to \infty$. 
By \eqref{MDP-bn} and 
the fact that $\log |\langle f, G_n v \rangle| \leq \sigma (G_n, x)$ for $x = \bb R v \in \bb P(V)$ with $\|v\| = \|f\| = 1$,
we get the desired upper bound: for any $t >0$,  
 uniformly in $v \in V$ and $f \in V^*$ with $\|v\| = \|f\| =1$, 
\begin{align}\label{MDP_Entry_y_Posi_Upp}  
\limsup_{n\to \infty} \frac{n}{b_n^{2}}
\log \bb{P}  \left(   \frac{\log |\langle f, G_n v \rangle|  - n\lambda_1 }{b_n} \geq t    \right)
\leq  - \frac{t^2}{2\sigma^2}.  
\end{align}
The lower bound can be deduced from Lemma \ref{Lem_MDP_VecrorNorm} 
together with Theorem \ref{Thm-Regularity-Subex}.
Specifically, 
by \eqref{thmRegularity01}, 
there exist constants $c_1 >0$ and $k_0 \in \bb{N}$ such that for any $n \geq k \geq k_0$,
any $v \in V$ and $f \in V^*$ with $\|v\| = \|f\| =1$, 
\begin{align}\label{Entry_Inequa_01}
 I_n: & =  \bb{P}  \left(   \frac{\log |\langle f, G_n v \rangle|  - n\lambda_1 }{b_n} \geq t    \right)  \nonumber\\
 & \geq   \bb{P}  \left(   \frac{\log |\langle f, G_n v \rangle|  - n\lambda_1 }{b_n} \geq t,  \  \log \delta (G_n \!\cdot\! x, y)  >   - k \right) 
   \nonumber\\
 & \geq  \bb{P}  \left( \frac{ \sigma(G_n, x)  - n\lambda_1 }{ b_n' } \geq t,  \  \log \delta (G_n \!\cdot\! x, y)  >   - k  \right)  \nonumber\\   
 & \geq   \bb{P}  \left(  \frac{ \sigma(G_n, x)  - n\lambda_1 }{ b_n' } \geq t  \right)
  - e^{-c_1 k^{ \alpha } },
\end{align}
where 
 $b_n' = b_n + \frac{ k }{t}$ and $t > 0$.  
We take
\begin{align}\label{Pf_Range_k}
k = \floor[\Big]{ \Big( c_2 \frac{b_n^2}{n} \Big)^{\frac{1}{ \alpha }} },  
\end{align}
where $c_2 >0$ is a  constant whose value will be chosen large enough.
Since $\frac{b_n}{\sqrt{n}}\to \infty$ and $\frac{b_n}{n} \to 0$ as $n \to \infty$, 
from $b_n' = b_n + \frac{ k }{t}$ and \eqref{Pf_Range_k},
we get that as $n \to \infty$,
\begin{align*}
\frac{b_n'}{\sqrt{n}}\to \infty  \quad  \mbox{and} \quad
\frac{b_n'}{n} =  \frac{b_n}{n} +  \frac{ k }{t n }  
\leq \frac{b_n}{n} + \frac{c}{t}  n^{- \frac{1 - \alpha}{2 - \alpha}}  \to 0. 
\end{align*}
From \eqref{MDP-bn} 
it follows that for any $t > 0$ and $\ee > 0$, there exists $n_0 \in \bb{N}$ such that for any $n \geq n_0$ and $x \in \bb P(V)$,
\begin{align*}
\bb{P}  \left(  \frac{ \sigma(G_n, x)  - n\lambda_1 }{ b_n' } \geq t  \right)
 \geq  e^{ - \frac{ (b_n')^2 }{n} \big(  \frac{t^2}{2 \sigma^2} + \ee \big) }. 
\end{align*}
Substituting this into \eqref{Entry_Inequa_01},  we obtain
\begin{align*}
I_n \geq  e^{ - \frac{ (b_n')^2 }{n} \big(  \frac{t^2}{2 \sigma^2} + \ee \big) }
   \Big[ 1 -  e^{ - c_1 k^{ \alpha } + \frac{ (b_n')^2 }{n} \big( \frac{t^2}{2 \sigma^2} + \ee \big) }   \Big]. 
\end{align*}
Choosing $c_2 >0$ in \eqref{Pf_Range_k} to be sufficiently large, 
one gets $\frac{ (b_n')^2 }{n} \big( \frac{t^2}{2 \sigma^2} + \ee \big) < c_1 k^{ \alpha }$,
so that there exists a constant $c_3 >0$ such that 
\begin{align*}
I_n \geq  e^{ - \frac{ (b_n')^2 }{n} \big(  \frac{t^2}{2 \sigma^2} + \ee \big) }
   \Big( 1 -  e^{ - c_3  k^{ \alpha } }  \Big). 
\end{align*}
Since $k = \floor[]{ ( c_2 \frac{b_n^2}{n} )^{\frac{1}{ \alpha }} } \to \infty$  
and $1 \leq \frac{b_n'}{b_n} =  1 + \frac{k}{ t b_n } \leq 1 + c \frac{b_n^{ \frac{2 - \alpha}{ \alpha } } }{ n^{ \frac{1}{\alpha} } } \to 1$
as $n \to \infty$, 
it follows that 
\begin{align*}
\liminf_{ n \to \infty } \frac{n}{ b_n^2 }  \log I_n
& \geq   \liminf_{ n \to \infty }  \frac{n}{ b_n^2 }  
   \Big[  - \frac{ (b_n')^2 }{n} \Big(  \frac{t^2}{2 \sigma^2} + \ee \Big) \Big]   
  =  - \Big(  \frac{t^2}{ 2 \sigma^2 } + \ee \Big).
\end{align*}
Letting $\ee \to 0$, the desired lower bound follows:
for any $t >0$, uniformly in $v \in V$ and $f \in V^*$ with $\|v\| = \|f\| =1$, 
\begin{align*}
\liminf_{n\to \infty} \frac{n}{b_n^{2}}
\log \bb{P}  \left(   \frac{\log |\langle f, G_n v \rangle|  - n\lambda_1 }{b_n} \geq t    \right)
\geq  - \frac{t^2}{2\sigma^2}.   
\end{align*}
This, together with the upper bound \eqref{MDP_Entry_y_Posi_Upp}, concludes the proof of \eqref{MDP_Entry_y_Posi}. 
The proof of \eqref{MDP_Entry_y_Neg} can be carried out in the same way. 
\end{proof}

\section{Proof of Cram\'er type moderate deviation expansion} 

%
%

\subsection{Smoothing inequality} \label{Ch7secAuxres001}

For any integrable function $h: \bb{R} \to \bb{C}$,
denote its Fourier transform by
$\widehat{h}(u) = \int_{\bb{R}} e^{-iuw} h(w) dw,$  $u \in \bb{R}$. 
If $\widehat{h}$ is integrable on $\bb{R}$, then by the Fourier inversion formula we have 
$h(w) = \frac{1}{2 \pi} \int_{\bb{R}} e^{iuw} \widehat{h}(u) du,$ for almost all $w \in \bb{R}$
with respect to the Lebesgue measure on $\bb R$.

Now we fix a non-negative density function $\rho$ on $\bb{R}$  with compact support $[-1, 1]$, 
whose Fourier transform $\widehat{\rho}$ is integrable on $\bb{R}$. 
For any $0< \ee <1$, define the scaled density function
$\rho_{\ee}(w) = \frac{1}{\ee} \rho(\frac{w}{\ee})$, $w \in \bb R,$ 
which has a compact support on $[-\ee^{-1},\ee^{-1}]$. 

For any $\ee >0$ and non-negative integrable function $\psi$ on $\bb R$, set
\begin{align}\label{smoo001}
\psi^+_{\varepsilon}(w) = \sup_{|w - w'| \leq \ee} \psi(w') 
\quad  \mbox{and}  \quad 
\psi^-_{\varepsilon}(w) = \inf_{|w - w'| \leq \ee} \psi(w'),
\quad w \in \bb R. 
\end{align}
We need the following smoothing inequality shown in \cite{GLL17}.
Denote by $h_1 * h_2$ the convolution of functions $h_1$ and $h_2$ on the real line. 

\begin{lemma}[\cite{GLL17}]  \label{estimate u convo}
Assume that  $\psi$ is a non-negative integrable function on $\bb R$ and that 
${\psi}^+_{\varepsilon}$ and ${\psi}^-_{\varepsilon}$ are measurable for any $\varepsilon \in (0,1)$. 
Then, 
there exists a positive constant $C_{\rho}(\ee)$ with $C_{\rho}(\ee) \to 0$ as $\ee \to 0$,
such that for any $w \in \mathbb{R}$, 
\begin{align}
\psi^-_{\ee} * \rho_{\ee^2}(w) - 
\int_{ |u| \geq \ee } \psi^-_{\varepsilon}( w-u ) \rho_{\varepsilon^2}(u) du
\leq \psi(w) \leq (1+ C_{\rho}(\varepsilon))
\psi^+_{\varepsilon} * \rho_{\varepsilon^2}(w). \nonumber
\end{align}
\end{lemma}

For sufficiently small constant $s_0 >0$ and $s \in (0, s_0)$, let 
\begin{align*}
\psi_s(w) = e^{-sw} \mathds{1}_{\{ w \geq 0 \} },  \quad  w \in \bb{R}. 
\end{align*}
With the notation $\psi_{s,\ee}^{-}(w) = \inf_{|w - w'| \leq \ee} \psi_s(w')$ (cf. \eqref{smoo001}),
we have that for any $\ee \in (0, 1)$, 
\begin{align}\label{Def-psi-see}
\psi_{s,\ee}^{-}(w) = e^{-s(w + \ee)} \mathds{1}_{ \{w \geq \ee\} },  \quad  w  \in \bb{R}. 
\end{align}
By elementary calculations, for any $s \in (0, s_0)$ and $\ee \in (0,1)$, 
the Fourier transform of $\psi_{s,\ee}^{-}$ is given by 
\begin{align}\label{Diffepsi}
\widehat {\psi}^-_{s,\varepsilon}(u) 
= \int_{\bb R} e^{-iuw} \psi_{s,\ee}^{-}(w) dw
=  e^{-2 \varepsilon s} \frac{e^{-2i \varepsilon u} }{s + iu},  \quad u \in \bb R. 
\end{align}

\subsection{An asymptotic expansion of the perturbed operator}
The goal of this section is to 
establish the precise asymptotic behaviors of 
the perturbed operator $R_{s, it}$. 
This asymptotic result will play an important role for establishing 
the Cram\'{e}r type moderate deviation expansion for the coefficients $\langle f, G_n v \rangle$
in Theorem \ref{Thm-Cram-Entry_bb}. 
In the sequel, for any fixed $t >1$, 
we shall choose $s>0$  satisfying the following equation:
\begin{align}\label{SaddleEqua}
\Lambda'(s) - \Lambda'(0) = \frac{\sigma t}{\sqrt{n}}. 
\end{align}
For brevity we denote $\sigma_s = \sqrt{\Lambda''(s)}$. 
By \cite{XGL19b}, the function $\Lambda$ is strictly convex in a small neighborhood of $0$, 
so that $\sigma_s >0$ uniformly in $s \in (0, s_0)$.

\begin{proposition}\label{KeyPropo}
Assume conditions \ref{Ch7Condi-Moment} and \ref{Ch7Condi-IP}.
Let $s > 0$ be such that \eqref{SaddleEqua} holds. 
Then, for any $\ee \in (0,1)$ and any positive sequence $(t_n)_{n \geq 1}$ satisfying 
$t_n \to \infty$ and $t_n/\sqrt{n} \to 0$ as $n \to \infty$, 
we have, uniformly in $s \in (0, s_0)$, $x\in \bb P(V)$, $t \in [t_n, o(\sqrt{n} )]$, 
$\varphi \in \mathscr{B}_{\gamma}$ and $|l| =  O(\frac{1}{\sqrt{n}})$, 
\begin{align} \label{KeyPropUpper}
&   \left|  s \sigma_s \sqrt{n}  \int_{\mathbb{R}} e^{-iuln} R_{s, iu}^n (\varphi)(x) 
\widehat {\psi}^-_{s,\varepsilon}(u) \widehat\rho_{\varepsilon^{2}}(u) du 
 -  \sqrt{ 2 \pi }  \pi_s(\varphi)  \right|  \nonumber\\
&  \leq  c \left( \frac{t}{\sqrt{n}}  + \frac{1}{t^2} \right) \|\varphi \|_{\infty}
     + c \left( |l| \sqrt{n} + \frac{1}{\sqrt{n}} \right) \|\varphi\|_{\gamma}.  
\end{align}   
\end{proposition}

%

\begin{proof}
Without loss of generality, we assume that the target function $\varphi$ is non-negative on $\bb P(V)$. 
By Lemma \ref{Ch7perturbation thm}, we have the following decomposition:
with $\delta > 0$ small enough,
\begin{align} \label{Thm1 integral1 J Nol}
s \sigma_s \sqrt{n} \int_{\mathbb{R}} e^{-iuln}  R_{s, iu}^n (\varphi)(x) 
\widehat {\psi}^-_{s,\varepsilon}(u) \widehat\rho_{\varepsilon^{2}}(u) du  
 =: I_1 + I_2 + I_3, 
\end{align}
where 
\begin{align*} 
I_1  & =  s \sigma_s \sqrt{n} \int_{|u|\geq\delta}  e^{-iuln}  R^{n}_{s, iu}(\varphi)(x) 
              \widehat {\psi}^-_{s,\varepsilon}(u) \widehat\rho_{\varepsilon^{2}}(u) du,  \nonumber\\
I_2  & =  s \sigma_s \sqrt{n}  \int_{|u|< \delta }  e^{-itln}  N^{n}_{s, iu}(\varphi )(x) 
               \widehat {\psi}^-_{s,\varepsilon}(u) \widehat\rho_{\varepsilon^{2}}(u) du,   \nonumber\\
I_3  & =  s \sigma_s \sqrt{n}  \int_{|t|<\delta}  e^{-iuln}  \lambda^{n}_{s, iu} \Pi_{s, iu}(\varphi )(x) 
               \widehat {\psi}^-_{s,\varepsilon}(u) \widehat\rho_{\varepsilon^{2}}(u) du.        
\end{align*}
For simplicity, we denote $K_s(iu)= \log \lambda_{s, iu}$ and choose the branch where $K_s(0)=0$.
Using the formula \eqref{relationlamkappa001}, we have that for $u \in (-\delta, \delta)$, 
\begin{align*}
K_s(iu) = \Lambda'(s + iu) - \Lambda'(s) - iu \Lambda'(s). 
\end{align*}
Since the function $\Lambda$ is analytic in a small neighborhood of $0$,
using Taylor's formula gives
\begin{align} \label{Ch7Expan Ks 01}
K_s(iu) = \sum_{k=2}^{\infty}  \frac{\Lambda^{(k)}(s) }{k!} (iu)^k,
\quad \mbox{where}  \  \Lambda(s) = \log \kappa(s) 
\end{align}
and 
\begin{align} \label{Expan Ks 02}
\Lambda'(s)-\Lambda'(0)=\sum_{k=2}^{\infty}\frac{\gamma_k}{(k-1)!}s^{k-1},  \quad
\mbox{where} \   \gamma_k = \Lambda^{(k)}(0).   
\end{align}
Combining \eref{SaddleEqua} and \eref{Expan Ks 02}, we get
\begin{align}\label{equation t and s 1}
 \frac{ \sigma t }{ \sqrt{n} } = \sum_{k=2}^{\infty} \frac{ \gamma_k }{ (k-1)! } s^{k-1}.
\end{align}
Since $\gamma_2 = \sigma^2>0$,  from \eqref{equation t and s 1} 
we deduce that for any $t >1$ and sufficiently large $n$, 
the equation \eqref{SaddleEqua} has a unique solution given by
\begin{align}\label{root s 1}
s = \frac{1}{\gamma_2^{ 1/2 }} \frac{t}{\sqrt{n}} 
    - \frac{\gamma_3}{ 2 \gamma_2^2 } \left( \frac{t}{ \sqrt{n} }  \right)^2
    - \frac{\gamma_4 \gamma_2 - 3\gamma_3^2}{6\gamma_2^{7/2} } \left( \frac{t}{ \sqrt{n} } \right)^3
    + \cdots. 
\end{align}
For sufficiently large $n \geq 1$, the series on the right-hand side of \eqref{root s 1} is absolutely convergent
according to the theorem on the inversion of analytic functions. 
Besides, from \eqref{SaddleEqua} and $t = o(\sqrt{n})$ we see that $s \to 0^+$ as $n \to \infty$, 
so that we can assume $s \in (0, s_0)$ for sufficiently small constant $s_0 >0$.

\textit{Estimate of $I_1$.} 
By Lemma \ref{Lem-St-NonLatt}, for fixed $\delta >0$, there exist constants $c, C > 0$ such that
\begin{align}\label{EstiRsit}
\sup_{s \in (0, s_0)} \sup_{ |u| \geq \delta }\sup_{x\in \bb P(V) }
|R^{n}_{s, iu}\varphi(x)|  \leq  C e^{-cn} \|\varphi \|_{\gamma}.  
\end{align}
From \eqref{Diffepsi} and the fact that $\rho_{\ee^2}$ is a density function on $\bb R$, we see that 
\begin{align}\label{Contrpsirho}
\sup_{u \in \bb{R}} |\widehat {\psi}^-_{s,\ee}(u)|
\leq \widehat {\psi}^-_{s,\ee}(0) = \frac{1}{s} e^{ -2\ee s }, 
\quad  
\sup_{u \in \bb{R}} |\widehat\rho_{\ee^{2}}(u)| 
\leq \widehat\rho_{\ee^{2}}(0) = 1.  
\end{align}
Using \eqref{EstiRsit} and the first inequality in \eqref{Contrpsirho}, 
and taking into account that the function $\widehat\rho_{\ee^{2}}$ is integrable on $\bb{R}$, 
we obtain the desired bound for $I_1$: 
\begin{align}\label{EstimI1n00a}
|I_1| 
         \leq  C e^{-cn } \|\varphi \|_{\gamma}. 
\end{align}

\textit{Estimate of $I_2$.}
Using the bound \eqref{Ch7SpGapContrN}, we have
that uniformly in $s \in (0, s_0)$, $u \in (-\delta, \delta)$ and $x \in \bb P(V)$, 
\begin{align*}
| N^{n}_{s, iu}(\varphi )(x) |
\leq \| N^{n}_{s, iu} \|_{\mathscr{B}_{\gamma} \to \mathscr{B}_{\gamma}} \|\varphi \|_{ \gamma}
\leq  C e^{-cn} \|\varphi \|_{ \gamma}.
\end{align*} 
This, together with \eqref{Contrpsirho}, implies the desired bound for $I_2$: 
\begin{align}\label{Esti I2n}
 |I_2| 
          \leq  C e^{-cn}  \|\varphi \|_{\gamma}. 
\end{align}

\textit{Estimate of $I_3$.}
For brevity, we denote for any $s \in (0, s_0)$ and $x \in \bb P(V)$,
\begin{align}\label{Defpsisxt}
\Psi_{s,x}(u):= \Pi_{s, iu}(\varphi)(x) 
\widehat {\psi}^-_{s,\varepsilon}(u) \widehat\rho_{\varepsilon^{2}}(u),  \quad  -\delta < u < \delta. 
\end{align}
Recalling that $K_s(iu)= \log \lambda_{s,iu}$, we decompose the term $I_3$ into two parts: 
\begin{align}  \label{DecompoJ1}
I_3 =  I_{31} + I_{32}, 
\end{align}
where 
\begin{align*}
&   I_{31} = s \sigma_s \sqrt{n}  
  \int_{ n^{- \frac{1}{2}} \log n \leq |u| \leq \delta  }  e^{n K_s( iu) - iu ln} \Psi_{s,x}(u) du,  \nonumber\\
&   I_{32}  = s \sigma_s \sqrt{n}  \int_{ |u| <  n^{-\frac{1}{2}} \log n }  e^{n K_s( iu) - iu ln} \Psi_{s,x}(u) du. 
\end{align*}

\textit{Estimate of $I_{31}$.}
By \eqref{Ch7SpGapContrN}, there exists a constant $c >0$ 
such that for any $s \in (0, s_0)$, $u \in [-\delta, \delta]$ and $x \in \bb P(V)$,
\begin{align}\label{ContrPirs}
|\Pi_{s, iu}(\varphi)(x)| \leq c \|\varphi\|_{\gamma}.
\end{align}
This, together with \eqref{Contrpsirho}, yields that
there exists a constant $c >0$ such that for all $s \in (0, s_0)$, $|u| < \delta$, $x \in \bb P(V)$
and $\varphi \in \mathscr{B}_{\gamma}$,  
\begin{align} \label{Estipsi01}
|\Psi_{s,x}(u)|  \leq  \frac{c}{s} \|\varphi\|_{\gamma}.
\end{align}
Using \eqref{Ch7Expan Ks 01} and noting that $\Lambda''(s) = \sigma_s^2>0 $, 
we find that there exists a constant $c>0$ such that for all $s \in (0, s_0)$, 
\begin{align*}
\Re \big( K_s(iu) \big) = 
\Re \left( \sum_{k=2}^{\infty} \frac{K_s^{(k)}(0) (iu)^k }{k!} \right) < -\frac{1}{8} \sigma_s^2 u^2 < - c u^2. 
\end{align*}
Combining this with \eqref{Estipsi01}, we derive that 
there exists a constant $c >0$ such that for all $s \in (0, s_0)$, $x \in \bb P(V)$
and $\varphi \in \mathscr{B}_{\gamma}$,  
\begin{align} \label{EstiJ11}
|I_{31}| \leq  c \sqrt{n}  \|\varphi\|_{\gamma}
\int_{ n^{ - \frac{1}{2} }  \log n  \leq |u| \leq  \delta  } e^{- c n u^2 }  du  
\leq \frac{c}{ \sqrt{n} }  \|\varphi\|_{\gamma}. 
\end{align}

\textit{Estimate of $I_{32}$.}
By a change of variable $u' = \sigma_s \sqrt{n} u$, we get
\begin{align*}
 I_{32}  &  =  s \int_{- \sigma_s \log n}^{\sigma_s \log n }
  e^{-\frac{u^2}{2}}  \exp \left\{ \sum_{k=3}^{\infty} \frac{ K_s^{(k)}(0) (iu)^k}{ \sigma_s^k \, k! \, n^{k/2-1}} \right\}  
  e^{-iul \sqrt{n} / \sigma_s} \Psi_{s,x} \left( \frac{u}{\sigma_s \sqrt{n}} \right) du  \nonumber\\
& = I_{321} + I_{322} + I_{323}, 
\end{align*}
where 
\begin{align*}
&  I_{321} = s \int_{- \sigma_s \log n}^{\sigma_s \log n }  e^{-\frac{u^2}{2}}  
  \left[ \exp \left\{ \sum_{k=3}^{\infty} \frac{ K_s^{(k)}(0) (iu)^k}{ \sigma_s^k \, k! \, n^{k/2-1}} \right\} - 1 \right] 
  e^{-iul \sqrt{n} / \sigma_s} \Psi_{s,x} \left( \frac{u}{\sigma_s \sqrt{n}} \right) du  \nonumber\\ 
&  I_{322} = s \int_{- \sigma_s \log n}^{\sigma_s \log n }  e^{-\frac{u^2}{2}}   
  \left[ e^{-iul \sqrt{n} / \sigma_s} - 1 \right] 
   \Psi_{s,x} \left( \frac{u}{\sigma_s \sqrt{n}} \right) du  \nonumber\\ 
&  I_{323} = s \int_{- \sigma_s \log n}^{\sigma_s \log n }  e^{-\frac{u^2}{2}}    
   \Psi_{s,x} \left( \frac{u}{\sigma_s \sqrt{n}} \right) du. 
\end{align*}

\textit{Estimate of $I_{321}$.}
By simple calculations, there exists a constant $c>0$ such that for all 
$|u| \leq \sigma_s \log n$ and $s \in (0, s_0)$, 
\begin{align*} 
\left| \exp \left\{ \sum_{k=3}^{\infty} \frac{ K_s^{(k)}(0) (iu)^k}{ \sigma_s^k \, k! \, n^{k/2-1}} \right\} - 1  \right|
\leq  c \frac{|u|^3}{ \sqrt{n} }. 
\end{align*}
This, together with \eqref{Estipsi01} and the fact $|e^{-iul \sqrt{n} / \sigma_s}| = 1$, implies that 
\begin{align}\label{Esti_I321_aa}
|I_{321}| \leq \frac{c}{\sqrt{n}} \|\varphi\|_{\gamma}. 
\end{align}

\textit{Estimate of $I_{322}$.}
Since 
$|e^{-iul \sqrt{n} / \sigma_s} - 1| \leq |u l \sqrt{n} / \sigma_s| \leq c |l| \sqrt{n} |u|$, 
combining this with \eqref{Estipsi01} gives
\begin{align}\label{Esti_I322_aa}
|I_{322}| \leq  c |l| \sqrt{n} \|\varphi\|_{\gamma}. 
\end{align}

\textit{Estimate of $I_{323}$.}
We shall establish the following bound for $I_{323}$: there exists a constant $c>0$ such that
for all $s \in (0, s_0)$, $x\in \bb P(V)$, $t \in [t_n, o(\sqrt{n} )]$
and $\varphi \in \mathscr{B}_{\gamma}$,
\begin{align}\label{Esti_I32_bb}
\left|  I_{323} -   \sqrt{ 2 \pi } \pi_s(\varphi)  \right| 
\leq   cs \| \varphi \|_{\infty} + \frac{c}{\sqrt{n}} \| \varphi \|_{\gamma}
     + \frac{c}{t^2} \| \varphi \|_{\infty}. 
\end{align}
Now let us prove \eqref{Esti_I32_bb}. 
For brevity, denote 
\begin{align*}
u_n = \frac{u}{ \sigma_s \sqrt{n} }. 
\end{align*}
In view of \eqref{Defpsisxt}, we write 
\begin{align*}
\Psi_{s,x} (u_n) =  h_1(u_n) + h_2(u_n) + h_3(u_n) + h_4(u_n), 
\end{align*}
where
\begin{align*}
&   h_1(u_n) =  \big[ \Pi_{s, i u_n}(\varphi)(x) -  \pi_s(\varphi) \big] 
\widehat {\psi}^-_{s,\varepsilon}(u_n)
\widehat\rho_{\varepsilon^{2}}(u_n),    \nonumber\\
&   h_2(u_n)  =    \pi_s(\varphi) \widehat {\psi}^-_{s,\varepsilon}(u_n) 
\left[ \widehat\rho_{\varepsilon^{2}}(u_n) -  \widehat\rho_{\varepsilon^{2}}(0)  \right]  \nonumber\\
&   h_3(u_n) =   \pi_s(\varphi) 
 \left[ \widehat {\psi}^-_{s,\varepsilon}(u_n) - \widehat {\psi}^-_{s,\varepsilon}(0) \right]
   \widehat\rho_{\varepsilon^{2}}(0),  \nonumber\\
&   h_4(u_n) =   \pi_s(\varphi) \widehat {\psi}^-_{s,\varepsilon}(0)
\widehat\rho_{\varepsilon^{2}}(0). 
\end{align*}
With the above notation, the term $I_{323}$ can be decomposed into four parts: 
\begin{align}\label{DecompoJ122}
I_{323} = J_{1} + J_{2} + J_{3} + J_{4}, 
\end{align}
where for $j = 1,2,3,4$, 
\begin{align*}
J_{j} =  s  \int_{- \sigma_s \log n }^{ \sigma_s \log n } e^{-\frac{u^2}{2}} 
h_j(u_n) du. 
\end{align*}

\textit{Estimate of $J_{1}$.}
By Lemma \ref{Ch7perturbation thm}, we have
$|\Pi_{s, i u_n}(\varphi)(x) -  \pi_s(\varphi)|
\leq  c  \frac{|u|}{\sqrt{n}} \| \varphi \|_{\gamma}$, uniformly in $x \in \bb P(V)$, $s \in (0, s_0)$
and $|u| \leq \sigma_s \log n $. 
Combining this with \eqref{Contrpsirho} gives
$|h_1(u_n)|  \leq  c  \frac{|u|}{ s \sqrt{n}} \| \varphi \|_{\gamma},$ 
and hence
\begin{align}\label{Controlh1}
| J_1 | \leq  \frac{c}{ \sqrt{n} } \| \varphi \|_{\gamma}. 
\end{align}

\textit{Estimate of $J_{2}$.}
It is easy to verify that 
$|\widehat\rho_{\varepsilon^{2}}(u_n) -  \widehat\rho_{\varepsilon^{2}}(0)| \leq c \frac{|u_n|}{\varepsilon^4} $.
This, together with \eqref{Contrpsirho}, leads to 
$|h_2(u_n)|  \leq  \frac{c}{\varepsilon^4}  \frac{|u|}{ s \sqrt{n}},$
and therefore, 
\begin{align}\label{Controlh2}
|J_2| \leq \frac{c}{\varepsilon^4}  \frac{1}{ \sqrt{n} } \|\varphi\|_{\infty}
 \leq  \frac{C}{ \sqrt{n} } \|\varphi\|_{\infty}.
\end{align}

\textit{Estimate of $J_{3}$.}
By the definition of the function $\widehat {\psi}^-_{s,\varepsilon}$ (see \eqref{Diffepsi}), we have 
\begin{align*}
 \widehat {\psi}^-_{s,\varepsilon}(u_n) - \widehat {\psi}^-_{s,\varepsilon}(0) 
& =  e^{-2 \varepsilon s} \left( \frac{e^{- i \varepsilon u_n} }{s + iu_n} - \frac{1}{s}  \right) \nonumber\\
& =  e^{-2 \varepsilon s}  e^{- i \varepsilon u_n} \left( \frac{1}{s + iu_n} - \frac{1}{s} \right)
    + e^{-2 \varepsilon s} \frac{1}{s} \left( e^{- i \varepsilon u_n} - 1 \right)   \nonumber\\
& =  e^{-2 \varepsilon s} \left( e^{- i \varepsilon u_n} - 1 \right) \frac{-i s u_n - u_n^2}{s(s^2 + u_n^2)}   \nonumber\\
& \quad  +  e^{-2 \varepsilon s} \frac{-i s u_n - u_n^2}{s(s^2 + u_n^2)} 
    +  e^{-2 \varepsilon s} \frac{1}{s} \left( e^{- i \varepsilon u_n} - 1 \right)     \nonumber\\
& =: A_1 (t)  +  A_2 (t) + A_3 (t). 
\end{align*}
Since $\widehat\rho_{\varepsilon^{2}}(0) = 1$, it follows that $J_3 = J_{31} + J_{32} + J_{33}$, 
where  
\begin{align*}
J_{3j} = s \pi_s(\varphi) 
\int_{- \sigma_s \log n }^{ \sigma_s \log n } e^{-\frac{u^2}{2}}  A_j(u) du,  \quad  j = 1, 2, 3.  
\end{align*}
To deal with $J_{31}$, we first give a bound for $A_1 (u)$. 
Using the basic inequality $|e^z - 1| \leq e^{\Re z} |z|$, $z \in \bb C$, we get 
\begin{align*}
|A_1 (u)|  & \leq e^{-2 \varepsilon s} |u_n| \frac{s |u_n| + u_n^2}{s(s^2 + u_n^2)}  \nonumber\\
& = e^{-2 \varepsilon s}  \frac{u_n^2}{s^2 + u_n^2} + e^{-2 \varepsilon s} \frac{|u_n|}{s}  \frac{u_n^2}{s^2 + u_n^2} 
   \nonumber\\
& \leq c \frac{u^2}{t^2 + u^2} + c \frac{|u|}{t} \frac{u^2}{t^2 + u^2}  \nonumber\\
& \leq c \frac{u^2}{t^2} + c \frac{|u|^3}{t^3}. 
\end{align*}
Then, recalling that $s = O(\frac{t}{\sqrt{n}})$ and $|\pi_s (\varphi)| \leq c \|\varphi\|_{\infty}$,  
we obtain that there exists a constant $c>0$ such that for all $n \geq 1$, 
$s \in (0, s_0)$ and $\varphi \in \mathscr{B}_{\gamma}$, 
\begin{align}\label{Esti_J31}
|J_{31}| \leq \frac{c}{ \sqrt{n} } \|\varphi\|_{\infty}. 
\end{align}
For $J_{32}$, using again $s = O(\frac{t}{\sqrt{n}})$ and the fact that the integral of an odd function 
over a symmetric interval is identically zero, 
by elementary calculations we deduce that there exists a constant $c>0$ such that for all $n \geq 1$, 
$s \in (0, s_0)$ and $\varphi \in \mathscr{B}_{\gamma}$, 
\begin{align}\label{Esti_J32}
|J_{32}| = \left| e^{-2 \varepsilon s} \pi_s(\varphi) \int_{- \sigma_s \log n }^{ \sigma_s \log n } e^{-\frac{u^2}{2}}  
     \frac{u_n^2}{s^2 + u_n^2} du \right| 
 \leq \frac{c}{t^2} \|\varphi\|_{\infty}. 
\end{align}
For $J_{33}$, using again the inequality $|e^z - 1| \leq e^{\Re z} |z|$, $z \in \bb C$,
we see that there exists a constant $c>0$ such that for all $n \geq 1$, 
$s \in (0, s_0)$ and $\varphi \in \mathscr{B}_{\gamma}$, 
\begin{align}\label{Esti_J33}
|J_{33}| \leq \frac{c}{ \sqrt{n} } \|\varphi\|_{\infty}. 
\end{align}
Consequently, putting together the bounds \eqref{Esti_J31}, \eqref{Esti_J32} and \eqref{Esti_J33}, we obtain
\begin{align} \label{IntegJ123Nol}
J_{3} \leq  \frac{c}{ \sqrt{n} } \|\varphi\|_{\infty} +   \frac{c}{ t^2 }  \|\varphi\|_{ \infty }. 
\end{align}

\textit{Estimate of $J_{4}$.}
It follows from \eqref{Diffepsi} and $\widehat\rho_{\varepsilon^{2}}(0) = 1$ that 
\begin{align}\label{EstiInteJ4n}
J_{4} 
 =   \pi_s(\varphi) e^{ -2\varepsilon s }   
  \int_{- \sigma_s \log n }^{ \sigma_s \log n }
e^{-\frac{u^2}{2}}  du. 
\end{align}
Since there exists a constant $c>0$ such that for all $s \in (0, s_0)$, 
\begin{align*} 
\sqrt{ 2 \pi }  >  \int_{-\sigma_s \log n  }^{ \sigma_s \log n } e^{- \frac{u^2}{2}}  du  
 >  \sqrt{ 2 \pi }  - \frac{c}{n}, 
\end{align*}
we get 
\begin{align*}
J_4 =  \sqrt{ 2 \pi } \pi_s(\varphi) e^{ -2\varepsilon s }  \left[ 1 + O \Big( \frac{1}{n} \Big) \right]. 
\end{align*}
Note that $s = O(\frac{t}{\sqrt{n}})$ with $y >1$. It follows that
\begin{align}\label{EstiIntegral00a}
\Big|  J_4 -   \sqrt{ 2 \pi } \pi_s(\varphi)  \Big| \leq  \frac{ct}{\sqrt{n}} \| \varphi \|_{\infty}. 
\end{align}
In view of \eqref{DecompoJ122}, putting together the bounds \eqref{Controlh1}, \eqref{Controlh2}, 
\eqref{IntegJ123Nol} and \eqref{EstiIntegral00a}, and noting that $\|\varphi\|_{\infty} \leq \|\varphi\|_{\gamma}$, 
we obtain the desired bound \eqref{Esti_I32_bb}. 
Combining \eqref{Esti_I321_aa}, \eqref{Esti_I322_aa} and \eqref{Esti_I32_bb}, 
we derive that 
there exists a constant $c>0$ such that
for all $s \in (0, s_0)$, $x\in \bb P(V)$, $t \in [t_n, o(\sqrt{n} )]$, 
$\varphi \in \mathscr{B}_{\gamma}$ and $|l| =  O \big( \frac{1}{\sqrt{n}} \big)$,
\begin{align}\label{Esti_I32_Bound}
\left|  I_{32} - \sqrt{ 2 \pi } \pi_s(\varphi)  \right| 
\leq  c \left( \frac{t}{\sqrt{n}}  + \frac{1}{t^2} \right) \|\varphi \|_{\infty}
     + c \left( |l| \sqrt{n} + \frac{1}{\sqrt{n}} \right) \|\varphi\|_{\gamma}.  
\end{align}
Putting together \eqref{EstimI1n00a}, \eqref{Esti I2n},  \eqref{EstiJ11} and \eqref{Esti_I32_Bound}, 
we conclude the proof of Proposition \ref{KeyPropo}.
\end{proof}

For $s \in (-s_0, 0)$ where $s_0 >0$ is sufficiently small, let 
\begin{align*}
\phi_s(w) = e^{-sw} \mathds{1}_{\{ w \leq 0 \} },  \quad  w \in \bb{R}. 
\end{align*}
With the notation in \eqref{smoo001}, for $\ee \in (0, 1)$, 
the function $\phi_{s,\ee}^{+}$ is given as follows: 
$\phi_{s,\ee}^{+}(w) = 0$ when $w > \ee$; 
$\phi_{s,\ee}^{+}(w) = 1$ when $w \in [-\ee, \ee]$;
$\phi_{s,\ee}^{+}(w) = e^{-s(w + \ee)}$ when $w < -\ee$.  
By calculations, one can give the explicit expression for
the Fourier transform of $\phi_{s,\ee}^{+}$:
\begin{align}\label{Diffephi}
 \widehat {\phi}^+_{s,\ee}(u)  
=  \int_{\bb{R}} e^{-iuw} \phi_{s,\ee}^{+}(w) dw
= 2 \frac{\sin (\ee u)}{u}   
    + e^{i \ee u }  \frac{1}{-s - i u },  \quad   u \in \bb{R}.   
\end{align}
In the sequel, for any fixed $t >1$, we choose $s<0$  satisfying the equation:
\begin{align}\label{SaddleEqua-bis}
\Lambda'(s) - \Lambda'(0) = - \frac{\sigma t}{\sqrt{n}}. 
\end{align}

\begin{proposition}\label{KeyPropo-02}
Assume conditions \ref{Ch7Condi-Moment} and \ref{Ch7Condi-IP}. 
Let $\phi^+_{s,\ee}$ be defined in \eqref{Diffephi}. 
Suppose that $s<0$ satisfies the equation \eqref{SaddleEqua-bis}. 
Then, for any $0< \varepsilon < 1$ and
any positive sequence $(t_n)_{n \geq 1}$ satisfying $\lim_{n \to \infty} t_n = + \infty$
with $t_n = o(\sqrt{n} )$, 
we have, uniformly in $s \in (-s_0, 0)$, $x\in \bb P(V)$, $t \in [t_n, o(\sqrt{n} )]$, 
$\varphi \in \mathscr{B}_{\gamma}$ and $l \in \bb R$ with $|l| =  O(\frac{1}{\sqrt{n}})$, 
\begin{align*} 
&  \left|  -s \sigma_s \sqrt{n} \, e^{n h_s(l)} \int_{\bb{R}} e^{-itln} R_{s, iu}^n (\varphi)(x) 
\widehat{\phi}^+_{s,\ee}(u) \widehat\rho_{\ee^{2}}(u) du 
 -  \sqrt{ 2 \pi }  \pi_s(\varphi)  \right|  \nonumber\\
& \leq  c \left( \frac{t}{\sqrt{n}}  + \frac{1}{t^2} \right) \|\varphi \|_{\infty}
     + c \left( |l| \sqrt{n} + \frac{1}{\sqrt{n}} \right) \|\varphi\|_{\gamma}.
\end{align*}
\end{proposition}

\begin{proof}
Since the proof of Proposition \ref{KeyPropo-02} can be carried out 
in an analogous way as that of Proposition \ref{KeyPropo}, we only sketch the main differences.

Without loss of generality, we assume that the target function $\varphi$ is non-negative. 
From Lemma \ref{Ch7perturbation thm}, we have the following decomposition:
with $\delta > 0$ small enough,
\begin{align} \label{Pro_Neg_s_I}
-s \sigma_s \sqrt{n} \int_{\mathbb{R}} e^{-iuln}  R_{s, iu}^n (\varphi)(x) 
\widehat{\phi}^+_{s,\ee}(u) \widehat\rho_{\varepsilon^{2}}(u) du  
 = I_1 + I_2 + I_3, 
\end{align}
where 
\begin{align*} 
I_1  & =  -s \sigma_s \sqrt{n} \int_{|u|\geq\delta}  e^{-itln}  R^{n}_{s, iu}(\varphi)(x) 
              \widehat{\phi}^+_{s,\ee}(t)  \widehat\rho_{\varepsilon^{2}}(u) du,  \nonumber\\
I_2  & =  -s \sigma_s \sqrt{n}  \int_{|u|< \delta }  e^{-iuln}  N^{n}_{s, iu}(\varphi )(x) 
               \widehat{\phi}^+_{s,\ee}(u)  \widehat\rho_{\varepsilon^{2}}(u) du,   \nonumber\\
I_3  & =  -s \sigma_s \sqrt{n}  \int_{|u|<\delta}  e^{-iuln}  \lambda^{n}_{s, iu} \Pi_{s, iu}(\varphi )(x) 
              \widehat{\phi}^+_{s,\ee}(u)  \widehat\rho_{\varepsilon^{2}}(u) du.        
\end{align*}
Similarly to the proof of \eqref{equation t and s 1}, from \eqref{SaddleEqua-bis} one can verify that 
\begin{align}\label{equation_s_02}
- \frac{ \sigma t }{ \sqrt{n} } = \sum_{k=2}^{\infty} \frac{ \gamma_k }{ (k-1)! } s^{k-1},
\end{align}
where $\gamma_k = \Lambda^{(k)}(0).$
For any $y>1$ and sufficiently large $n$, 
the equation \eqref{equation_s_02} has a unique solution given by
\begin{align}\label{root_s_2}
s = \frac{1}{\gamma_2^{ 1/2 }} \left( -\frac{t}{\sqrt{n}} \right) 
    - \frac{\gamma_3}{ 2 \gamma_2^2 } \left( -\frac{t}{ \sqrt{n} }  \right)^2
    - \frac{\gamma_4\gamma_2-3\gamma_3^2}{6\gamma_2^{7/2} } \left( -\frac{t}{\sqrt{n}} \right)^3
    + \cdots. 
\end{align}
The series on the right-hand side of \eqref{root_s_2} is absolutely convergent,
and we can assume that $s \in (-s_0, 0)$ for sufficiently small constant $\eta>0$.

\textit{Estimate of $I_1$.} 
By Lemma \ref{Lem-St-NonLatt}, there exist  constants $c, C > 0$ such that
\begin{align}\label{EstiRsit_bb}
\sup_{s \in (-s_0, 0)} \sup_{ |u| \geq \delta }\sup_{x \in \bb P(V) }
|R^{n}_{s, iu}\varphi(x)|  \leq  C e^{-cn} \|\varphi \|_{\gamma}.  
\end{align}
From \eqref{Diffephi} and the fact that $\rho_{\ee^2}$ is a density function on $\bb R$, we see that 
\begin{align}\label{Contrpsirho_bb}
\sup_{u \in \bb{R}} |\widehat {\phi}^-_{s,\ee}(u)|
\leq \widehat {\phi}^-_{s,\ee}(0) = \frac{1}{-s} + 2 \ee, 
\quad  
\sup_{u \in \bb{R}} |\widehat\rho_{\ee^{2}}(u)| 
\leq \widehat\rho_{\ee^{2}}(0) = 1.  
\end{align}
From \eqref{EstiRsit_bb} and the first inequality in \eqref{Contrpsirho_bb}, 
the desired bound for $I_1$ follows: 
\begin{align}\label{EstimI1n00_b}
|I_1| 
 \leq  C e^{-cn } \|\varphi \|_{\gamma}. 
\end{align}

\textit{Estimate of $I_2$.}
Using the bound \eqref{Ch7SpGapContrN} and \eqref{Contrpsirho_bb}, one has
\begin{align}\label{Esti_I2n_bb}
 |I_2| 
          \leq  C e^{-cn}  \|\varphi \|_{\gamma}. 
\end{align}

\textit{Estimate of $I_3$.}
For brevity, we denote for any $s \in (-s_0, 0)$ and $x \in \bb P(V)$,
\begin{align}\label{Defpsisxt_bb}
\Psi_{s,x}(u):= \Pi_{s, iu}(\varphi)(x) 
\widehat {\phi}^+_{s,\varepsilon}(u) \widehat\rho_{\varepsilon^{2}}(u),  \quad  -\delta < u < \delta. 
\end{align}
Recalling that $K_s(iu)= \log \lambda_{s,iu}$, we decompose the term $I_3$ into two parts: 
\begin{align}  \label{DecompoJ1_bis}
I_3 =  I_{31} + I_{32}, 
\end{align}
where 
\begin{align*}
&   I_{31} = -s \sigma_s \sqrt{n}  
  \int_{ n^{- \frac{1}{2}} \log n \leq |u| \leq \delta  }  e^{n K_s( iu) - iu ln} \Psi_{s,x}(u) du,  \nonumber\\
&   I_{32}  = -s \sigma_s \sqrt{n}  \int_{ |u| <  n^{-\frac{1}{2}} \log n }  e^{n K_s( iu) -  iu ln} \Psi_{s,x}(u) du. 
\end{align*}

\textit{Estimate of $I_{31}$.}
Using \eqref{Ch7SpGapContrN} and \eqref{Contrpsirho_bb}, we get that
there exists a constant $c >0$ such that for all $s \in (-s_0, 0)$, $|u| < \delta$, $x \in \bb P(V)$
and $\varphi \in \mathscr{B}_{\gamma}$,  
\begin{align} \label{Estipsi01_bis}
|\Psi_{s,x}(u)|  \leq  \frac{c}{-s} \|\varphi\|_{\gamma}.
\end{align}
Similarly to the proof of \eqref{EstiJ11}, it follows that 
there exists a constant $c >0$ such that for all $s \in (-s_0, 0)$, $x \in \bb P(V)$
and $\varphi \in \mathscr{B}_{\gamma}$,  
\begin{align} \label{EstiJ11_bis}
|I_{31}| \leq  c \sqrt{n}  \|\varphi\|_{\gamma}
\int_{ n^{ - \frac{1}{2} }  \log n  \leq |u| \leq  \delta  } e^{- c n u^2 }  du  
\leq \frac{c}{ \sqrt{n} }  \|\varphi\|_{\gamma}. 
\end{align}

\textit{Estimate of $I_{32}$.}
Making a change of variable $u' = u \sigma_s \sqrt{n} $, we get
\begin{align*}
& I_{32}    =  -s \int_{- \sigma_s \log n}^{\sigma_s \log n }
  e^{-\frac{u^2}{2}}  \exp \left\{ \sum_{k=3}^{\infty} \frac{ K_s^{(k)}(0) (iu)^k}{ \sigma_s^k \, k! \, n^{k/2-1}} \right\}  
  e^{-iul \sqrt{n} / \sigma_s} \Psi_{s,x} \left( \frac{u}{\sigma_s \sqrt{n}} \right) du  \nonumber\\
& = (-s) \int_{- \sigma_s \log n}^{\sigma_s \log n }  e^{-\frac{u^2}{2}}  
  \left[ \exp \left\{ \sum_{k=3}^{\infty} \frac{ K_s^{(k)}(0) (iu)^k}{ \sigma_s^k \, k! \, n^{k/2-1}} \right\} - 1 \right] 
  e^{-iul \sqrt{n} / \sigma_s} \Psi_{s,x} \left( \frac{u}{\sigma_s \sqrt{n}} \right) du  \nonumber\\ 
& \qquad  +  (-s) \int_{- \sigma_s \log n}^{\sigma_s \log n }  e^{-\frac{u^2}{2}}   
  \left[ e^{-iul \sqrt{n} / \sigma_s} - 1 \right] 
   \Psi_{s,x} \left( \frac{u}{\sigma_s \sqrt{n}} \right) du  \nonumber\\ 
& \qquad  +  (-s) \int_{- \sigma_s \log n}^{\sigma_s \log n }  e^{-\frac{u^2}{2}}    
   \Psi_{s,x} \left( \frac{u}{\sigma_s \sqrt{n}} \right) du  \nonumber\\
& = : I_{321} + I_{322} + I_{323}. 
\end{align*}
For $I_{321}$ and $I_{322}$, 
in a similar way as in the proof of \eqref{Esti_I321_aa} and \eqref{Esti_I322_aa}, we have
\begin{align}\label{Esti_I321_I322_bb}
|I_{321}| \leq \frac{c}{\sqrt{n}} \|\varphi\|_{\gamma},
\qquad 
|I_{322}|  \leq  c |l| \sqrt{n} \|\varphi\|_{\gamma}. 
\end{align}
For $I_{323}$, we shall establish the following bound: there exists a constant $c>0$ such that
for all $s \in (-s_0, 0)$, $x\in \bb P(V)$, $t \in [t_n, o(\sqrt{n} )]$ and  
$\varphi \in \mathscr{B}_{\gamma}$, 
\begin{align}\label{Esti_I32_bb_02}
\left|  I_{323} -   \sqrt{ 2 \pi } \pi_s(\varphi)  \right| 
\leq   c  \frac{t}{\sqrt{n}}  \| \varphi \|_{\infty} + \frac{c}{\sqrt{n}} \| \varphi \|_{\gamma}
     + \frac{c}{t^2} \| \varphi \|_{\infty}. 
\end{align}
For brevity, denote $u_n = \frac{u}{ \sigma_s \sqrt{n} }.$
In view of \eqref{Defpsisxt_bb}, we write 
\begin{align*}
\Psi_{s,x} (u_n ) =  h_1(u_n) + h_2(u_n) + h_3(u_n) + h_4(u_n), 
\end{align*}
where
\begin{align*}
&   h_1(u_n) =  \big[ \Pi_{s, i u_n}(\varphi)(x) -  \pi_s(\varphi) \big] 
\widehat {\phi}^+_{s,\varepsilon}(u_n)
\widehat\rho_{\varepsilon^{2}}(u_n),    \nonumber\\
&   h_2(u_n)  =    \pi_s(\varphi) \widehat {\phi}^+_{s,\varepsilon}(u_n)
\left[ \widehat\rho_{\varepsilon^{2}}(u_n) -  \widehat\rho_{\varepsilon^{2}}(0)  \right]  \nonumber\\
&   h_3(u_n) =   \pi_s(\varphi) 
 \left[ \widehat {\phi}^+_{s,\varepsilon}(u_n) - \widehat {\phi}^+_{s,\varepsilon}(0) \right]
   \widehat\rho_{\varepsilon^{2}}(0),  \nonumber\\
&   h_4(u_n) =   \pi_s(\varphi) \widehat {\phi}^+_{s,\varepsilon}(0) \widehat\rho_{\varepsilon^{2}}(0). 
\end{align*}
Then $I_{323}$ can be decomposed into four parts: 
\begin{align}\label{DecompoJ122_b}
I_{323} = J_{1} + J_{2} + J_{3} + J_{4}, 
\end{align}
where for $j = 1,2,3,4$, 
\begin{align*}
J_{j} =  -s  \int_{- \sigma_s \log n }^{ \sigma_s \log n } e^{-\frac{u^2}{2}} 
h_j(u_n) du. 
\end{align*}

\textit{Estimates of $J_1$ and $J_2$.}
Similarly to the proof of \eqref{Controlh1} and \eqref{Controlh2},
one can verify that 
\begin{align}\label{Esti_J1_J2_bb}
| J_1 | \leq  \frac{c}{ \sqrt{n} } \| \varphi \|_{\gamma},
\qquad 
|J_2| \leq \frac{c}{\varepsilon^4}  \frac{1}{ \sqrt{n} } \|\varphi\|_{\infty}
 \leq  \frac{C}{ \sqrt{n} } \|\varphi\|_{\infty}.
\end{align}

\textit{Estimate of $J_{3}$.}
By the definition of the function $\widehat {\phi}^+_{s,\varepsilon}$ (see \eqref{Diffephi}), we have 
\begin{align*}
 \widehat {\phi}^+_{s,\varepsilon}(u_n) - \widehat {\phi}^+_{s,\varepsilon}(0) 
&  = 2 \left( \frac{\sin \ee u_n}{u_n} - \ee \right)
    + \left( e^{i \ee u_n} \frac{1}{-s - i u_n} - \frac{1}{-s} \right) \nonumber\\
& =  2 \left( \frac{\sin \ee u_n}{u_n} - \ee \right) +
      \frac{ e^{i \ee u_n} - 1 }{-s - i t_n}
     + \left( \frac{1}{-s - i u_n} - \frac{1}{-s}  \right)   \nonumber\\
& =: A_1 (u)  +  A_2 (u) + A_3 (u). 
\end{align*}
Since $\widehat\rho_{\varepsilon^{2}}(0) = 1$, it follows that $J_3 = J_{31} + J_{32} + J_{33}$, 
where  
\begin{align*}
J_{3j} = - s \pi_s(\varphi) 
\int_{- \sigma_s \log n }^{ \sigma_s \log n } e^{-\frac{u^2}{2}}  A_j(u) du,  \quad  j = 1, 2, 3.  
\end{align*}
For $J_{31}$, since $| A_1 (u) | \leq c |u|^2/n$ and $|\pi_s(\varphi)| \leq c\|\varphi\|_{\infty}$, 
we have
\begin{align}\label{Esti_J31_bb}
| J_{31} | \leq c \frac{-s}{n} \|\varphi\|_{\infty} \leq \frac{C}{n} \|\varphi\|_{\infty}. 
\end{align} 
For $J_{32}$,
using the inequality $|e^z - 1| \leq e^{\Re z} |z|$, $z \in \bb C$, we get 
\begin{align*}
|A_2 (u)|  \leq  \frac{ \ee |u_n| }{|-s - i u_n|} \leq  \frac{ \ee |u_n| }{-s}. 
\end{align*}
Hence, 
\begin{align}\label{Esti_J32_bb}
|J_{32}| \leq \frac{c}{ \sqrt{n} } \|\varphi\|_{\infty}. 
\end{align}
For $J_{33}$, note that $A_3 (u) = \frac{-is u_n - u_n^2}{-s(s^2 + u_n^2)}$ and $s = O(\frac{-t}{\sqrt{n}})$. 
Using the fact that the integral of an odd function over a symmetric interval is identically zero, 
by elementary calculations we derive that 
\begin{align}\label{Esti_J33_bb}
|J_{33}| = \left| \pi_s(\varphi) \int_{- \sigma_s \log n }^{ \sigma_s \log n } e^{-\frac{u^2}{2}}  
     \frac{u_n^2}{s^2 + u_n^2} du \right| 
 \leq \frac{c}{t^2} \|\varphi\|_{\infty}. 
\end{align}
Putting together \eqref{Esti_J31_bb}, \eqref{Esti_J32_bb} and \eqref{Esti_J33_bb}, we obtain
\begin{align} \label{Esti_J3_bb}
J_{3} \leq  \frac{c}{ \sqrt{n} } \|\varphi\|_{\infty} +   \frac{c}{ t^2 }  \|\varphi\|_{ \infty }. 
\end{align}

\textit{Estimate of $J_{4}$.}
Similarly to the proof of \eqref{EstiIntegral00a},
 one has
\begin{align}\label{Esti_J4_bb}
\left|  J_4 - \sqrt{ 2 \pi } \pi_s(\varphi) \right| \leq  c \frac{t}{\sqrt{n}} \| \varphi \|_{\infty}. 
\end{align}
In view of \eqref{DecompoJ122_b}, combining \eqref{Esti_J1_J2_bb}, \eqref{Esti_J3_bb}, 
and \eqref{Esti_J4_bb}, 
we obtain the desired bound \eqref{Esti_I32_bb_02} for $I_{323}$. 
Putting together the bounds \eqref{EstimI1n00_b}, \eqref{Esti_I2n_bb}, \eqref{EstiJ11_bis}, \eqref{Esti_I321_I322_bb}
and \eqref{Esti_I32_bb_02}, 
we finish the proof of Proposition \ref{KeyPropo-02}.  
\end{proof}

\subsection{Proof of Theorem \ref{Thm-Cram-Entry_bb}} \label{Ch7_Sec_MDE_Entry}

The goal of this section is to establish Theorem \ref{Thm-Cram-Entry_bb}. 
To this aim we first prove the moderate deviation expansion in the normal range $y \in [0, o(n^{1/6})]$
for the couple $(G_n \!\cdot\! x, \log |\langle f, G_n v \rangle|)$ with a target function $\varphi$ on $G_n \!\cdot\! x$. 

\begin{theorem}\label{Thm-CramSca-02}
Assume \ref{Ch7Condi-Moment} and \ref{Ch7Condi-IP}. 
Then, as $n \to \infty$, 
uniformly in $x = \bb R v \in \bb P(V)$ and $y = \bb R f \in \bb P(V^*)$  with $\|v\| = \|f\| =1$,  
$t \in [0, o(n^{1/6})]$ and $\varphi \in \mathscr{B}_{\gamma}$, 
\begin{align*}
\qquad \quad  \frac{\bb{E}
\big[ \varphi(G_n \!\cdot\! x) \mathds{1}_{ \{ \log| \langle f,  G_n v \rangle | - n\lambda_1 \geq \sqrt{n} \sigma t \} } \big] }
{ 1-\Phi(t) }    
&  =  \nu(\varphi) +  \| \varphi \|_{\gamma}  o(1),  \nonumber\\
\frac{\bb{E}
\big[ \varphi(G_n \!\cdot\! x) \mathds{1}_{ \{ \log| \langle f,  G_n v \rangle | - n\lambda_1 \leq - \sqrt{n} \sigma t  \} } \big] }
{ \Phi(-t)  }   
&  = \nu(\varphi) +  \| \varphi \|_{\gamma}  o(1).  
\end{align*}
%
\end{theorem}


In order to establish Theorem \ref{Thm-CramSca-02}, 
we need the exponential H\"{o}lder regularity of the invariant measure $\nu$ (Lemma \ref{Lem_Regu_pi_s} with $s =0$)
and the following moderate deviation expansion for the norm cocycle 
$\sigma (G_n, x)$ proved in \cite{XGL19b}.

\begin{lemma}[\cite{XGL19b}] \label{Lem_Cramer_Cocy}
Assume \ref{Ch7Condi-Moment} and \ref{Ch7Condi-IP}. 
Then, we have, as $n \to \infty$,  
uniformly in $x\in \bb P(V)$, $t \in [0, o(\sqrt{n} )]$ and $\varphi \in \mathscr{B}_{\gamma}$, 
\begin{align*} 
\frac{\bb{E} \left[ \varphi(G_n \!\cdot\! x) \mathds{1}_{ \left\{ \sigma(G_n, x) - n \lambda_1 \geq \sqrt{n}\sigma t \right\} } \right] }
{ 1-\Phi(t) }
& =  e^{ \frac{t^3}{\sqrt{n}} \zeta( \frac{t}{\sqrt{n}} ) }
\left[ \nu(\varphi) +  \| \varphi \|_{\gamma} O\left( \frac{t+1}{\sqrt{n}} \right) \right],   \nonumber\\
\frac{\bb{E} \left[ \varphi(G_n \!\cdot\! x) \mathds{1}_{ \left\{ \sigma(G_n, x) - n \lambda_1 \leq - \sqrt{n}\sigma t  \right\} } \right] }
{ 1-\Phi(t) }
& =  e^{ -\frac{t^3}{\sqrt{n}} \zeta( -\frac{t}{\sqrt{n}} ) }
\left[ \nu(\varphi) +  \| \varphi \|_{\gamma} O\left( \frac{t+1}{\sqrt{n}} \right) \right].  
\end{align*}
\end{lemma}

%

\begin{proof}[Proof of Theorem \ref{Thm-CramSca-02}]
Without loss of generality, we assume that the target function $\varphi$ is non-negative. 
We only show the first expansion in Theorem \ref{Thm-CramSca-02}
since the proof of the second one can be carried out in a similar way.

The upper bound is a direct consequence of Lemma \ref{Lem_Cramer_Cocy}.
Specifically, since $\log |\langle f, G_n v \rangle| \leq \sigma(G_n, x)$
for $f \in V^*$ with $\|f\| = 1$, $x = \bb R v \in \bb P(V)$ with $\|v\| = 1$, 
using the first expansion in Lemma \ref{Lem_Cramer_Cocy}, 
we get that there exists a constant $c>0$ such that 
uniformly in $f \in V^*$, $x = \bb R v \in \bb P(V)$ with $\|v\| = \|f\| =1$, and $t \in [0, o(\sqrt{n} )]$, 
\begin{align}\label{Pf-Cram-Scal-Upp}
  \frac{\bb{E} \left[ \varphi(G_n \!\cdot\! x) 
  \mathds{1}_{ \left\{ \log| \langle f,  G_n v \rangle | - n\lambda_1 \geq \sqrt{n}\sigma t \right\} } \right] }
{ 1-\Phi(t)  }    
\leq   e^{ \frac{t^3}{\sqrt{n}}\zeta (\frac{t}{\sqrt{n}} ) }
\left[ \nu(\varphi) + c \| \varphi \|_{\gamma} \frac{t +1}{\sqrt{n}} \right]. 
\end{align}

The lower bound follows from Lemmas \ref{Lem_Regu_pi_s} and \ref{Lem_Cramer_Cocy}. 
By Lemma \ref{Lem_Regu_pi_s} with $s =0$, we have that for any $\varepsilon>0$, 
there exist constants $c, C >0$ such that for all $n \geq k \geq 1$, 
and $f \in V^*$ with $\|f\| = 1$, $x = \bb R v \in \bb P(V)$ with $\|v\| = 1$,
\begin{align*}
\mathbb{P} \Big( \log |\langle f, G_n v \rangle| -  \sigma(G_n, x) \leq -\ee k \Big) \leq C e^{ -c k }. 
\end{align*}
Using this inequality, we get
\begin{align}\label{Pf_Cra_ine_jjj}
& \bb{E} \left[ \varphi(G_n \!\cdot\! x) 
  \mathds{1}_{ \left\{ \log| \langle f,  G_n v \rangle | - n\lambda_1 \geq \sqrt{n}\sigma t \right\} } \right]  \nonumber\\
 & \geq  \mathbb{E} \left[ \varphi(G_n \!\cdot\! x) 
 \mathds{1}_{ \left\{ \log| \langle f,  G_n v \rangle | - n\lambda_1 \geq \sqrt{n}\sigma t \right\} }
 \mathds{1}_{ \left\{ \log |\langle f, G_n v \rangle| -  \sigma(G_n, x)  > -\ee k  \right\} }  \right]  \nonumber\\
 & \geq   \mathbb{E} \left[ \varphi(G_n \!\cdot\! x) 
 \mathds{1}_{ \left\{ \sigma (G_n, x) - n\lambda_1 \geq \sqrt{n}\sigma t + \varepsilon k  \right\} }
 \mathds{1}_{  \left\{ \log |\langle f, G_n v \rangle| -  \sigma(G_n, x) > -\ee k  \right\} }  \right]  \nonumber\\
 &  \geq   \mathbb{E} \left[ \varphi(G_n x) 
 \mathds{1}_{ \left\{ \sigma(G_n, x) - n\lambda_1 \geq \sqrt{n}\sigma t + \ee k  \right\} }  \right]
    - C e^{ -c k } \| \varphi \|_{\infty}.
\end{align}
Take $k = \floor{A \log n}$ in \eqref{Pf_Cra_ine_jjj}, 
where $A >0$ is a fixed sufficiently large constant. 
Denote $t_1 = t + \frac{\ee k}{ \sigma \sqrt{n} }$. 
By Lemma \ref{Lem_Cramer_Cocy}, we have, as $n \to \infty$, 
uniformly in $x \in \bb P(V)$, $t \in [0, \sqrt{\log n}]$ and $\varphi \in \mathscr{B}_{\gamma}$, 
\begin{align}\label{NormUpp02}
\frac{ \mathbb{E} \left[ \varphi(G_n \!\cdot\! x) 
 \mathds{1}_{ \left\{ \sigma(G_n, x) - n\lambda_1 \geq \sqrt{n}\sigma t + \varepsilon k  \right\} }  \right] }
{ 1-\Phi(t_1)  }
=   \nu(\varphi) + \| \varphi \|_{\gamma} O \left( \frac{t_1^3  + 1}{\sqrt{n}} \right).
\end{align}
We claim that uniformly in $t \in [0, \sqrt{\log n}]$, 
\begin{align}\label{Inequ_Normal_aaa}
1 > \frac{1 - \Phi(t_1) }{1 - \Phi(t) }
  = 1 - \frac{ \int_{t}^{t_1} e^{-\frac{u^2}{2}} du }{ \int_{t}^{\infty} e^{-\frac{u^2}{2}} du } 
  > 1 - c \frac{ t + 1 }{\sqrt{n}} \log n.  
\end{align}
Indeed, when $y \in [0,2]$, the inequality \eqref{Inequ_Normal_aaa} holds
due to the fact that $t_1 = t + \frac{\ee k}{ \sigma \sqrt{n} }$ and $k = \floor{A \log n}$;
when $t \in [2, \sqrt{\log n}]$, we can use the inequality
$e^{\frac{t^2}{2}} \int_{t}^{\infty} e^{-\frac{u^2}{2}} du \geq \frac{1}{t} - \frac{1}{t^3} > \frac{1}{2t}$ to get
\begin{align*}
1 - \frac{ \int_{t}^{t_1} e^{-\frac{u^2}{2}} du }{ \int_{t}^{\infty} e^{-\frac{u^2}{2}} du } 
  > 1 - \frac{ (t_1 - t) e^{- t^2/2} }{\frac{1}{2t} e^{- t^2/2}}
  > 1 - c \frac{ t + 1 }{\sqrt{n}} \log n.  
\end{align*}
Hence \eqref{Inequ_Normal_aaa} holds. 
It is easy to check that $\frac{t_1^3  + 1}{\sqrt{n}} = O( \frac{t^3 + 1}{\sqrt{n}} )$,
uniformly in $t \in [0, \sqrt{\log n}]$. 
Consequently, we get that
uniformly in $x \in \bb P(V)$, $t \in [0, \sqrt{\log n}]$ and $\varphi \in \mathscr{B}_{\gamma}$, 
\begin{align*}
\frac{ \mathbb{E} \left[ \varphi(G_n \!\cdot\! x) 
 \mathds{1}_{ \left\{ \sigma(G_n, x) - n\lambda_1 \geq \sqrt{n}\sigma t + \varepsilon k  \right\} }  \right] }
{ 1-\Phi(t)  }
\geq  \nu(\varphi) -  c \| \varphi \|_{\gamma} \frac{ t + 1 }{\sqrt{n}} \log n.
\end{align*}
This, together with \eqref{Pf_Cra_ine_jjj} and the fact that $e^{-ck}/[1 - \Phi(t)]$ decays to $0$ faster than $\frac{1}{n}$, 
implies that uniformly in $f \in V^*$ with $\|f\| = 1$, $x = \bb R v \in \bb P(V)$ with $\|v\| = 1$,
$t \in [0, \sqrt{\log n}]$ and $\varphi \in \mathscr{B}_{\gamma}$, 
\begin{align} \label{Pf_Cram_Low_logn}
\frac{\bb{E} \left[ \varphi(G_n \!\cdot\! x) 
  \mathds{1}_{ \left\{ \log| \langle f,  G_n v \rangle | - n\lambda_1 \geq \sqrt{n}\sigma t  \right\} } \right] }
{ 1-\Phi(t)  }   
\geq   \nu(\varphi) - c \| \varphi \|_{\gamma} \frac{t + 1}{\sqrt{n}} \log n. 
\end{align}

It remains to prove the lower bound when $t \in [\sqrt{\log n}, o(n^{1/6})]$. 
We take $k = \floor{A t^2}$ in \eqref{Pf_Cra_ine_jjj}, where $A >0$ is a fixed sufficiently large constant. 
In the same way as in \eqref{NormUpp02}, we get that, with $t_1 = t + \frac{\ee k}{ \sigma \sqrt{n} }$, 
uniformly in $x \in \bb P(V)$, $t \in [\sqrt{\log n}, o(n^{1/6})]$ and $\varphi \in \mathscr{B}_{\gamma}$, 
\begin{align}\label{NormUpp_yy}
\frac{ \mathbb{E} \left[ \varphi(G_n \!\cdot\! x) 
 \mathds{1}_{ \left\{ \sigma(G_n, x) - n\lambda_1 \geq \sqrt{n}\sigma t + \varepsilon k  \right\} }  \right] }
{ 1-\Phi(t_1)  }
=   \nu(\varphi) + \| \varphi \|_{\gamma} O \left( \frac{t_1^3  + 1}{\sqrt{n}} \right).
\end{align}
Using the inequality 
$\frac{1}{t} \geq e^{\frac{t^2}{2}} \int_{t}^{\infty} e^{-\frac{u^2}{2}} du \geq \frac{1}{t} - \frac{1}{t^3}$
for $t > 0$, by elementary calculations, 
we get that uniformly in $t \in [\sqrt{\log n}, o(n^{1/6})]$, 
\begin{align}\label{Rate-Phi-t}
1 > \frac{1 - \Phi(t_1) }{1 - \Phi(t) }
  \geq  \frac{t}{t_1} \Big( 1 - \frac{1}{t_1^2} \Big) e^{ \frac{t^2}{2} - \frac{t_1^2}{2} }
  > \left( 1 - \frac{c}{t^2} \right) \left( 1 - \frac{c t^3}{\sqrt{n}} \right).  
\end{align}
Taking into account that $\frac{t_1^3  + 1}{\sqrt{n}} = O( \frac{t^3  + 1}{\sqrt{n}} )$
and that $e^{-c A t^2}/[1 - \Phi(t)]$ decays to $0$ faster than $\frac{1}{n}$ (by taking $A>0$ to be sufficiently large),
from \eqref{Pf_Cra_ine_jjj}, \eqref{NormUpp_yy} and \eqref{Rate-Phi-t}
we deduce that uniformly in $f \in V^*$ with $\|f\| = 1$, $x = \bb R v \in \bb P(V)$ with $\|v\| = 1$,
$t \in [\sqrt{\log n}, o(n^{1/6})]$ and $\varphi \in \mathscr{B}_{\gamma}$, 
\begin{align} \label{Pf_Cram_Low_n16}
\frac{\bb{E} \left[ \varphi(G_n \!\cdot\! x) 
  \mathds{1}_{ \left\{ \log| \langle f,  G_n v \rangle | - n\lambda_1 \geq \sqrt{n}\sigma t  \right\} } \right] }
{ 1-\Phi(t)  }   
\geq   \nu(\varphi) - c \| \varphi \|_{\gamma} \bigg( \frac{1}{t^2} + \frac{t^3}{ \sqrt{n} } \bigg). 
\end{align}
Combining \eqref{Pf-Cram-Scal-Upp}, \eqref{Pf_Cram_Low_logn} and \eqref{Pf_Cram_Low_n16} 
finishes the proof of Theorem \ref{Thm-CramSca-02}. 
\end{proof}

%
%
%
Now we apply Propositions \ref{KeyPropo} and \ref{KeyPropo-02}
to establish the following moderate deviation expansion when $t \in [n^{\alpha}, o(n^{1/2})]$ for any $\alpha \in (0, 1/2)$. 

\begin{theorem}\label{Thm-Cram-Scalar-tag}
Assume \ref{Ch7Condi-Moment} and \ref{Ch7Condi-IP}. 
Then, for any $\varphi \in \mathscr{B}_{\gamma}$ and $\alpha \in (0, 1/2)$, as $n \to \infty$, we have, 
uniformly in $x = \bb R v \in \bb P(V)$ and $y = \bb R f \in \bb P(V^*)$ with $\|v\| = \|f\| =1$, 
and $t \in [n^{\alpha}, o(\sqrt{n} )]$,  
\begin{align}
\frac{\bb{E} \left[ \varphi(G_n \!\cdot\! x) 
\mathds{1}_{ \left\{ \log| \langle f,  G_n v \rangle | - n\lambda_1 \geq \sqrt{n} \sigma t \right\} } \right] }
{ 1-\Phi(t) }    
& =  e^{ \frac{t^3}{\sqrt{n}}\zeta (\frac{t}{\sqrt{n}} ) }
\Big[ \nu(\varphi) +   o(1) \Big],  \label{Cramer-Coeffi-upper} \\
\frac{\bb{E}  \left[ \varphi(G_n \!\cdot\! x) 
\mathds{1}_{ \left\{ \log| \langle f,  G_n v \rangle | - n\lambda_1 \leq - \sqrt{n} \sigma t \right\} } \right] }
{ \Phi(-t)  }   
& =  e^{ - \frac{t^3}{\sqrt{n}}\zeta (-\frac{t}{\sqrt{n}} ) }
\Big[ \nu(\varphi) +  o(1) \Big].   \label{Cramer-Coeffi-lower}
\end{align}
\end{theorem}

\begin{proof} 
As in Section \ref{sec-proof of Edgeworth exp} we define a partition of the unity. 
Let $U$ be the uniform distribution function on the interval $[0,1]$. 
 Let $a \in (0,\frac{1}{2}]$ be a constant.
For any integer $k\geq 0$, define 
\begin{align*} 
U_{k}(t)= U\left(\frac{t-(k-1) a}{a}\right),  \qquad 
h_{k}(t)=U_{k}(t) - U_{k+1}(t),  \quad  t \in \bb R. 
\end{align*}
Note that  $U_{m} = \sum_{k=m}^\infty h_{k}$ for $m\geq 0$ and for any $t\geq 0$ and $m\geq0$,
\begin{align} \label{unity decomposition h-001-222}
\sum_{k=0}^{\infty} h_{k} (t) =1, \quad \sum_{k=0}^{m} h_{k} (t) + U_{m+1} (t) =1.
\end{align}
For any $x=\bb R v \in \bb P(V)$ and $y=\bb R f \in \bb P(V^*)$, 
set 
\begin{align}\label{Def-chi-nk-222}
\chi_{k}^y(x)=h_{k}(-\log \delta(y, x))  \quad  \mbox{and}  \quad 
\overline \chi_{k}^y(x)= U_{k} ( -\log \delta(y, x) ). 
\end{align}
From \eqref{unity decomposition h-001-222}
 we have the following partition of the unity  on $\bb P(V)$: for any $x\in \bb P (V)$, $y \in \bb P(V^*)$ and $m\geq 0$,
\begin{align} \label{Unit-partition001-222}
\sum_{k=0}^{\infty} \chi_{k}^y (x) =1, \quad 
\sum_{k=0}^{m} \chi_{k}^y (x) + \overline \chi_{m+1}^y (x) =1.
\end{align}
Denote by $\supp (\chi_{k}^y)$ the support of the function $\chi_{k}^y$.  
It is easy to see that for any $k\geq 0$ and $y\in \bb P(V^*)$,
\begin{align} \label{on the support on chi_k-001-222}
 -\log \delta(y, x) \in [a (k-1), a (k+1)] \quad \mbox{for any}\ x\in \supp (\chi_{k}^y). 
\end{align}
%
As in Lemma \ref{lemmaHolder property001} we show that
there exists a constant $c>0$ such that 
for any $\gamma\in(0,1]$, 
 $k\geq 0$ and $y\in \bb P(V^*)$, it holds
  $\chi_{k}^y\in \scr B_{\gamma}$ and 
\begin{align} \label{Holder prop ohCHI_k-001-222}
\| \chi_{k}^y \|_{\gamma} \leq \frac{c e^{\gamma k a}}{a^\gamma}.
\end{align}

Without loss of generality, we assume that the target function $\varphi$ is non-negative. 
We first establish \eqref{Cramer-Coeffi-upper}.
Since the upper bound of the moderate deviation expansion 
for the couple $(G_n \!\cdot\! x, \log |\langle f, G_n v \rangle|)$ has been shown in \eqref{Pf-Cram-Scal-Upp},
 it remains to establish the following lower bound: 
 for any $\varphi \in \mathscr{B}_{\gamma}$ and $\alpha \in (0, 1/2)$,  
uniformly in $x = \bb R v \in \bb P(V)$ and $y = \bb R f \in \bb P(V^*)$ with $\|v\| = \|f\| =1$, 
and $t \in [n^{\alpha}, o(\sqrt{n} )]$,  
\begin{align}\label{ScalLowerBound}
\liminf_{n \to \infty}  
\frac{\mathbb{E} \left[ \varphi(G_n \!\cdot\! x) 
  \mathds{1}_{ \left\{ \log| \langle f,  G_n v \rangle | - n \lambda_1 \geq \sqrt{n} \sigma t  \right\} } \right] }
{ e^{ \frac{t^3}{\sqrt{n}} \zeta(\frac{t}{\sqrt{n}}) } [1-\Phi(t)] }   
\geq  \nu(\varphi). 
\end{align}
Now we are going to prove \eqref{ScalLowerBound}. 
From the change of measure formula \eqref{Ch7basic equ1}
and the fact $\lambda_1 = \Lambda'(0)$, we get 
\begin{align} 
A_n: &  =  \bb{E} \left[ \varphi(G_n \!\cdot\! x) 
  \mathds{1}_{ \left\{ \log| \langle f,  G_n v \rangle | - n \lambda_1 \geq \sqrt{n}\sigma t \right\} } \right]
 \label{ChanMeaScal}\\
& =  r_s(x) \kappa^{n}(s) \bb{E}_{\bb{Q}^{x}_{s}}
\left[ (\varphi r_s^{-1})(G_n \!\cdot\! x) 
e^{-s \sigma (G_n, x) }\mathds{1}_{ \left\{ \log |\langle f, G_n v \rangle | \geq n \Lambda'(0) + \sqrt{n} \sigma t  \right\} }  
\right].    \nonumber
\end{align}
For brevity, we denote 
\begin{align*}
T_n^x = \sigma (G_n, x) - n\Lambda'(s),  \qquad   Y_n^{x,y}: = \log \delta(y, G_n \!\cdot\! x). 
\end{align*}
Choosing $s>0$ as the solution of the equation \eqref{SaddleEqua}, from \eqref{ChanMeaScal} and \eqref{Pf_LLN_Equality} it follows that 
\begin{align*}
A_n = r_s(x) e^{-n [ s \Lambda'(s) - \Lambda(s) ] }  
   \bb{E}_{\bb{Q}^{x}_{s}}
 \left[ (\varphi r_s^{-1})(G_n \!\cdot\! x) e^{-s T_n^x }
   \mathds{1}_{ \left\{ T_n^x + Y_n^{x,y} \geq 0 \right\} } \right]. 
\end{align*}
Using \eqref{SaddleEqua}, one can verify that 
\begin{align} \label{LambdaQ01}
s \Lambda'(s) - \Lambda(s) 
= \frac{t^2}{2n} - \frac{t^3}{n^{3/2} } \zeta \Big( \frac{t}{\sqrt{n}} \Big), 
\end{align}
where $\zeta$ is the Cram\'{e}r series defined by \eqref{Ch7Def-CramSeri}. 
Thus $A_n$ can be rewritten as  
\begin{align}\label{An-lambdaq}  
A_n = r_s(x) e^{ -\frac{t^2}{2} + \frac{t^3}{ \sqrt{n} } \zeta(\frac{t}{\sqrt{n}}) }
   \bb{E}_{\bb{Q}^{x}_{s}}
\left[ (\varphi r_s^{-1})(G_n \!\cdot\! x)  e^{-s T_n^x }
\mathds{1}_{ \left\{ T_n^x + Y_n^{x,y} \geq 0 \right\} } \right]. 
\end{align}
Since the functions $\varphi$ and $r_s$ are positive, 
using the partition of the unity \eqref{Unit-partition001-222} and \eqref{on the support on chi_k-001-222}, 
we have that for $M_n = \floor{\log n}$ and $a \in (0, 1/2)$, 
\begin{align}\label{Ch7SmoothIne Holder 01}
A_{n} & \geq   r_s(x) e^{ -\frac{t^2}{2} + \frac{t^3}{ \sqrt{n} } \zeta(\frac{t}{\sqrt{n}}) }
  \sum_{k = 0}^{ M_n } \bb{E}_{\bb{Q}^{x}_{s}}
\left[ (\varphi r_s^{-1} \chi_{k}^y)(G_n \!\cdot\! x)  e^{-s T_n^x }
\mathds{1}_{\{ T_n^x + Y_n^{x,y} \geq 0 \}}  \right]  \nonumber\\
& \geq   r_s(x) e^{ -\frac{t^2}{2} + \frac{t^3}{ \sqrt{n} } \zeta(\frac{t}{\sqrt{n}}) }
  \sum_{k = 0}^{M_n} \bb{E}_{\bb{Q}^{x}_{s}}
\left[ (\varphi r_s^{-1} \chi_{k}^y)(G_n \!\cdot\! x) 
e^{-s T_n^x } \mathds{1}_{\{ T_n^x -  a (k+1) \geq 0 \}} \right].  
\end{align}
Let
\begin{align*}
\varphi_{s,k}^y(x) = (\varphi r_s^{-1} \chi_{k}^y )(x), 
\quad  x \in \bb P(V), 
\end{align*}
and
\begin{align*}
\psi_s(w) = e^{-s w} \mathds{1}_{ \{ w \geq 0 \} },  \quad  w \in \bb{R}. 
\end{align*}
It then follows from \eqref{Ch7SmoothIne Holder 01} that 
\begin{align*} 
A_{n} 
& \geq   r_s(x) e^{ -\frac{t^2}{2} + \frac{t^3}{ \sqrt{n} } \zeta(\frac{t}{\sqrt{n}}) }
  \sum_{k = 0}^{M_n} e^{-s a (k+1)} \bb{E}_{\bb{Q}^{x}_{s}}
\left[ \varphi_{s,k}^y(G_n \!\cdot\! x) \psi_s(T_n^x - a (k+1) )  \right],    
\end{align*}
which implies that
\begin{align*}
& \frac{ A_n }{ e^{ \frac{t^3}{\sqrt{n}} \zeta(\frac{t}{\sqrt{n}}) } [1-\Phi(t)] }   
 \geq  r_s(x) \sum_{k = 0}^{ M_n }   
\frac{  e^{-s a (k+1)}  \bb{E}_{\bb{Q}^{x}_{s}}
\big[ \varphi_{s,k}^y(G_n \!\cdot\! x) \psi_s(T_n^x - a (k+1))  \big] }
   { e^{ \frac{t^2}{2} } [1-\Phi(t)] }. 
\end{align*}
From Lemma \ref{estimate u convo}, it follows that for any small constant $\ee >0$, 
\begin{align*}
\psi_s(w) \geq 
\psi^-_{s,\ee} * \rho_{\ee^2}(w) - 
\int_{ |u| \geq \ee } \psi^-_{s,\varepsilon}( w-u ) \rho_{\varepsilon^2}(u) du, \quad w \in \bb R,
\end{align*}
where $\psi^-_{s,\ee}$ is given by \eqref{Def-psi-see}. 
Using this inequality we get
\begin{align}\label{LowerScalarACD}
\frac{ A_n }{ e^{ \frac{t^3}{\sqrt{n}} \zeta(\frac{t}{\sqrt{n}}) } [1-\Phi(t)] } 
  \geq  r_s(x) \sum_{k = 0}^{ M_n }
   \frac{ B_{n,k} - D_{n,k} }{ e^{ \frac{t^2}{2} } [1-\Phi(t)] },   
\end{align}
where 
\begin{align*}
&  B_{n,k} =  e^{-s a (k+1)} \mathbb{E}_{\mathbb{Q}^{x}_{s}}
  \Big[ \varphi_{s,k}^y (G_n \!\cdot\! x)
  ( \psi^-_{s,\ee} * \rho_{\ee^2}) (T_{n}^x - a (k+1)  ) \Big],    \nonumber \\
&  D_{n,k} =  e^{-s a (k+1)} \int_{|w| \geq \ee} 
    \mathbb{E}_{\mathbb{Q}^{x}_{s}}
  \Big[ \varphi_{s,k}^y (G_n \!\cdot\! x) \psi^-_{s,\ee}(T_n^x - a (k+1) - w) \Big] 
    \rho_{\ee^2}(w) dw. 
\end{align*}
Since $B_{n,k} \geq D_{n,k}$ and $\sup_{ x \in \bb P(V) } |r_s (x) - 1| \to 0$ as $n \to \infty$, 
by Fatou's lemma, we get
\begin{align*}
\liminf_{n \to \infty}
\frac{ A_n }{ e^{ \frac{t^3}{\sqrt{n}} \zeta(\frac{t}{\sqrt{n}}) } [1-\Phi(t)] }
\geq \sum_{k = 0}^{\infty} \liminf_{n \to \infty} 
   \frac{ B_{n,k} - D_{n,k} }{ e^{ \frac{t^2}{2} } [1-\Phi(t)] }  \mathds{1}_{ \left\{ k \leq M_n \right\}}. 
\end{align*}

\textit{Estimate of $B_{n,k}$.}
Since the function $\widehat\rho_{\ee^{2}}$ is integrable on $\bb{R}$, 
by the Fourier inversion formula, we have 
\begin{align*}
{\psi}^-_{s,\ee}\!\ast\!\rho_{\ee^{2}}(w)
= \frac{1}{2\pi} \int_{\bb{R}}e^{iuw} \widehat {\psi}^-_{s,\ee}(u) \widehat\rho_{\ee^{2}} (u) du, 
\quad  w \in \bb R. 
\end{align*}
Substituting $w = T_{n}^x - a (k+1)$, 
taking expectation with respect to $\bb{E}_{\bb{Q}_{s}^{x}},$
and using Fubini's theorem, we get
\begin{align}\label{IdenScalCn}
B_{n,k} = \frac{1}{2 \pi }  e^{-s a (k+1)}
\int_{\bb{R}} e^{-i u a (k+1)}   R^{n}_{s, iu}(\varphi_{s,k}^y)(x)
\widehat {\psi}^-_{s,\ee}(u) \widehat\rho_{\ee^{2}}(u) du, 
\end{align}
where
\begin{align*}
R^{n}_{s, iu}(\varphi_{s,k}^y)(x)
= \bb{E}_{\bb{Q}_{s}^{x}} \left[ \varphi_{s,k}^y(G_n \!\cdot\! x) e^{iu T_{n}^x} \right],  \quad x \in \bb P(V). 
\end{align*}
Applying Proposition \ref{KeyPropo} with $l=\frac{ a (k+1) }{n}$ and $\varphi = \varphi_{s,k}^y$, 
by simple calculations we deduce that uniformly in $x \in \bb P(V)$, $y \in \bb P(V^*)$, $s \in (0, s_0)$, 
$0 \leq k \leq M_n$ and $t \in [n^{\alpha}, o(\sqrt{n} )]$, 
\begin{align*}
&   \int_{\bb{R}} e^{-i u a (k+1)}   R^{n}_{s, iu}(\varphi_{s,k}^y)(x)
\widehat {\psi}^-_{s,\ee}(u) \widehat\rho_{\ee^{2}}(u) du 
  -  \pi_s(\varphi_{s,k}^y)  \frac{ \sqrt{ 2 \pi } }{ s \sigma_s \sqrt{n}}  \nonumber\\
&  \geq  - C \left( \frac{t}{sn} + \frac{1}{s t^2 \sqrt{n}} \right) \| \varphi_{s,k}^y \|_{\infty}
      -  C \frac{\log n}{sn} \| \varphi_{s,k}^y \|_{\gamma}. 
\end{align*}
Note that $\sigma_s = \sqrt{ \Lambda''(s) } = \sigma [ 1 + O(s)]$. 
From \eqref{root s 1}, we have $\frac{t}{ s \sigma \sqrt{n} } =  1 + O(s)$, 
and thus $\frac{ t }{ s \sigma_s \sqrt{n} }  = 1 + O(s).$  
Using the inequality $ \sqrt{ 2 \pi } t e^{\frac{t^2}{2}} [ 1-\Phi(t) ] \leq 1$ for any $t >0$, 
it follows that 
\begin{align}\label{LowerScalCn}
&   \frac{\int_{\bb{R}} e^{-i u a (k+1)}   R^{n}_{s, iu}(\varphi_{s,k}^y)(x)
\widehat {\psi}^-_{s,\ee}(u) \widehat\rho_{\ee^{2}}(u) du }
{ e^{\frac{t^2}{2}} [ 1- \Phi(t) ]  } 
  -    \pi_s(\varphi_{s,k}^y)  \frac{ 2 \pi t }{ s \sigma_s \sqrt{n} }   \nonumber\\
&  \geq  - C \left( \frac{t^2}{sn} + \frac{1}{s t \sqrt{n}} \right) \| \varphi_{s,k}^y \|_{\infty}
      -  C \frac{t \log n}{sn} \| \varphi_{s,k}^y \|_{\gamma}   \nonumber\\
&  \geq  - C \left( \frac{t}{\sqrt{n}} + \frac{1}{t^2} \right) \| \varphi_{s,k}^y \|_{\infty}
      -  C \frac{t \log n}{sn} \| \varphi_{s,k}^y \|_{\gamma}. 
\end{align}
Using the construction of the function $\varphi_{s,k}^y$ and  \eqref{Holder prop ohCHI_k-001-222}
 give that uniformly in $s \in (0, s_0)$, $w \in \bb P(V^*)$, 
$\varphi \in \mathscr{B}_{\gamma}$ and $0 \leq k \leq M_n$, 
\begin{align}\label{Bound_varphi_infty_norm}
\| \varphi_{s,k}^y  \|_{\infty}  \leq c \| \varphi \|_{\infty}
\end{align}
and 
\begin{align}\label{Bound_varphi_Holder_norm}
\| \varphi_{s,k}^y  \|_{\gamma} 
& \leq c \|  \varphi  \|_{\gamma} + c  \frac{e^{ \gamma k a}}{a^{\gamma}} \|\varphi\|_{\infty}.   
\end{align}
Recalling that $t \in [n^{\alpha}, o(\sqrt{n} )]$ and 
 taking $\gamma>0$ sufficiently small such that $ \frac{e^{ \gamma k a}}{a^{\gamma}} < n^{\alpha/2}$,  
 from \eqref{LowerScalCn} we deduce that 
\begin{align*}
\lim_{n \to \infty} \frac{\int_{\bb{R}} e^{-i u a (k+1)}   R^{n}_{s, iu}(\varphi_{s,k}^y)(x)
\widehat {\psi}^-_{s,\ee}(u) \widehat\rho_{\ee^{2}}(u) du }
{ e^{\frac{t^2}{2}} [ 1- \Phi(t) ]  } 
  \geq    \nu(\varphi_{0,k}^y) \times 2 \pi.  
\end{align*}
As $s=o(1)$ as $n \to \infty$ for any fixed $k \geq 1$ and $a\in (0,\frac{1}{2})$, 
it holds that $\lim_{n \to \infty}e^{-s a (k+1)} = 1$. 
Then, in view of \eqref{IdenScalCn}, it follows that 
\begin{align*}
\lim_{n \to \infty} \frac{ B_{n,k}}{ e^{ \frac{t^2}{2} } [1- \Phi(t)] }
= \nu(\varphi_{0,k}^y). 
\end{align*}
This, together with \eqref{Unit-partition001}, implies the desired lower bound:  
\begin{align}\label{Pf_Lower_bound_B_nk}
 \sum_{k = 0}^{\infty}  \liminf_{n \to \infty} \frac{ B_{n,k}}{ e^{ \frac{t^2}{2} } [1 - \Phi(t)] }
   \mathds{1}_{ \left\{ k \leq M_n \right\}} 
\geq  \sum_{k = 0}^{\infty}  \nu(\varphi_{0,k}^y) = \nu(\varphi). 
\end{align}

\textit{Estimate of $D_{n,k}$.}
We shall apply Fatou's lemma to provide an upper bound for $D_{n,k}$.
An important issue here is to find a dominating function, which is possible due to the integrability of 
the density function $\rho_{\ee^2}$ on the real line. 
More specifically, in the same way as in the proof of \eqref{IdenScalCn}, 
we use the Fourier inversion formula and the Fubini theorem to get 
\begin{align*}
D_{n,k} &  = \frac{1}{2 \pi }  e^{-s a (k+1)}  \nonumber\\
 & \quad \times \int_{|w| \geq \ee}  \left[ \int_{\bb{R}} e^{-i u (a (k+1) + w)}   R^{n}_{s, iu}(\varphi_{s,k}^y)(x)
  \widehat {\psi}^-_{s,\ee}(u) \widehat\rho_{\ee^{2}}(u) du  \right]
 \rho_{\ee^2}(w) dw. 
\end{align*}
We decompose the integral in $D_{n,k}$ into two parts:
\begin{align}\label{Pf_Decom_D_nk}
\frac{ D_{n,k} }{ e^{\frac{t^2}{2}} [1 - \Phi(t) ] }  
& =  \frac{1}{2 \pi }  e^{-s a (k+1)}  \left\{ \int_{ \ee \leq |w| < \sqrt{n}} + \int_{|w| \geq \sqrt{n}}  \right\}  \nonumber\\
&  \quad \times  
  \frac{ \int_{\bb{R}} e^{-i u (a (k+1) + w)}   R^{n}_{s, iu}(\varphi_{s,k}^y)(x)
  \widehat {\psi}^-_{s,\ee}(u) \widehat\rho_{\ee^{2}}(u) du  }
   { e^{\frac{t^2}{2}} [1 - \Phi(t) ] } 
    \rho_{\ee^2}(w) dw     \nonumber\\
& =: E_{n,k} + F_{n,k}. 
\end{align}

\textit{Estimate of $E_{n,k}$.}
Since $k \leq M_n$ and $|w| \leq \sqrt{n}$, we have $|l|: = |a (k+1) + w|/n = O(\frac{1}{\sqrt{n}})$. 
Note that $\frac{ t }{ s \sigma_s \sqrt{n} }  = 1 + O(s)$ and  
$e^{\frac{t^2}{2}} [1 - \Phi(t)] \geq \frac{1}{\sqrt{2 \pi}} (\frac{1}{t} - \frac{1}{t^3})$, $t>1$. 
Applying Proposition \ref{KeyPropo} with $l = (a (k+1) + w)/n$,
we get that uniformly in $\ee \leq |w| < \sqrt{n}$,
\begin{align*}
\lim_{n \to \infty} \frac{ \int_{\bb{R}} e^{-i u (a (k+1) + w)}   R^{n}_{s, iu}(\varphi_{s,k}^y)(x)
  \widehat {\psi}^-_{s,\ee}(u) \widehat\rho_{\ee^{2}}(u) du  }
   { e^{\frac{t^2}{2}} [1 - \Phi(t) ] }   
 = 2 \pi \times \pi_s(\varphi_{0,k}^y).   
\end{align*}
As above, for any fixed $k \geq 1$ and $a\in (0,\frac{1}{2})$, 
we have $\lim_{n \to \infty}e^{-s a (k+1)} = 1$. 
Since the function $\rho_{\ee^2}$ is integrable on $\bb R$, 
by the Lebesgue dominated convergence theorem, we obtain that there exists a constant $c>0$ such that
\begin{align*}
\lim_{n \to \infty} E_{n,k} \mathds{1}_{ \left\{ k \leq M_n \right\}}
=  \nu(\varphi_{0,k}^y) \int_{ |w| \geq \ee } \rho_{\ee^2}(w) dw 
\leq c \ee \nu(\varphi_{0,k}^y). 
\end{align*}
This, together with \eqref{Unit-partition001},  implies that
\begin{align}\label{Pf_E_nk1_limit}
\sum_{k = 0}^{\infty} \lim_{n \to \infty} E_{n,k} \mathds{1}_{ \left\{ k \leq M_n \right\}}
\leq  c \ee  \sum_{k = 0}^{\infty}  \nu (\varphi_{0,k}^y)
\leq  c \ee \nu(\varphi). 
\end{align}

\textit{Estimate of $E_{n,k,2}$.}
Notice that there exists a constant $c>0$ such that for any $n \geq 1$, 
\begin{align*}
\left| \int_{\bb{R}} e^{-i u (a (k+1) + w)}   R^{n}_{s, iu}(\varphi_{s,k}^y)(x)
  \widehat {\psi}^-_{s,\ee}(u) \widehat\rho_{\ee^{2}}(u) du \right|
& \leq \| \varphi_{s,k}^y \|_{\infty} \int_{\bb{R}} \widehat\rho_{\ee^{2}}(u) du  \nonumber\\ 
& \leq  c  \| \varphi_{s,k}^y \|_{\infty}. 
\end{align*}
Using the fact that $t = o(\sqrt{n})$ and $\rho_{\ee^{2}}(u) \leq \frac{c_{\ee}}{u^2}$, and again the inequality 
$e^{\frac{t^2}{2}} [1 - \Phi(t)] \geq \frac{1}{\sqrt{2 \pi}} (\frac{1}{t} - \frac{1}{t^3})$, $t>1$,
we get
\begin{align*}
\limsup_{n \to \infty} F_{n,k} \mathds{1}_{ \left\{ k \leq M_n \right\}}
\leq c_{\ee} \limsup_{n \to \infty} t  \int_{|w| \geq \sqrt{n}}  \rho_{\ee^2}(w) dw 
\leq c_{\ee} \lim_{n \to \infty} \frac{t}{\sqrt{n}} = 0.
\end{align*}
Hence 
\begin{align}\label{Pf_E_nk2_limit}
\sum_{k = 0}^{\infty} \limsup_{n \to \infty} F_{n,k} \mathds{1}_{ \left\{ k \leq M_n \right\}}
= 0. 
\end{align}
Since $\ee >0$ can be arbitrary small, 
combining \eqref{Pf_Decom_D_nk}, \eqref{Pf_E_nk1_limit} and \eqref{Pf_E_nk2_limit},
we get the desired bound for $D_{n,k}$:
\begin{align}\label{Pf_Upper_bound_D_nk}
\sum_{k = 0}^{\infty} \limsup_{n \to \infty} 
   \frac{ D_{n,k} }{ e^{ \frac{t^2}{2} } [1-\Phi(t)] }  \mathds{1}_{ \left\{ k \leq M_n \right\}} = 0. 
\end{align}
Putting together \eqref{Pf_Lower_bound_B_nk} and \eqref{Pf_Upper_bound_D_nk},
we conclude the proof of \eqref{ScalLowerBound} as well as the first expansion \eqref{Cramer-Coeffi-upper}. 

The proof of the second expansion \eqref{Cramer-Coeffi-lower}
can be carried out in a similar way. 
As in \eqref{ChanMeaScal} and \eqref{An-lambdaq}, using \eqref{Unit-partition001-222}, we have
\begin{align*}
& \bb{E} \left[ \varphi(G_n \!\cdot\! x) 
  \mathds{1}_{ \left\{ \log| \langle f,  G_n v \rangle | - n \lambda_1 \leq  - \sqrt{n}\sigma t \right\} } \right]   \notag\\
 & =   r_s(x) e^{ -\frac{t^2}{2} - \frac{t^3}{ \sqrt{n} } \zeta(-\frac{t}{\sqrt{n}}) }
  \sum_{k = 0}^{ M_n } \bb{E}_{\bb{Q}^{x}_{s}}
\left[ (\varphi r_s^{-1} \chi_{k}^y)(G_n \!\cdot\! x)  e^{-s T_n^x }
\mathds{1}_{\{ T_n^x + Y_n^{x,y} \leq 0 \}}  \right]  \nonumber\\
& \quad +  r_s(x) e^{ -\frac{t^2}{2} - \frac{t^3}{ \sqrt{n} } \zeta(-\frac{t}{\sqrt{n}}) }
  \bb{E}_{\bb{Q}^{x}_{s}}
\left[ (\varphi r_s^{-1} \overline \chi_{M_n + 1}^y)(G_n \!\cdot\! x)  e^{-s T_n^x }
\mathds{1}_{\{ T_n^x + Y_n^{x,y} \leq 0 \}}  \right]  \notag\\
& = : A_n' +  A_n'', 
\end{align*}
where this time we choose $M_n = \floor[]{A \log n}$ with $A>0$, 
and $s < 0$ satisfies the equation \eqref{SaddleEqua-bis}. 
The main difference for handling the first term consists in using Proposition \ref{KeyPropo-02} instead of Proposition \ref{KeyPropo}. 
For second term,  we have 
\begin{align}\label{Inequa-Qsx-001}
& \bb{E}_{\bb{Q}^{x}_{s}}
\left[ (\varphi r_s^{-1} \overline \chi_{M_n + 1}^y)(G_n \!\cdot\! x)  e^{-s T_n^x }
\mathds{1}_{\{ T_n^x + Y_n^{x,y} \leq 0 \}}  \right]  \nonumber\\
& \leq  c \|\varphi\|_{\infty} \bb{E}_{\bb{Q}^{x}_{s}}
\left[  \overline \chi_{M_n + 1}^y(G_n \!\cdot\! x)  e^{ s Y_n^{x,y} }  \right]   \nonumber\\
& \leq  c \|\varphi\|_{\infty}  e^{-s a A \log n} \bb{Q}^{x}_{s} \left( - \log \delta(y, G_n \cdot x) \geq a A \log n \right)  \notag\\
& \leq  \frac{c}{n}  \|\varphi\|_{\infty},  
\end{align}
where in the last inequality we use Lemma \ref{Lem_Regu_pi_s} and choose $A$ large enough. 
Using \eqref{Inequa-Qsx-001}, the inequality $t e^{\frac{t^2}{2}} \Phi(-t) \geq \frac{1}{\sqrt{2 \pi}}$ for all $t >0$ and the fact that $t = o(\sqrt{n})$, we get
\begin{align*}
\frac{A_n''}{ e^{  - \frac{t^3}{ \sqrt{n} } \zeta(-\frac{t}{\sqrt{n}}) }  \Phi(-t) } 
\leq  \frac{c}{n}  \|\varphi\|_{\infty} \frac{1}{ e^{\frac{t^2}{2}} \Phi(-t) }  
\leq \frac{ct}{n}  \|\varphi\|_{\infty}  \leq \frac{c}{\sqrt{n}}  \|\varphi\|_{\infty}. 
\end{align*}
This finishes the proof of the expansion \eqref{Cramer-Coeffi-lower}. 
\end{proof}

\begin{proof}[Proof of Theorem \ref{Thm-Cram-Entry_bb}]
Theorem \ref{Thm-Cram-Entry_bb} follows from Theorems \ref{Thm-CramSca-02} and \ref{Thm-Cram-Scalar-tag}. 
\end{proof}

\section{Proof of the local limit theorem with moderate deviations} 

The goal of this section is to establish Theorem \ref{ThmLocal02} on the local limit theorems with moderate deviations
for the coefficients $\langle f, G_n v \rangle$.



The following result which is proved in \cite{XGL19b} will be used to prove Theorem \ref{ThmLocal02}. 
Assume that $\psi: \mathbb R \mapsto \mathbb C$
is a continuous function with compact support in $\mathbb{R}$, 
and that $\psi$ is differentiable in a small neighborhood of $0$ on the real line.

\begin{lemma}[\cite{XGL19b}] \label{Lem_R_st_limit}
Assume conditions \ref{Ch7Condi-Moment} and \ref{Ch7Condi-IP}. 
Then, there exist constants $s_0, \delta, c, C >0$ such that for all $s \in (-s_0, s_0)$, 
$x \in \bb P(V)$, $|l|\leq \frac{1}{\sqrt{n}}$, $\varphi \in \mathscr{B}_{\gamma}$ and $n \geq 1$, 
\begin{align} \label{Ine_R_st_limit}
&  \left| \sigma_s \sqrt{n}  \,  e^{ \frac{n l^2}{2 \sigma_s^2} }
\int_{\mathbb R} e^{-it l n} R^{n}_{s, iu}(\varphi)(x) \psi (t) dt
  - \sqrt{2\pi} \pi_s(\varphi) \psi(0) \right|   \nonumber\\
& \leq  \frac{ C }{ \sqrt{n} } \| \varphi \|_\gamma 
  + \frac{C}{n} \|\varphi\|_{\gamma} \sup_{|t| \leq \delta} \big( |\psi(t)| + |\psi'(t)| \big)
  + Ce^{-cn} \|\varphi\|_{\gamma} \int_{\bb R} |\psi(t)| dt. 
\end{align}
\end{lemma}

We also need the result below, which is proved in \cite[Lemma 6.2]{XGL19d}. 

\begin{lemma}[\cite{XGL19d}] \label{Lem_Inte_Regu_a}
Assume \ref{Ch7Condi-Moment} and \ref{Ch7Condi-IP}. 
Let $p>0$ be any fixed constant.
Then, there exists a constant $s_0 > 0$ such that
\begin{align*}
\sup_{n \geq 1} \sup_{ s\in (-s_0, s_0) } \sup_{y \in \bb P(V^*) }  
\sup_{x \in \bb P(V) } 
\bb E_{\bb Q_s^x} \left( \frac{1}{ \delta(y, G_n \!\cdot\! x)^{p|s|} } \right) < + \infty. 
\end{align*}
\end{lemma}

Using Lemmas \ref{Lem_Regu_pi_s},  \ref{Lem_R_st_limit} and \ref{Lem_Inte_Regu_a},
we are equipped to prove Theorem \ref{ThmLocal02}. 

\begin{proof}[Proof of Theorem \ref{ThmLocal02}]
It suffices to establish the second assertion of the theorem 
since the first one is its particular case. 
Without loss of generality, we assume that the target functions $\varphi$ and $\psi$ are non-negative.

By the change of measure formula \eqref{Ch7basic equ1}, 
we get that for any $s \in (-s_0, s_0)$ with sufficiently small $s_0 >0$,  
\begin{align}\label{LLT_Def_An_d}
& I :  =  \mathbb{E} \Big[ \varphi(G_n \!\cdot\! x) \psi \big( \sigma(G_n, x) - n\lambda_1 - \sqrt{n}\sigma t \big) \Big]   \\
      & =  r_s(x) \kappa^{n}(s) \mathbb{E}_{ \mathbb{Q}^{x}_{s} }
\Big[ (\varphi r_s^{-1})(G_n x) e^{-s \sigma(G_n, x) } 
   \psi \big( \log|\langle f, G_n v \rangle| - n\lambda_1 - \sqrt{n}\sigma t \big) \Big]. \nonumber
\end{align}
As in the equation \eqref{SaddleEqua}, for any $|t| = o(\sqrt{n})$ (not necessarily $t > 1$), 
we choose $s \in (-s_0, s_0)$ satisfying the equation
\begin{align}\label{Saddle_Equa_cc}
\Lambda'(s) - \Lambda'(0) = \frac{\sigma t}{\sqrt{n}}. 
\end{align}
Note that $s \in (-s_0, 0]$ if $t \in (-o(\sqrt{n}), 0]$,
and $s \in [0, s_0)$ if $t \in [0, o(\sqrt{n}))$. 
In the same way as in the proof of \eqref{LambdaQ01}, 
from \eqref{Saddle_Equa_cc} it follows that for any $|t| = o(\sqrt{n})$, 
\begin{align*}
s \Lambda'(s) - \Lambda(s) 
= \frac{t^2}{2n} - \frac{t^3}{n^{3/2} } \zeta \Big( \frac{t}{\sqrt{n}} \Big), 
\end{align*}
where $\zeta$ is the Cram\'{e}r series defined by \eqref{Ch7Def-CramSeri}. 
For brevity, denote 
\begin{align*}
T_n^x = \sigma(G_n, x) - n \Lambda'(s),  \qquad  Y_n^{x,y} = \log \delta(y, G_n \!\cdot\! x). 
\end{align*}
Hence, using \eqref{Pf_LLN_Equality}, we have
\begin{align*}
I & = r_s(x) e^{-n [s\Lambda'(s) - \Lambda(s)]} \mathbb{E}_{ \mathbb{Q}^{x}_{s} }
\Big[ (\varphi r_s^{-1})(G_n \!\cdot\! x) e^{-s T_n^x } \psi \Big( T_n^x + Y_n^{x,y} \Big) \Big] \nonumber\\
& = r_s(x) e^{- \frac{t^2}{2} + \frac{t^3}{ \sqrt{n} } \zeta( \frac{t}{\sqrt{n} } )} \mathbb{E}_{ \mathbb{Q}^{x}_{s} }
\Big[ (\varphi r_s^{-1})(G_n \!\cdot\! x) e^{-s T_n^x } \psi \Big( T_n^x + Y_n^{x,y} \Big) \Big]. 
\end{align*}
Notice that $r_s(x) \to 1$ as $n \to \infty$, uniformly in $x \in \bb P(V)$. 
Thus in order to establish Theorem \ref{ThmLocal02}, it suffices to prove 
the following asymptotic: as $n \to \infty$, 
\begin{align}\label{LLT_Pf_Object}
 J : = \sigma \sqrt{2 \pi n} \,  \bb E_{\bb Q_s^x}
\Big[ (\varphi r_s^{-1})(G_n \!\cdot\! x) e^{-s T_n^x } \psi \Big( T_n^x + Y_n^{x,y} \Big) \Big]  
 \to \nu(\varphi) \int_{\bb R} \psi(u) du. 
\end{align}
We shall apply Lemmas \ref{Lem_Regu_pi_s} and \ref{Lem_R_st_limit} to establish \eqref{LLT_Pf_Object}.  
Recall that the functions $\chi_k^y$ and $\overline \chi_k^y$ are defined by  \eqref{Def-chi-nk-222}.
Then, using the partition of the unity \eqref{Unit-partition001-222} as in the proof of Theorem \ref{Thm-Cram-Scalar-tag}, 
we have
\begin{align} \label{PosiScalAn 01}
J  =:  J_1 + J_2, 
\end{align}
where, with $M_n = \floor{A \log n}$ and a sufficiently large constant $A >0$, 
\begin{align*}
&  J_1 = \sigma  \sqrt{2 \pi n} \,  \bb E_{\bb Q_s^x}
  \Big[ (\varphi r_s^{-1} \overline \chi_{M_n}^y)(G_n \!\cdot\! x) e^{-s T_n^x } \psi \Big( T_n^x + Y_n^{x,y} \Big) \Big]  \nonumber\\
& J_2 =  \sigma  \sqrt{2 \pi n}  \,   
 \sum_{k =0}^{ M_n - 1 } \bb E_{\bb Q_s^x}
  \Big[ (\varphi r_s^{-1} \chi_{k}^y)(G_n \!\cdot\! x) e^{-s T_n^x } \psi \Big( T_n^x + Y_n^{x,y} \Big)  \Big].
\end{align*}

\textit{Upper bound of $J_1$.} 
In order to prove that $J_1 \to 0$ as $n \to \infty$, we are led to consider two cases:
$s \geq 0$ and $s<0$.

When $s \geq 0$, since the function $\psi$ has a compact support, say $[b_1, b_2]$,
we have $T_n^x + Y_n^{x,y} \in [b_1, b_2]$. 
Noting that $Y_n^{x,y} \leq 0$, we get $T_n^x \geq b_1$,
and hence $e^{-s T_n^x} \leq e^{-s b_1} \leq c$.
Since the function $\psi$ is directly Riemann integrable on $\bb R$,
it is bounded and hence
\begin{align*}
J_1 \leq  c \sqrt{n}  \,  \bb Q_s^x \Big( Y_n^{x,y} \leq - \floor{A \log n} \Big). 
\end{align*}
Applying Lemma \ref{Lem_Regu_pi_s} with $k = \floor{A \log n}$ 
and choosing $A$ sufficiently large, we obtain that,
as $n \to \infty$, uniformly in $s \in [0, s_0)$, 
\begin{align}\label{Pf_LLTUpp_J1}
J_1 \leq  C \sqrt{n}  \,  e^{-c \floor{A \log n} } \to 0. 
\end{align}

When $s < 0$, 
from $T_n^x + Y_n^{x,y} \in [b_1, b_2]$, we get that
$e^{-s T_n^x} \leq e^{-s b_2 + sY_n^{x,y}}$.
Hence, by the H\"{o}lder inequality, Lemmas \ref{Lem_Regu_pi_s} and \ref{Lem_Inte_Regu_a},
we obtain that, as $n \to \infty$, uniformly in $s \in (-s_0, 0]$,
\begin{align*}
J_1 & \leq  C \sqrt{n}  \, \left\{ \bb E_{\bb Q_s^x} \left( \frac{1}{ \delta(y, G_n \!\cdot\! x)^{-2s} } \right)
 \bb Q_s^x \Big( Y_n^{x,y} \leq - \floor{A \log n} \Big)  \right\}^{1/2}  \nonumber\\
& \leq  C \sqrt{n}  \,  e^{-c \floor{A \log n} }  \to 0,  
\end{align*}
where again $A$ is large enough.

\textit{Upper bound of $J_2$.} 
Note that the support of the function $\chi_{k}^y$ is contained in the set $\{x \in \bb P(V): -Y_n^{x,y} \in [a (k-1), a(k+1)]\}$.
Therefore on  $\supp \chi_{k}^y$ 
we have  
$- a \leq Y_n^{x,y} + a k \leq a$.
For any $w \in \bb R$ set $\Psi_{s} (w) = e^{-sw} \psi(w)$ and, according to \eqref{smoo001}, define   
${\Psi}^+_{s, \ee}(w) = \sup_{w' \in \mathbb{B}_{\ee}(w)} \Psi_{s} (w')$, for $\varepsilon\in(0,\frac{1}{2})$.
Let also
\begin{align}\label{Def_varphi_skee}
\varphi_{s,k}^y (x) 
   = (\varphi r_s^{-1} \chi_{k}^y)(x),
   \quad  x \in \bb P(V). 
\end{align}  
With this notation, choosing $a \in (0,\ee)$, it follows that 
\begin{align*}
J_2 
 & \leq  \sigma  \sqrt{2 \pi n}     
  \sum_{k =0}^{ M_n -1 }  
\bb E_{\bb Q_s^x}
  \Big[ \varphi_{s,k}^y (G_n \!\cdot\! x) 
  e^{-s Y_n^{x,y} }
    \Psi_{s,\ee}^+ ( T_n^x - a k)  \Big] \\
 & \leq  \sigma  \sqrt{2 \pi n}     
  \sum_{k =0}^{ M_n -1 }  e^{-s a (k-1) }
\bb E_{\bb Q_s^x}
  \Big[ \varphi_{s,k}^y (G_n \!\cdot\! x) 
    \Psi_{s,\ee}^+ ( T_n^x - a k)  \Big].
\end{align*} 
Since the function $\Psi^+_{s, \ee}$ is non-negative and integrable on the real line, 
using Lemma \ref{estimate u convo}, we get
\begin{align*} 
J_2  & \leq  ( 1+ C_{\rho}(\ee) )
\sigma  \sqrt{2\pi n}  \sum_{k =0}^{ M_n -1 }  e^{-s a (k-1)}   \\
&  \quad \times  \mathbb{E}_{\mathbb{Q}_{s}^{x}}
\left[ \varphi_{s,k}^y (G_n \!\cdot\! x)  
({\Psi}^+_{s, \ee} * \rho_{\ee^2}) (T_n^x - a k) \right], 
\end{align*} 
where $C_{\rho}(\ee) >0$ is a constant  converging to $0$ as $\ee \to 0$. 
Since the function $\widehat\rho_{\ee^{2}}$ is integrable on $\bb{R}$, 
by the Fourier inversion formula, we have 
\begin{align*}
(\Psi^+_{s, \eta, \ee} * \rho_{\ee^2}) (T_n^x - a k)
= \frac{1}{2\pi} \int_{\bb{R}}e^{iu (T_n^x - a k)} 
 \widehat {\Psi}^+_{s, \ee}(u) \widehat \rho_{\ee^2}(u) du. 
\end{align*}
By the definition of the perturbed operator $R_{s,iu}$ (cf. \eqref{Def_Rsz_Ch7}),  
and Fubini's theorem, we obtain
\begin{align} \label{Scal Bnxl 01}
J_2 & \leq  (1+ C_{\rho}(\ee)) \sigma   \sqrt{\frac{n}{2\pi}} 
  \sum_{k =0}^{\infty}  \mathds{1}_{ \{k \leq M_n - 1 \} }  e^{-s a (k-1)}  \nonumber\\
&  \qquad \times   \int_{\mathbb{R}} e^{-iu a k}  R_{s,iu}^n (\varphi_{s,k}^y) (x)
\widehat {\Psi}^+_{s,  \ee}(u) \widehat \rho_{\ee^2}(u) du. 
\end{align}
To deal with the integral in \eqref{Scal Bnxl 01}, we shall use Lemma \ref{Lem_R_st_limit}. 
Note that $e^{ \frac{C k^2}{n} } \to 1$ as $n \to \infty$, uniformly in $0 \leq k \leq M_n -1$. 
Since the function $\widehat {\Psi}^+_{s, \ee} \widehat \rho_{\ee^2}$ is compactly supported on $\bb R$,  
applying Lemma \ref{Lem_R_st_limit} with $\varphi = \varphi_{s,k}^y$, 
$\psi = \widehat {\Psi}^+_{s,  \ee} \widehat \rho_{\ee^2}$
and  $l = \frac{a k}{n \sigma}$, 
we obtain that there exists a constant $C >0$ such that for all $s \in (-s_0, s_0)$, 
$x \in \bb P(V)$, $y \in \bb P(V^*)$, $0 \leq k \leq M_n -1$, $\varphi \in \mathscr{B}_{\gamma}$ and $n \geq 1$, 
\begin{align*}  
& \left| \sigma  \sqrt{\frac{n}{2\pi}}  
  \int_{\mathbb{R}} e^{-it a k} R_{s,iu}^n (\varphi_{s,k}^y) (x)
\widehat {\Psi}^+_{s,  \ee}(u) \widehat \rho_{\ee^2}(u) du   
- \widehat {\Psi}^+_{s,  \ee}(0) \widehat \rho_{\ee^2}(0) \pi_s (\varphi_{s,k}^y) \right|  \\
& \leq \frac{C}{\sqrt{n}} \| \varphi_{s,k}^y \|_{\gamma}. 
\end{align*}
Using \eqref{Bound_varphi_Holder_norm}
and choosing a sufficiently small $\gamma > 0$, one can verify that 
the series $\frac{C}{\sqrt{n}} \sum_{k = 0}^{ M_n -1 }  
  \| \varphi_{s,k}^y \|_{\gamma}$
converges to $0$ as $n \to \infty$. 
Consequently, we are allowed to interchange the limit as $n \to \infty$
and the sum over $k$ in \eqref{Scal Bnxl 01}. 
Then, noting that $\widehat\rho_{\ee^{2}}(0) =1$ and 
$\widehat {\Psi}^+_{0, \ee}(0)  = \int_{\mathbb{R}} \sup_{w' \in \mathbb{B}_{\ee}(w)} \Psi_{0}(w') dw$, 
we obtain that uniformly in $v \in V$ and $f \in V^*$ with $\|v\|=1$ and $\| f \| =1$,
\begin{align} \label{LimsuBn a}
\limsup_{n \to \infty}  J_{2}  
& \leq    (1+ C_{\rho}(\ee)) \int_{\mathbb{R}} \sup_{w' \in \mathbb{B}_{\ee}(w)} \Psi_{0}(w') dw
 \sum_{k =1}^{\infty} \nu \left( \varphi_{0,k}^y \right)  \notag\\
 & =    (1+ C_{\rho}(\ee)) \nu(\varphi) \int_{\mathbb{R}} \sup_{w' \in \mathbb{B}_{\ee}(w)} \Psi_{0}(w') dw,
\end{align}
where in the last equality we used \eqref{Unit-partition001}. 
Letting $\ee \to 0$, $n \to \infty$, 
and noting that $C_{\rho}(\ee) \to 0$, 
we obtain the desired upper bound for $J_2$: 
uniformly in $v \in V$ and $f \in V^*$ with $\|v\|=\| f \| =1$,
\begin{align} \label{ScaProLimBn Upper 01}
\limsup_{n \to \infty} J_{2}  \leq  \nu(\varphi) \int_{\bb R} \psi(u) du. 
\end{align}

\textit{Lower bound of $J_{2}$.} 
Since on the set $\{-Y_n^{x,y} \in [a (k-1), a(k+1)]\}$, we have
$- a \leq Y_n^{x,y} + a k \leq a$.
Set $\Psi_{s} (w) = e^{-sw} \psi(w)$, $w \in \bb R$
and ${\Psi}^+_{s, \ee}(w) = \inf_{w' \in \mathbb{B}_{\ee}(w)} \Psi_{s} (w')$, for $\varepsilon\in(0,\frac{1}{2})$.  
Then, with $a \in (0, \ee)$,
\begin{align*}
J_2 & \geq  \sigma  \sqrt{2 \pi n}     
  \sum_{k =0}^{ M_n-1 }
\bb E_{\bb Q_s^x}
  \Big[ \varphi_{s,k}^{y}(G_n \!\cdot\! x) e^{-s Y_n^{x,y} } \Psi_{s,\ee}^- ( T_n^x - a k) \Big]  \nonumber\\
 & \geq  \sigma  \sqrt{2 \pi n}     
  \sum_{k =0}^{ M_n-1 }  e^{-s a k}
\bb E_{\bb Q_s^x}
  \Big[ \varphi_{s,k}^{y}(G_n \!\cdot\! x) 
    \Psi_{s,\ee}^- ( T_n^x - a k)  \Big].
\end{align*} 
By 
Fatou's lemma, it follows that
\begin{align*} 
\liminf_{n \to \infty} J_2  
\geq  \sum_{k =0}^{\infty} \liminf_{n \to \infty}  \sigma   \sqrt{2 \pi n}      
   e^{-s a k}   \mathbb{E}_{\mathbb{Q}_{s}^{x}}
\left[  \varphi_{s,k}^y (G_n \!\cdot\! x)  \Psi^-_{s, \ee} (T_n^x - a k)\right]. 
\end{align*}
Since $s=o(1)$ as $n\to\infty$, 
we see that for fixed $k \geq 1$, we have $e^{-s a k} \to 1$ as $n \to \infty$. 
Since the function $\Psi^-_{s, \ee}$ is non-negative and integrable on the real line, 
by Lemma \ref{estimate u convo}, we get
\begin{align} \label{ScaProLimAn Bn 0134}
& \liminf_{n \to \infty} J_2  
 \geq  \sum_{k =0}^{\infty}  \liminf_{n \to \infty}  \sigma \sqrt{2 \pi n}
\bb E_{\bb Q_s^x}
  \Big[ \varphi_{s,k}^y (G_n \!\cdot\! x)
    (\Psi^-_{s,  \ee} * \rho_{\ee^2}) (T_n^x - a k)  \Big]  \nonumber\\
&  \  -  \sum_{k =0}^{\infty} \limsup_{n \to \infty}  \sigma  \sqrt{2 \pi n}
  \int_{|w| \geq \ee}  \mathbb{E}_{\mathbb{Q}_{s}^{x}}  
  \left[  \varphi_{s,k}^y (G_n \!\cdot\! x) \Psi^-_{s, \ee}  (T_n^x - a k - w)
   \right] \rho_{\ee^2}(w)  dw  \nonumber\\
& =: J_3 - J_4. 
\end{align}


\textit{Lower bound of $J_3$.} 
Proceeding as in the proof of the upper bound \eqref{LimsuBn a} and using Lemma \ref{Lem_R_st_limit},
we obtain 
\begin{align} \label{Low_Boun_J3}
  J_3  \geq  \int_{\mathbb{R}} \inf_{w' \in \mathbb{B}_{\ee}(w)} \Psi^-_{0, \ee}(w') dw
 \sum_{k =0}^{\infty} \nu \left(  \varphi_{0, k}^y \right).  
\end{align}
Taking the limit as $\ee\to 0$ and $n\to\infty$,
we get the lower bound for $J_3$: 
uniformly in $v \in V$ and $f \in V^*$ with $\|v\|=1$ and $\| f \| =1$, 
\begin{align} \label{ScaProLimBnLow_J3}
\liminf_{n \to \infty}  J_3 \geq  \nu(\varphi) \int_{\bb R} \psi(u) du. 
\end{align}

\textit{Upper bound of $J_4$.} 
%
%
By Lemma \ref{estimate u convo}, we have
$\Psi^-_{s, \ee} \leq (1+ C_{\rho}(\ee)) \Psi^+_{s, \ee} * \rho_{\ee^2}$. 
Applying Lemma \ref{Lem_R_st_limit} with $\varphi = \varphi_{s,k}^y$ 
and $\psi = \widehat{\Psi}^+_{s, \ee} \widehat{\rho}_{\ee^2}$,
it follows from the Lebesgue dominated convergence theorem that  
\begin{align}\label{Pf_Lower_J41}
J_{4}  \leq (1+ C_{\rho}(\ee)) \sum_{k=0}^{ \infty }
\nu \left( \varphi_{0,k}^y \right)  \widehat{\Psi}^+_{0, \ee}(0) \widehat{\rho}_{\ee^2}(0)
\int_{|w| \geq \ee} \rho_{\ee^2}(w) dw, 
\end{align}
which converges to $0$ as $\ee \to 0$.

%

Combining \eqref{ScaProLimAn Bn 0134}, \eqref{ScaProLimBnLow_J3} and \eqref{Pf_Lower_J41}, 
we get the desired lower bound for $J_2$: 
uniformly in $v \in V$ and $f \in V^*$ with $\|v\|=1$ and $\| f \| =1$, 
\begin{align} \label{lowelast001}
\liminf_{n\to\infty} J_2 \geq  \nu(\varphi) \int_{\bb R} \psi(u) du.
\end{align}
Putting together
\eqref{PosiScalAn 01}, \eqref{Pf_LLTUpp_J1}, \eqref{ScaProLimBn Upper 01} and \eqref{lowelast001},
we obtain the asymptotic \eqref{LLT-Moderate-01}. 
This ends the proof of Theorem \ref{ThmLocal02}. 
\end{proof}


\end{document}